\documentclass[a4paper,one side, oldfontcommands,11pt]{memoir}
\usepackage[table]{xcolor}
\usepackage{pdfpages}
\usepackage[utf8x]{inputenc}

\usepackage{a4wide}
\usepackage{bbm}
\usepackage{amssymb,amsmath,amsthm,booktabs}
\usepackage[square,numbers]{natbib}
\usepackage{mathrsfs}
\usepackage{enumerate}
\usepackage{mdwlist}
\usepackage{psfrag}
\usepackage{caption}
\usepackage{makecell}

\usepackage{subfig}
\usepackage{setspace}
\usepackage[chapter]{algorithm}
\usepackage{algorithmic,bm}
\usepackage{multirow}
\usepackage{url,longtable}
\usepackage{xcolor}
\usepackage{booktabs}
\usepackage{hyperref}
\usepackage{datetime}
\usepackage{lscape}
\usepackage{natbib}
\usepackage{bibunits}
\usepackage{enumitem}
\usepackage{lipsum}
\usepackage{calc,soul}
\usepackage[export]{adjustbox}

\usepackage{amsfonts}

\usepackage{diagbox}
\usepackage{placeins}
\usepackage{multirow}
\usepackage{enumerate}
\usepackage{enumitem}
\usepackage{float}
\numberwithin{figure}{chapter}
\numberwithin{table}{chapter}
\usepackage{tikz}
\usepackage{afterpage}
\usepackage[T1]{fontenc}
\newcommand{\sys}{\mathrm{sys}}
\usetikzlibrary{arrows.meta}

\usetikzlibrary{decorations.markings}
\usepackage{booktabs}
\newcommand\blankpage{%
    \null
    \thispagestyle{empty}%
    \addtocounter{page}{-1}%
    \newpage}
\usepackage{epstopdf}
\usepackage{mathtools}
\usepackage{upgreek}

\def\leftalgn{0.05}\def\rightalgn{0.7}
\def\algnrow{\rule{\leftalgn\textwidth}{0ex}&\rule{\rightalgn\textwidth}{0ex}}
 {\\\algnrow%
  \end{array}\vspace{-\baselineskip}%
  \end{equation}\ignorespacesafterend%
}

\usepackage{adjustbox}
\DeclareMathOperator{\sech}{sech}
\DeclareMathOperator{\csch}{csch}
\newcommand{\R}{\mathbb{R}}
\newcommand{\C}{\mathbb{C}}

\newcommand{\N}{\mathbb{N}}

\newcommand{\p}{\mathrm{PSL}(2,\mathbb{R})}

\newcommand{\h}{\mathbb{H}}

\newtheorem{theorem}{Theorem}
\newtheorem*{repthm}{Theorem}
\newtheorem*{repprop}{Proposition}
\newtheorem{lemma}{Lemma}

\newtheorem{question}{Question}
\newtheorem{example}{Example}
\newtheorem{non-example}{Non-example}

\def\bc{\begin{center}}
\def\ec{\end{center}}

\definecolor{mycolor}{rgb}{0,0,0}
\setsecheadstyle{\color{mycolor}\Large\bfseries\sffamily\raggedright}
\setsubsecheadstyle{\color{mycolor}\large\bfseries\sffamily\raggedright}
\setsubsubsecheadstyle{\color{mycolor}\normalsize\bfseries\sffamily\raggedright}

\hypersetup
{
    pdfauthor={Achintya Dey},
    pdfsubject={SCSs},
    pdftitle={SPP},
    pdfkeywords={SPP, CFS},
    colorlinks=true,       % false: boxed links; true: colored links
    linkcolor=black,          % color of internal links
    citecolor=blue,        % color of links to bibliography
    filecolor=cyan,      % color of file links
    urlcolor=blue,
}

\makeatletter
\newlength\dlf@normtxtw
\newsavebox{\feline@chapter}
\newcommand\feline@chapter@marker[1][4cm]{%
\sbox\feline@chapter{%
\resizebox{!}{#1}{\fboxsep=1pt%
\colorbox{mycolor}{\color{white}\bfseries\sffamily\thechapter}%
}}%
\rotatebox{90}{%
\resizebox{%
\heightof{\usebox{\feline@chapter}}+\depthof{\usebox{\feline@chapter}}}%
{!}{\scshape\so\@chapapp}}\quad%
\raisebox{\depthof{\usebox{\feline@chapter}}}{\usebox{\feline@chapter}}%
}
\newcommand\feline@chm[1][4cm]{%
\sbox\feline@chapter{\feline@chapter@marker[#1]}%
\makebox[0pt][l]{% aka \rlap
\makebox[1cm][r]{\usebox\feline@chapter}%
}}
\makechapterstyle{daleif1}{

\renewcommand\printchapternum{\null\hfill\feline@chm[2.5cm]\par}

}
\makeatother
\chapterstyle{daleif1}
\captionnamefont{\small\sffamily\bfseries}
\captiontitlefont{\small\sffamily}
\makepagestyle{myruled}
\makeheadrule {myruled}{\textwidth}{1pt}%
\makefootrule {myruled}{\textwidth}{1pt}{\footruleskip}
\makeoddhead {myruled}{\small\sffamily\rightmark}{}{\small\thepage}
\makeatletter % because of \@chapapp
\makepsmarks{myruled}{
  \nouppercaseheads\createmark{chapter}{both}{shownumber}{
  \@chapapp\ }{.\ }
 \createmark{section}{right}{shownumber}{}{. \ }
  \createmark{subsection}{right}{shownumber}{}{. \ }
  \createmark{subsubsection}{right}{shownumber}{}{. \ }
  \createplainmark {toc}{both} {\contentsname}
  \createplainmark {lof}{both} {\listfigurename}
  \createplainmark {lot}{both} {\listtablename}
  \createplainmark {index}{both} {\indexname}
  \createplainmark {glossary}{both} {\glossaryname}
}
\makeatother

\usepackage{stackengine}
\usepackage{scalerel}
\usepackage{xcolor}

\numberwithin{equation}{section}
\newtheorem{definition}[]{Definition}
\newtheorem{remark}[theorem]{Remark}
\newtheorem{corollary}[theorem]{Corollary}
\newtheorem{proposition}[theorem]{Proposition}
\numberwithin{theorem}{chapter}
\numberwithin{lemma}{chapter}
\numberwithin{algorithm}{chapter}
\numberwithin{definition}{chapter}
\allowdisplaybreaks
\begin{document}
\pagenumbering{roman}
\setcounter{page}{1}
%\pagenumbering{gobble}
%\include{Cover_page}
%\cleardoublepage
%\include{title1}
%\blankpage
\thispagestyle{empty}

%\vglue 48mm
\begin{table}[h]
\centering
\begin{tabular}{c}
    \toprule
    \midrule\\[-2mm]
    {\LARGE \textbf{Systolic Geometry and Combinatorics of Surfaces}}\\[3mm]
    \midrule
    \bottomrule
\end{tabular}
\end{table}

\vfill

\begin{center}
{{\em A thesis submitted in partial fulfillment of the requirements }} \\[1.5ex]
{{\em for the degree of}} \\[4ex]
{\large {\textbf{ DOCTOR OF PHILOSOPHY}}}\\
\vglue 3mm
%\vfill
{\textsl by}\\[4ex]
{\large {\textbf {ACHINTYA DEY}}\\ \vskip .2cm (Roll No. 19208263)}\\[12ex]
%\vglue 20mm
\vfill
\includegraphics[width=1.9in]{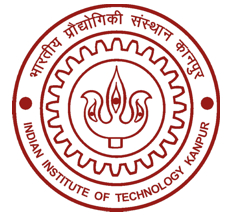} \\
\vfill
{\textsl to the} \\
{\normalsize \textbf  {DEPARTMENT OF MATHEMATICS}}  \\
{\Large  INDIAN INSTITUTE OF TECHNOLOGY KANPUR}\\
\vspace{1 cm}
June 2026
\end{center}

\blankpage
\pagestyle{plain}
{\raggedleft
\includegraphics[width=2.5in]{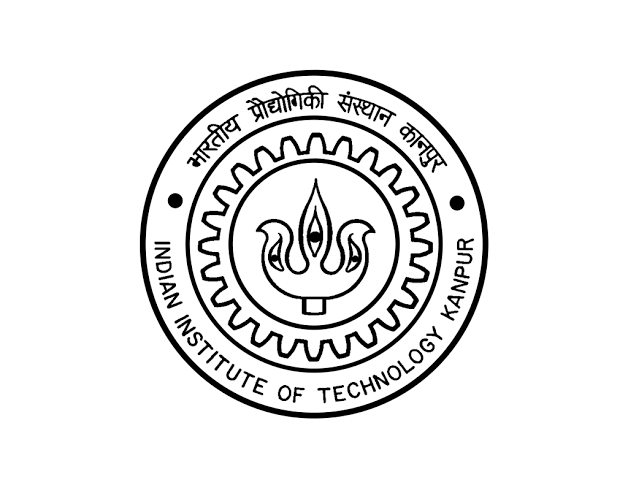}}
\begin{center}
{\textbf {\LARGE Certificate}}
\end{center}
\vspace{0.5 in} 
It is certified that the work contained in the thesis titled \textbf{\emph{``Systolic  geometry and combinatorics of surfaces”}} by \textbf{\emph{``Achintya Dey”}} has been carried out under my supervision and that this work has not been submitted elsewhere for a degree.
\par
\vspace{0.7 in}

\vspace{-0.2 in}

\begin{tabular}[c]{ll}
%{} & {\bf (Signature of Supervisor)}\\
{} & {\textbf {Prof. Bidyut Sanki}}\\
{} & {(Thesis Supervisor)} \\
{} & \textbf {Department of Mathematics,} \\
 \textbf{June, 2026} $\hspace{2.0cm}$ &  {\textbf {Indian Institute of Technology, Kanpur.}} 
\end{tabular}

\blankpage

%\blankpage
%\tableofcontents
%\cleardoublepage
%\chapter*{Aknowledgment}
%\addcontentsline{toc}{chapter}{\numberline{}Acknowledgment}
%\include{Acknowledgment}
%\chapter*{Declaration from Author}
%\addcontentsline{toc}{chapter}{\numberline{}Papers from the Thesis}
%\include{Author_declaration}
\pagestyle{plain}
{\raggedleft
\includegraphics[width=2.5in]{IITK_Logo.png}}
\begin{center}
{\textbf {\LARGE Declaration}}
\end{center}
\vspace{0.3 in} 
This is to certify that the thesis titled \textbf{\emph{``Systolic  geometry and combinatorics of surfaces”} } has been authored by me. It presents the research conducted by me under the supervision of \textbf {\itshape Prof. Bidyut Sanki}.

\noindent To the best of my knowledge, it is an original work, both in terms of research content and narrative and has not been submitted elsewhere, in part or in full, for a degree. Further, due credit has been attributed to the relevant state-of-the-art collaborations (if any) with appropriate citations and acknowledgments in line with established norms and practices.

\par
\vspace{0.2 in}
\begin{flushright}
\begin{quote}
\begin{figure}[hb]
 % {\hspace*{2 in}\includegraphics[width=5cm]{Achintya's signature.png_}}
\end{figure}
\vspace{-0.1 in}
\begin{tabular}[c]{ll}
%{} & {}\\
{} & \textbf{Achintya Dey} \\
{} & \textbf{Department of Mathematics,}\\
\textbf{June, 2026}$\hspace{2.0cm}$ & \textbf{Indian Institute of Technology, Kanpur.} 
\end{tabular}
\end{quote}
\end{flushright}
\blankpage
\cleardoublepage
\newpage
%\addcontentsline{toc}{chapter}{Synopsis}
% \begin{bibunit}[abbrv]
% \include{Synopsis}
% \renewcommand{\bibsection}{\section*{ \huge Synopsis Bibliography\\}}
% \putbib[synops]
% \end{bibunit}

%\addcontentsline{toc}{chapter}{Synopsis}
% =============================================================== %
%                         Title and authors                       %
% =============================================================== %
%\title{\textbf{SYNOPSIS}}
\thispagestyle{empty}

\thispagestyle{empty}
\begin{center}
\textbf{\huge{SYNOPSIS}}
\end{center}
\vspace{0.5cm}

%\chapter*{Synopsis}
%\addcontentsline{toc}{chapter}{Synopsis}

%\begin{center} 
%{\sc Synopsis}
%\end{center}
%\begin{table}[!ht]
%\scriptsize
%\centering
%\begin{tabular}{rcl}
%\toprule
%Name of The Student & : & \textbf{Debashis Samanta}  %\\
%Degree for which Submitted & : & \textbf{Ph.D} \\
%Department & : & \textbf{Statistics}  \\
%Title & : & \textbf{Analysis of Life Testing Data}  % \\
%& & \textbf{Using Step-stress Model}   \\
%Supervisor & : & \textbf{Dr. Arindam Gupta}  \\
%Month and Year of Submission & : & %\textbf{December, 2018} \\
%\bottomrule
%\end{tabular}
%\end{table}

\vspace{-5mm}
\begin{table}[h]
\begin{center}
\begin{tabular}[c]{ll}
\toprule\midrule\\
Name of the Student: & \textbf{Achintya Dey} \\[3mm]
Roll Number: & \textbf{19208263}\\[3mm]
Degree for which Submitted: & \textbf{Doctor of Philosophy}\\[3mm]
Department: & \textbf{Mathematics}\\[3mm]
Thesis Title: &
\textbf{Systolic Geometry and}\\
&
\textbf{Combinatorics of Surfaces}\\[3mm]
Thesis Supervisor: & \textbf{Dr. Bidyut Sanki}\\[3mm]
Month and Year of Submission:$\hspace{8.5mm}$ & \textbf{June 2026}\\[0.4cm]
\midrule
\bottomrule
\end{tabular}
\end{center}
\end{table}

 This thesis investigates two themes. The first studies \emph{systolic embeddings of graph on translation surfaces}, proving existence results for \emph{essential} and \emph{cellular systolic embeddings} and determining extremal genera for certain graph families. The second examines \emph{hyperbolic surfaces in the Thurston set}, constructing a one-parameter deformation that increases the \emph{systole length} and produces \emph{Bolza-like surfaces}, i.e., surfaces whose systolic geodesics decompose the surfaces into triangles, for infinitely many genera. A symmetry property of geodesic intersections on these surfaces is also established, together with an application of these Bolza-like surfaces.
 
\section*{Systolic Embedding of Graphs On Translation Surfaces}
A \emph{translation surface} is obtained by identifying pairs of parallel sides of a finite collection of Euclidean polygons via translations. The vertices of the polygons give rise to \emph{cone points}, with total angle $2k\pi$ for some $k \in \mathbb{N}$; points with $k>1$ are \emph{singular points}, while points with $k=1$ are \emph{marked points}. A \emph{saddle connection} is a geodesic segment connecting two singular points without singularities in its interior, and a \emph{shortest saddle connection} is one of minimal length. For surfaces without singular points, marked points are used in place of singular points for these definitions.

Let \(S\) be a translation surface. The \emph{systolic graph} \(\Gamma_S\) of \(S\) is the graph whose vertices correspond to the singular (or marked) points of \(S\) and whose edges correspond to the shortest saddle connections. A finite graph \(G\) is said to admit a \emph{systolic embedding} on \(S\) if \(G\) is isomorphic to \(\Gamma_S\).
 An embedding of a graph on a translation surface is said to be \emph{celluar}(respectively, \emph{essential}) if each complementary region is a topological disk (respectively, no complementary region is a tolpological disk). 

Embedding of graphs on closed hyperbolic surfaces has been studied by Sanki~\cite{sa19} and Gadgil--Sanki~\cite{san19}. Sanki~\cite{sa19} showed that any metric graph $G$ can be essentially and isometrically embedded on a closed hyperbolic surface, possibly up to some scaling and determined the minimal genus required in terms of the Betti number and Betti deficiency of $G$. Later Gadgil-Sanki~\cite{san19} studied which fat graphs can arise as systolic graphs. In this thesis, we extend these ideas to \emph{systolic embeddings of graphs on translation surfaces}.
We shall study the following problems regarding to systolic embedding of graphs on translation surfaces. 

\begin{enumerate}
      \item Given a finite graph $G$, does there exist a translation surface $S$ such that $G$ admits a systolic embedding on $S$?
      
    \item If such an embedding exists, then what is the minimum (and maximum) possible genus of a surface $S$ on which $G$ admits systolic embedding?
\end{enumerate} 

\subsection*{Essential-systolic embedding.}
We study the essential-systolic embedding of a finite graph $G$ on a translation surface. To answer Question~(1), we show that every finite graph admits such an embedding (see Theorem 1.1, \cite{MR4985436}). To establish this, we prove several intermediate results that serve as tools for this result as follows.

(i) Every complete graph $K_n$, with n odd and $n\geq5$, admits essential-systolic embedding embedded on a translation surface $S_{g_n}$ of genus $g_n= [1-n(1-\frac{n-1}{2})]$ (Theorem 1.2, \cite{MR4985436}).

(ii) For any odd $n\ge5$, every sub-graph of $K_n$ admits essential-systolic embedding on a translation surface of genus $g_n$  (Proposition 1.3, \cite{MR4985436}).

\subsection*{Genus of Translation Surfaces Admitting Essential-Systolic Graph Embeddings}
%{j}
For a given graph $G$ we define 
$P_G$ := \{$g$ $\mid$  $G$ is systolically and essentially embedded on a translation surface of genus $g$\}. $P_G \neq \emptyset$ by above result.
If a graph admits an essential-systolic embedding on a translation surface of genus $g$, then there exists an integer $g'>g$ such that the graph admits such an embedding on a translation surface of genus $g'$. Therefore, the maximum-genus problem of such surfaces does not make sense, and hence we study the minimum-genus problem. We define $
g^{\min}_G := \min \left\{ g \mid g \in P_G \right\}.$ 

To address Question (2), we compute the following upper bound for \( g^{\min}_G \) (Theorem 1.4, \cite{MR4985436})
\[
g^{\min}_G \leq 
\begin{cases}
1 - n\left(1 - \frac{n - 1}{2}\right), & \text{if } n \text{ is odd and } n \geq 5, \\[0.5em]
1 - (n + 1)\left(1 - \frac{n}{2}\right), & \text{if } n \text{ is even and } n \geq 5, \\[0.5em]
6, & \text{if } n < 5.
\end{cases}
\].
 
\subsection*{Cellular-systolic embedding}
In contrast to essential embedding, not all graphs admit cellular-systolic embedding on translation surfaces. For example, the complete graph $K_3$ admits cellular embedding only on the sphere, which is not a translation surface. Thus, in response to Question (1), it is natural to ask whether there exists a collection of graphs that can be embedded in this way. In this context, we have identified a family of graphs $\Sigma_n$, which denotes the wedge of $n$ circles for $n \geq 2$. We establish the following results, which jointly answer questions (1) and (2) in this setting.

Let $\Sigma_n$ be the wedge of $n$ circles, where $n \geq 2$. Then we have the following:
\begin{enumerate}
    \item $\Sigma_n$ admits cellular-systolic embedding on a translation surface of genus $\left\lfloor \frac{n}{2} \right\rfloor$. Furthermore, this is the maximum genus for which such an embedding is possible. 
    
    \item The minimum genus of a translation surface on which $\Sigma_n$ admits cellular-systolic embedding is $\left\lceil \frac{n+3}{6} \right\rceil$.  
    
    \item For any integer $g$ satisfying $\left\lceil \frac{n+3}{6} \right\rceil \leq g \leq \left\lfloor \frac{n}{2} \right\rfloor$, the graph $\Sigma_n$ admits  cellular-systolic embedding on a translation surface of the genus $g$ (See Theorem 1.5, \cite{MR4985436}).
\end{enumerate}

Furthermore, we identify a well-structured collection of graphs with more than one vertex and prove the existence of translation surfaces on which this collection admit cellular-systolic embedding (Theorem 4.2, \cite{MR4985436}).

\section*{Bolza like surfaces in Thurston set}

Let $\mathcal{T}_g$ be the Teichm\"uller space of a closed orientable surface $S_g$ of genus $g$~\cite{farb2011primer}. The \emph{systole function} $\sys:\mathcal{T}_g \to \R$ assigns to each hyperbolic surface $X\in \mathcal{T}_g$ the length of its shortest non-trivial geodesic~\cite{Schmutz}, called a \emph{systolic geodesic}. Surfaces whose systolic geodesics fill the surface are central to Thurston's spine problem~\cite{thurstonspine}. Thurston proposed that the set $\chi_g$ of such surfaces, known as the \emph{Thurston set}, forms a spine in the moduli space $\mathcal{M}_g$, but many properties of $\chi_g$, including its connectedness, contractibility, and explicit elements, remain unknown.

Thurston~\cite{thurstonspine} constructed a deformation of hyperbolic surfaces outside $\chi_g$ that increases the systole length. In this thesis, we introduce a deformation on surfaces in $\chi_g$ that is also systole length increasing. The systole function is a topological Morse function~\cite{akrout} and has no minima for $g\ge 2$~\cite{Buser}, with global maxima known only for $g=2$. Jenni~\cite{Jenni} showed that the unique genus two surface attaining the maximum, called the so-called \emph{Bolza surface}, has $12$ systolic geodesics decomposing it into $(4,4,4)$-triangles~\cite{Schmutz}. Motivated by this, we define that a hyperbolic surface is called \emph{Bolza-like} if its systolic geodesics decompose it into $(p,q,r)$-triangles with $\frac{1}{p}+\frac{1}{q}+\frac{1}{r}<1$. 

In particular, we prove that
for every integer \(g = kn + 1\), where \(k, n \in \mathbb{N}\), \(k \geq 2\), and \(n \geq 4\), there exists a continuous function
\[
\mathcal{D}_g : \left[ 0, \frac{\pi}{12} \right] \to \mathcal{T}_g
\]
satisfying the following properties: 

\begin{enumerate}
    \item \(\mathcal{D}_g(t) \in \chi_g\) for all \(t \in \left[ 0, \frac{\pi}{12} \right]\), where \(\chi_g\) is the Thurston set.
    \item \(\mathrm{sys}(\mathcal{D}_g(t_1)) < \mathrm{sys}(\mathcal{D}_g(t_2))\) for all \(0 \le t_1 < t_2 \le \frac{\pi}{12}\).
    \item The surface \(\mathcal{D}_g\left(\frac{\pi}{12}\right)\) has \(6g-6\) systolic geodesics that decompose it into triangles, i.e., \(\mathcal{D}_g\left(\frac{\pi}{12}\right)\) is Bolza-like (Theorem 1.3, \cite{DeySahaSankiToAppear}).
\end{enumerate}

Furthermore, we provide an application of these Bolza-like surfaces. In addition, we establish that the systolic geodesics on these surfaces are symmetrically distributed, which means that every simple closed geodesic on the Bolza-like surface \(\mathcal{D}_g\left(\frac{\pi}{12}\right)\) intersects the set \(\mathrm{Sys}\big(\mathcal{D}_g\left(\frac{\pi}{12}\right)\big)\) of all systolic geodesics in an even number of points (Theorem 1.4, \cite{DeySahaSankiToAppear}). We prove this by constructing a one-parameter family of explicit examples of elements in the Thurston set \(\chi_g\).

\clearpage

\phantomsection
\pagestyle{plain}
\newpage
\blankpage

%\noindent {\Large{\bf\huge Dedication}}

%\vglue 75mm

%\begin{center}
%{\sl Dedicated to} \\[2ex]
%\noindent {\Large  Kundu Sir (Prof. Debasis Kundu)}.
%\end{center}
\thispagestyle{empty}
\vspace*{\fill}
\begingroup
\centering
{\Huge {\textrm{\emph{Dedicated to my parents}}}} 
\vspace{10 cm}
%{\Huge {\sf{{\bf{\emph{Dedicated}}}}}  \\ {\bf{\emph{to}}} \\
%[24pt] {\bf{\emph{My Family}}}}

%{\huge {\bf {\emph{$\mathfrak{D}$edicated} \\[8pt] \emph{$\mathfrak{T}$o} \\ [25pt]
%\emph{$\mathfrak{M}$y $\mathfrak{F}$amily}}}}
\endgroup
\vspace*{\fill}
\newpage
\cleardoublepage
\cleardoublepage\phantomsection\addcontentsline{toc}{chapter}{Acknowledgements}
\newpage
\cleardoublepage
\blankpage
\begin{center}
  {\Large \textbf{\huge \underline{Acknowledgements}}}
\end{center}

\label{Acknowledgements}
\vglue 10mm

I am truly grateful to my advisor, Prof. Bidyut Sanki, for his unwavering support throughout my Ph.D.\ journey. His patience, encouragement, enthusiasm, and extensive knowledge have greatly supported both my research and the writing of this thesis. I am truly grateful for his guidance, which has continuously inspired me throughout all stages of my work. I could not have wished for a better advisor and mentor for my Ph.D.\ studies.

I am deeply grateful to Professors Abhijit Pal, Debasis Sen, Ajay Singh Thakur, and Vikramjeet Singh Chandel for graciously serving as members of my Doctoral Monitoring Committee and for their valuable time, guidance, and constructive feedback throughout my research journey.

The coursework during my Ph.D.\ journey offered a remarkable opportunity to learn from outstanding educators whose guidance greatly influenced my academic growth. Their teaching not only enriched my knowledge but also played a crucial role in shaping and refining my mathematical thinking and skills. I am deeply grateful to Professors Bidyut Sanki, Abhijit Pal, Debasis Sen, Arijit Ganguly, Malay Banerjee, Ashis Mandal, Aparna Dar, G.~Santhanam, and T.~Muthukumar for their invaluable contributions as course instructors during this formative stage of my research career.

I am sincerely thankful to the Head of the Department and all those who have contributed to fostering and sustaining a conducive research environment. I am equally indebted to the National Board of Higher Mathematics for their generous support through funding and fellowship, which made it possible for me to undertake and complete my research at IIT Kanpur.

I am grateful to IIT Kanpur for providing financial support to attend the CIMPA School at BHU, Varanasi, India; the ``Conference on Quantum Topology and Hyperbolic Geometry'' at VIASM, Vietnam; and the ``Geometric Topology Days'' conference in 2025 at Grenoble, France.

I am deeply grateful to Prof. Shiv Parsad from IIT Goa for his valuable insights and for serving as an external expert on the NBHM JRF to SRF upgrade committee.

I am deeply grateful to my early teachers Goutam Babu, Asim Babu, Naresh Babu, Sukumar Babu, Deepak Babu, Shanti Babu, Bimal Babu, Sushobhan Babu, Bodidhisatwa Babu, and Jayashri Ma’am of Babuijore Dharanidhar High School, as well as to Professors Shubhendu Ghosh, Ramprasad Saha, Sayan Mukherjee, and Shobhan Lal Mishra of Suri Vidyasagar College, for their invaluable teaching and guidance. I am equally thankful to Professors Anil Mani Tripathi, Buddhadev Pal, Krishnendu Mukherjee, Mukut Mani Tripathi, Ravi Pratap Gupta, A.~K.~Mishra, and Harish Chandra of Banaras Hindu University for their academic support and encouragement. I also extend my sincere thanks to my tuition teachers Pathik Sen, Sadananda Badyakar, Jiten Mahata, S.~P.~Tiwari, Sakhanath Ghosh and Partha Mukherjee Sir for their dedicated efforts in shaping my academic foundation.

I am grateful to all my school friends Rajib, Debapriya, Sudipta, Ramesh, Astik, Rana Bhandari, and Bikram—for their companionship and support. I would like to offer special thanks to my friend Rishi Kumar Acharya and his family members—Kaku, Kakima, Tulika Di, and Amma for their invaluable role in shaping my life. I also extend my heartfelt thanks to my college friends Chiranjit, Animesh, Manashi, Archita, Satwata, Ananta, Preeti, and Rituparna for their friendship and encouragement. Furthermore, I convey my sincere regards to my seniors Pinaki Da, Arindam Da, and Abhijit Da for for their consistent support and encouragement.

I warmly acknowledge the support and encouragement of my friends and seniors Bhola, Arvind, Suraj, Ritesh, Rana, Sourav, Hari, Santu, Annoy, Abhishek, Sanjit, Abhisek Da (Abhishek Sarkar), Sudipta Da, Gopal Da, Shantanu Da, Ankit Bhai, Chaman Bhaiya, and Rajkamal Bhaiya for the wonderful times we have shared at IIT Kanpur and for our engaging discussions on mathematics. I am also grateful to my postdoctoral senior, Tejbir, for his invaluable support and guidance throughout this journey.
 
 I extend my sincere appreciation to my other friends for their moral support and motivation, which has inspired me to strive for excellence. Sabyasachi, Prakash, Dipak, Subrata, Chitranjan, Samiran, Aman, M. Bapu, Damini, Debabrata, Rakesh, Guru, Mangala, Nishith, Chandan, Himanshi, Rinki, Vivek, Anil, Kanha, Preeti, Amritha, Dattatreya, Yogendra, Prakhar, Tuhin, Shubham, Tinu, Souvik, Kushal, Krishnogopal...There are too many to name, and I am grateful to all of them. I feel truly blessed to have such friends in my life.

 Finally, and most importantly, I would like to express my deepest gratitude to my family members, whose unwavering love, encouragement, and support have been the foundation of my journey. Like a banyan tree that provides shelter and strength through all seasons, they have stood by me through every challenge and difficulty.

I am especially grateful to my parents, Mr. Aniruddha Dey and Mrs. Bithika Dey, for their endless sacrifices, guidance, and belief in me. I would also like to acknowledge my grandmother, Mrs. Niyati Dey; my aunt, Ms. Anamika Dey; my elder sister, Mrs. Subhra Dey; my brother-in-law, Mr. Shubhendu Dey; my brother, Mr. Apurba Dey; and my beloved nephew, Abir, for their constant affection, encouragement, and support. This achievement would not have been possible without their presence and blessings in my life.

%I would also like to thank my family for their unwavering support throughout my life. In particular, I want to acknowledge my parents, my grandmother, my sister Debi Saha, my sister-in-law Pijush Saha, and my niece Pritishna.

% \begin{flushright}
% \hspace*{1.3in} \textbf{Haripada Roy} \\
% \hspace*{1.3in} September, 2024
%\end{flushright}

\newpage
\cleardoublepage

\begin{KeepFromToc}

\tableofcontents
\end{KeepFromToc}

%\cleardoublepage
%\cleardoublepage
%\addcontentsline{toc}{chapter}{\numberline{}Dedication}

%\cleardoublepage\phantomsection\addcontentsline{toc}{chapter}{Notations and Abbreviations}
%\include{Notations_and_Abbreviations}
\pagenumbering{arabic}
\pagestyle{myruled}

\tikzset{->-/.style={decoration={
  markings,
  mark=at position #1 with {\arrow{>}}},postaction={decorate}}}
  \tikzset{-<-/.style={decoration={
  markings,
  mark=at position #1 with {\arrow{<}}},postaction={decorate}}}

\chapter{Introduction}
\label{Chap1}
In this chapter, we introduce the thesis, which has two main themes, and we will discuss both of them. First, we study systolic embeddings of graphs on translation surfaces, defined via shortest saddle connections between singular points. We review relevant results in this direction and establish the motivation for essential and cellular systolic embeddings. Second, we discuss the systole function on Teichmüller spaces and the Thurston set of surfaces with filling systolic geodesics, highlighting the role of the Bolza surface as a unique maximizer in genus two. Motivated by this, we define Bolza-like surfaces. Finally, we state the central questions of the thesis, present the main results on graph embeddings and Bolza-like surfaces. We conclude the chapter by providing an outline of the subsequent chapters.

\section{Motivation and Background}

\subsection{Systolic Graph Embeddings on Translation Surfaces}
Embedding of graphs on surfaces has emerged as an active area of research in recent times. An embedding of a graph $G$ on a closed surface $S$ is called cellular (respectively, essential) if every component of the complement $S\setminus G$, of $G$ in $S$, is an open topological disk (respectively, not a topological disk). Embedding of graphs on closed hyperbolic surfaces has been studied by Sanki~\cite{sa19} and by Gadgil--Sanki~\cite{san19}. In~\cite{sa19}, Sanki shows that any metric graph $G$ (possibly after scaling its metric) admits an essential and isometric embedding on a closed hyperbolic surface, and further determines the minimum genus of surfaces admitting such embeddings is $\frac{1}{2}(\beta(G)- \zeta(G))+ 2q+r$, where $\beta(G)$ and $\zeta(G)$ denote the Betti number and the Betti deficiency of $G$, respectively, and $q,r$ are the unique integers satisfying $\zeta(G) + 1 = 3q + r,~ 0 \leq r < 3.$ In~\cite{san19}, the authors investigate the problem of characterizing which fat graphs can be realized as systolic graphs of hyperbolic surfaces. In this thesis, we study the systolic embedding of graphs on so-called translation surfaces.

A \emph{translation surface} is obtained by identifying pairs of parallel sides, with opposite orientations, of a finite union of Euclidean polygons by translations. After identification, the vertices of the polygons give rise to cone points or singular points on a translation surface. The total angle at a point corresponding to the equivalence class of a vertex of a polygon is $2k\pi$, for some $k\in\mathbb N$, as the sides of the polygons are glued by Euclidean translations. If $k>1$, then we call a cone point a \emph{singular point}, and if $k=1$, the point is called a \emph{marked point}. A \emph{saddle connection} on a translation surface is a geodesic segment connecting two singular points (may not be distinct), with no singularities in its interior. Among these, a \emph{shortest saddle connection} is one of minimal length. Note that for a translation surface without singular points, we use marked point in place of singular point to define these notions.

Let $S$ be a translation surface. The graph $\Gamma_S$, where vertices correspond to the singular points (or marked points) of $S$ and edges correspond to the shortest saddle connections of $S$ is called the \emph{systolic graph} of $S$.  We say that a finite graph $G$ admits a \emph{systolic embedding} on $S$ if $\Gamma_S$ is isomorphic to $G$. The techniques used to study systolic embeddings on translation surfaces differ notably from those in the hyperbolic setting. In the hyperbolic case, the geometry has constant negative curvature, and systoles are closed geodesics that minimize length in their homotopy class, analyzed using hyperbolic trigonometry, the collar lemma, and fat graph structures. In contrast, translation surfaces are flat except at finitely many singular points, where systoles correspond to the shortest saddle connections. The methods here are primarily combinatorial and Euclidean, relying on polygonal models, cone angle relations, and the maximum number of shortest saddle connections rather than curvature-based tools.

Judge-Parlier~\cite{Pa19} have computed the maximum number of noncontractible loops of minimum length (systole) on a genus two translation surface. 
Boissy-Geninska~\cite{Bo21} have shown that the length of the shortest saddle connection of a translation surface of genus $g$ and having $r$ number of singular points (or marked points) is at most $\left[\frac{\sqrt{3}}{2}\left(2g-2+r\right)\right]^{-\frac{1}{2}}$ and the maximum number of shortest saddle connections is $3(2g-2+r)$. A translation surface attaining both of these bounds is called a \emph{maximal surface}. Later, Columbus-Herrlich-M\"utzel-Schmith\"usen have studied the \emph{saddle connections graph} of a translation surface, whose vertices correspond to the singular points of the surface and whose edges represent saddle connections between them (see ~\cite{MR4790974}, Section 4). In~\cite{MR4790974}, the authors established a precise relationship between the systoles of this graph and those of the underlying translation surface. In particular, they showed that if \(c\) is a systole in the graph with geometric realization \(\gamma\), then \(\gamma\) is a systole of the surface whenever the combinatorial length of \(c\) is not \(3\); moreover, if the combinatorial length of \(c\) is \(3\) and all its interior angles are at least \(\pi\), then \(\gamma\) is also a systole of the surface. Motivated by this concept of saddle connection graph, we study the systolic graph, which can be viewed as the graph of the shortest saddle connections on a translation surface.

We now turn to the second major theme of this thesis, which focuses on the Bolza-like surfaces.

\subsection{Bolza-like Surfaces in the Thurston Set}
Let $\mathcal{T}_g$ denote the Teichm\"uller space of a closed orientable surface $S_g$ of genus $g$ (see~\cite{farb2011primer}). The systole function $\sys : \mathcal{T}_g \to \R,~ g \geq 2$, is defined on the Teichm\"uller space and for a hyperbolic surface $X\in \mathcal{T}_g$, it outputs the length of a shortest non-trivial geodesic in $X$ (see~\cite{Schmutz}). A closed geodesic in the surface $X$ having length $\sys(X)$ is called a \emph{systolic geodesic}. Hyperbolic surfaces with a filling set of systolic geodesics are of great interest due to Thurston's well-known spine problem (see~\cite{thurstonspine}).

A spine in the Teichm\"uller space $\mathcal{T}_g$ is a mapping class group equivariant deformation retract in $\mathcal{T}_g$ of minimal dimension. Thurston claimed that the set $\chi_g$ consisting of those hyperbolic surfaces with a filling set of systolic geodesics, called as the Thurston set, is a spine in $\mathcal{T}_g$. He sketched the proof in the preprint~\cite{thurstonspine}, but unfortunately, the proof is incomplete. Several aspects of the set $\chi_g$ are still not fully known, including its connectedness, contractibility, and dimension.
 Even finding elements of the set $\chi_g$ is a challenging problem. 

Anderson-Parlier-Pettet (see Theorem 3 \cite{anderson2011small}) have shown that if the systolic geodesics fill a closed hyperbolic surface of genus $g$, then it contains at least $\pi \sqrt{g(g-1)}/\log{(4g-2)}$ many systolic geodesics. The authors further constructed hyperbolic surfaces of genus $g$ with $2g+2$ systolic geodesics and showed that any $2g$ of these geodesics are sufficient to fill the surface. In \cite{anderson6146relative}, the authors constructed a sequence of surfaces $S_{g_k}$ in the Thurston set $\chi_{g_k}$ with large Bers constant, where $g_k$ is large enough. Recently, B. Sanki \cite{MR3877282} has constructed hyperbolic surfaces with $2g$ systolic geodesics that fill the surface, and he also introduced a systolic length increasing deformation in the moduli space. More recently, Bourque~\cite{bourque2022dimension} and Mathieu~\cite{mathieu2023estimating}) have studied the dimension of the set $\chi_g$.

In~\cite{thurstonspine}, Thurston has constructed a deformation of a hyperbolic surface $X$ outside the set $\chi_g$ that increases the systolic length. In this thesis, we introduce a deformation on hyperbolic surfaces in $\chi_g$ and show that it is also systolic length increasing. Akrout~\cite{akrout} has shown that the systole function is a topological Morse function on the Teichm\"uller space. By the collar lemma (see Theorem 4.1.1~\cite{Buser}), the function $\sys$ does not have a minima, for $g\geq 2$. The global maxima for the function $\sys$ is unknown except for $g=2$. Jenni \cite{Jenni} has shown that for any surface $X\in \mathcal{T}_2, $ $ \cosh\left(\frac{\sys(X)}{2}\right) \leq \sqrt{2}+1$ and equality occurs for a unique surface (up to the action of the mapping class group), namely the so-called Bolza surface. Note that $l_X(\gamma)$ denotes the length of the curve $\gamma$ in the hyperbolic surface $X$. The Bolza surface has 12 systolic geodesics, and they decompose the surface into triangles with each triangle of type $(4,4,4)$ (see Section 5, \cite{Schmutz}). Note that a hyperbolic triangle is called of type $(p,q,r)$, for some integers $p,q,r\in \N$ satisfying $\frac{1}{p}+ \frac{1}{q}+ \frac{1}{r}<1$ if its interior angles are $\pi/p, \pi/q$ and $\pi/r$. Motivating from this, we define the following.
\begin{definition}
    A closed hyperbolic surface is called Bolza-like, if the systolic geodesics decompose it into $(p,q,r)$-triangles for some integers $p,q,r$ satisfying $\frac{1}{p}+ \frac{1}{q}+ \frac{1}{r}<1$ .
\end{definition}

In this thesis, we construct Bolza-like surfaces for infinitely many genera $g$. Next, we will see an interesting application of the Bolza-like surfaces that we obtain, namely the construction of global maxima of the systole function. Furthermore, we show that the distribution of the systolic geodesics is symmetric in the sense that any simple closed geodesic intersects the Bolza-like surface at an even number of points.

\section{Central questions and the main results}

 For a translation surface $S$, we define its \emph{systolic graph} $\Gamma_S$ to be the graph whose vertices are the singular (or marked) points of $S$, and whose edges are the shortest saddle connections of $S$. A finite graph $G$ is said to admit a \emph{systolic embedding} on $S$ if $G \cong \Gamma_S$. An embedding of a graph on a translation surface is said to be \emph{cellular} (respectively, \emph{essential}) if each complementary region is a topological disk (respectively, not a topological disk). Motivated by the work~\cite{san19}, in this thesis, we study the following natural questions:

\begin{question}{\label{question 1}}
     Given a finite graph $G$, does there exist a translation surface $S$ such that $G$ admits a systolic embedding on $S$?
\end{question}

Here, we study two types of systolic embeddings of graphs on translation surfaces: (i) essential and (ii) cellular. In the context of essential-systolic embeddings, we show that any finite graph $G$ can be embedded on a translation surface of some genus depending only on the graph $G$. In particular, to answer Question~\ref{question 1}, for essential-systolic embedding, we obtain the following main result.

\begin{theorem}\label{Main Result} (\cite{MR4985436})
Every finite graph admits an essential-systolic embedding on a translation surface. 
\end{theorem}

We prove Theorem~\ref{Main Result} in two steps. In the first step, we prove the theorem below:

\begin{theorem}\label{main theorem 1} (\cite{MR4985436})
For every odd integer $n \geq 5$, the complete graph $K_n$ admits an essential-systolic embedding on a translation surface $S_{g_n}$ of genus $g_n$, where 
\[
g_n = 1 - n \left( 1 - \frac{n-1}{2} \right).
\]
\end{theorem}

The construction of $S_{g_n}$ provides an explicit translation surface that realizes such an embedding of $K_n$. In the second step, we prove the following proposition:

\begin{proposition}\label{proposition 1} (\cite{MR4985436})
For every odd $n \geq 5$, every subgraph of $K_n$ admits an essential-systolic embedding on a translation surface of genus $g_n$.
\end{proposition}

We obtain such an embedding for any subgraph by appropriately modifying the embedding of $K_n$ without altering the genus. Since every finite graph $G$ can be realized as a subgraph of some complete graph $K_n$ with odd $n \geq 5$, Theorem~\ref{main theorem 1} together with Proposition~\ref{proposition 1} yields Theorem~\ref{Main Result}.

\vspace{0.3 cm}

Unlike essential-systolic embedding, not all graphs admit cellular-systolic embedding on translation surfaces. For example, the complete graph $K_3$ admits topological cellular embedding only on the sphere, which is not a translation surface (see Lemma \ref{Non example}). Thus, in the context of cellular systolic embedding, for response to Question (\ref{question 1}), it is natural to ask whether there exists a collection of graphs that can be embedded in this way. In this context, we have identified a family of graphs $\Sigma_n$, which denotes the wedge of $n$ circles for $n \geq 2$. In addition, we find another following collection of graphs, each with more than one vertex, that can also be embedded on a translation surface in this manner:

\begin{enumerate}
    \item $\mathcal{P} = \left\{\, G_{n,m} = (V_n, E_{n,m}) \mid n \geq 2,\; m \in \mathbb{N} \right\}$, where 
    \begin{align*}
        V_n &= \{ v_0, v_1, \dots, v_{n-1} \} \text{ and}\\  E_{n,m} &= \{ e_{ij} \mid 0 \leq i \leq n-1,\; 1 \leq j \leq 2m \}.
    \end{align*}
    
     Note that $e_{ij}$ are the $2m$ parallel edges between~$v_i$ and~$v_{i+1}$ (indices of $v$'s taken modulo $n$).

\item $\mathcal{Q} = \left\{\, G_{n,m}^\prime = (V_n, E_{n,m}) \;\mid \; n \geq 2,\; m \in \mathbb{N}\right\}$, where 
\begin{align*}
V_n &= \left\{ v_0, v_1, \dots, v_{n-1}\right \}\text{ and}\\ E_{n,m} &= L \cup P \text{ where}\\ L &= \{\, \ell_{ij} \mid 0 \leq i \leq n-1,\; 1 \leq j \leq m \} \text{ and}\\ P &= \{\, e_{ij} \mid 0 \leq i \leq n-1,\; 1 \leq j \leq m \},
\end{align*}
such that each $\ell_{ij}$ is a loop at vertex $v_i$ and each $e_{ij}$ connects vertex $v_i$ to vertex $v_{i+1}$ (indices taken modulo $n$). That is, every vertex $v_i$ has exactly $m$ loops, and every pair  $(v_i, v_{i+1})$ is connected by $m$ parallel edges.
    
 \end{enumerate} 

Then, we prove the following result.

\begin{theorem}\label{another rich collection} (\cite{MR4985436})
Every graph $G\in \mathcal{P} \sqcup \mathcal{Q}$ admits a cellular-systolic embedding on some translation surface.    
\end{theorem}

\vspace{0.4 cm}
Here, we obtain a complete answer to Question~\ref{question 1} in the case of essential-systolic embeddings, and a partial answer in the case of cellular-systolic embeddings. 

So far, we have studied the problem of the existence of systolic embeddings of finite graphs. Next, we focus on the problem of extremum genera. It is therefore natural to pose the following question.

\begin{question}\label{question 2}
What are the minimum and maximum possible genera of translation surfaces $S$ on which a given graph $G$ admits a systolic embedding?
\end{question}

If a graph admits an essential-systolic embedding on a translation surface of genus $g$, then there exists an integer $g'>g$ such that the graph admits such an embedding on a translation surface of genus $g'$. Therefore, the maximum-genus problem of such essential embeddings does not make sense, and hence we study the minimum-genus problem. We define $g^{\min}_G$ as the minimum possible genus of the surfaces on which the graph $G$ admits an essential-systolic embedding. To address Question~\ref{question 2}, we have established the following result.

\begin{theorem} \label{upper bound} (\cite{MR4985436})
For a graph \( G \) with \( n \) vertices, an upper bound on the minimum genus $g_G^{\min}$ of a translation surface that admits an essential-systolic embedding is given by
\[
g^{\min}_G \leq 
\begin{cases}
1 - n\left(1 - \frac{n - 1}{2}\right), & \text{if } n \text{ is odd and } n \geq 5, \\[0.5em]
1 - (n + 1)\left(1 - \frac{n}{2}\right), & \text{if } n \text{ is even and } n \geq 5, \\[0.5em]
6, & \text{if } n < 5.
\end{cases}
\]
\end{theorem}

Now, to address Question~\ref{question 2}, in the context of cellular-systolic embeddings, we establish the following results for $\Sigma_n$, $n \geq 2$, which together provide answers to Questions~\ref{question 1} and~\ref{question 2}.

\begin{theorem} \label{cellular embedding} (\cite{MR4985436})
Let $\Sigma_n$ be the wedge of $n$ circles, where $n \geq 2$. Then we have the following:
\begin{enumerate}
    \item $\Sigma_n$ admits cellular-systolic embedding on a translation surface of genus $\left\lfloor \frac{n}{2} \right\rfloor$. Furthermore, this is the maximum genus for which such an embedding is possible.
    
    \item The minimum genus of a translation surface on which $\Sigma_n$ admits cellular-systolic embedding is $\left\lceil \frac{n+3}{6} \right\rceil$.
    
    \item For any integer $g$ satisfying $\left\lceil \frac{n+3}{6} \right\rceil \leq g \leq \left\lfloor \frac{n}{2} \right\rfloor$, the graph $\Sigma_n$ admits  cellular-systolic embedding on a translation surface of the genus $g$.
\end{enumerate}
\end{theorem}

\vspace{0.4 cm}
In the rest of the chapter, our focus will be on Bolza-like surfaces in Thurston set.

Thurston has introduced an idea of deformation for increasing the systolic length, i.e., the length of a systolic geodesic of a surface which is not an element of the set $\chi_g$. Thurston's idea is the following: consider a simple closed geodesic that does not intersect any of the systolic geodesics. Now, cut the surface along this curve and attach two funnels along these boundaries. Take a bi-infinite geodesic arc that intersects all the systolic geodesics and insert a thin strip along this arc. If we delete the funnels and reglue the boundaries, then we get a hyperbolic surface of the same genus and with a larger systolic length. Thus, by deforming a hyperbolic surface that is not a member of the Thurston set, we obtain a new surface with a larger systole. Now, one can ask the question:

\begin{question}{\label{question 3}}
    Does there exist a systolic length increasing deformation on a surface $X\in\chi_g$?
\end{question}

 In this thesis, we introduce a deformation on hyperbolic surfaces in $\chi_g$ and show that this deformation is also systolic length increasing. So, we answer Question (\ref{question 3}) in the theorem below.

 \begin{theorem}\label{Bolza like surfaces} (\cite{DeySahaSankiToAppear})
   For every integer $g=kn+1$, where $k,n\in \N, k\geq 2, n\geq 4$, there is a continuous function $\mathcal{D}_g:\left[ 0, \frac{\pi}{12}\right] \to \mathcal{M}_g$ satisfying the following.
    \begin{enumerate}
        \item $\mathcal{D}_g(t)\in \chi_g$, for all $t\in \left[ 0, \frac{\pi}{12}\right]$, where $\chi_g$ is the Thurston set.
        \item $\mathrm{sys}(\mathcal{D}_g(t_1))<\mathrm{sys}(\mathcal{D}_g(t_2))$, for all $0\leq t_1 < t_2 \leq \frac{\pi}{12}.$
        \item The surface $\mathcal{D}_g\left(\frac{\pi}{12}\right)$ has $(6g-6)$ many systolic geodesics that give a triangle decomposition of the surface. In other words, $\mathcal{D}_g\left(\frac{\pi}{12}\right)$ is Bolza-like.
    \end{enumerate}
\end{theorem}

\begin{remark}
Theorem~\ref{Bolza like surfaces} applies only to genera $g$ such that $g-1=kn$ for some
integers $k\geq 2$ and $n\geq 4$. Thus, for $g\geq 2$, the genera
not covered by this theorem are precisely those for which $g-1$ is
either $1$, a prime number, or one of $4,6,9$.
\end{remark}

In addition, we show that the distribution of the systolic geodesics is symmetric in the sense that any simple closed geodesic intersects the Bolza-like surface at an even number of points. We prove the following theorem for this.

\begin{theorem}\label{Distribution} (\cite{DeySahaSankiToAppear})
    Any simple closed geodesic on the Bolza-like surfaces $\mathcal{D}_g\left(\frac{\pi}{12} \right)$, as obtained in Theorem \ref{Bolza like surfaces}, intersects the collection $\mathrm{Sys}\left(\mathcal{D}_g\left(\frac{\pi}{12} \right)\right)$ of all systolic geodesics at an even number of points.
\end{theorem}
Note that in this result, we constract a one-parameter family of explicit examples of elements in the Thurston set $\chi_g$.
\section{Outline of the thesis}

In Chapter \ref{chap2}, we explore translation surfaces formed by gluing Euclidean polygons via translations. Originating from billiard dynamics in rational polygons, these surfaces connect deeply with flat geometry, complex analysis, and dynamical systems. After outlining the historical development, we define translation surfaces and discuss their singular and marked points. Key concepts such as cone angles, saddle connections, and systolic connections (shortest saddle connections) are introduced. The chapter concludes with results by Boissy-Geninska, which provide bounds on the length and number of systolic connections based on the surface's genus and singularities, and characterize optimal configurations via equilateral triangulations.

Chapter \ref{chap3} studies embeddings of graphs on translation surfaces with emphasis on systolic geometry. We define systolic, essential, and cellular embeddings, and prove in Theorem~\ref{Main Result} that every finite graph admits an essential-systolic embedding. We derive an upper bound on the minimal genus required for such embeddings (Theorem~\ref{upper bound}). For the wedge of $n$ circles $\Sigma_n$, we show the existence of cellular-systolic embeddings and compute the extremal genera (Theorem~\ref{cellular embedding}). Finally, Theorem~\ref{another rich collection} identifies additional graph families admitting similar embeddings.

Chapter~\ref{Chapter 4} introduces the foundational concepts of hyperbolic geometry. We begin with a brief overview of its historical development, followed by an introduction to various models of the hyperbolic plane, with particular emphasis on the Poincaré upper half-plane model. Using this model, we explore key geometric features, including geodesics, distances, isometries, areas, angles, and trigonometric relations.

In Chapter~\ref{Chapter 5}, we explore hyperbolic surfaces in depth. We begin by defining hyperbolic surfaces and present two fundamental theorems for constructing them: one via gluing hyperbolic polygons and another using group actions by isometries. We then examine geodesics on hyperbolic surfaces and state the Collar Lemma, which asserts that shorter curves have wider collars. The chapter further introduces the mapping class group—orientation-preserving self-homeomorphisms up to isotopy and highlights Dehn twists as its generators. We also briefly survey Teichmüller space and the moduli space of surfaces, concluding with a discussion on generalized Morse functions on the moduli space.

In Chapter~\ref{chapter 6}, we focus on the construction of Bolza-like surfaces within the Thurston set. We begin by defining Bolza-like surfaces and reviewing essential preliminaries. We then construct such surfaces for infinitely many genera $g$. An important application of these constructions is the realization of global maxima of the systole function. Furthermore, we show that the distribution of systolic geodesics on these surfaces is symmetric, in the sense that every simple closed geodesic intersects the surface in an even number of points.

\chapter[Translation surfaces]{Translation surfaces}\label{chap2}

In this chapter, we introduce the concept of \emph{translation surfaces}, which are constructed by pairwise identifying the sides of a given collection of Euclidean polygons via translations. These surfaces arise naturally in the study of billiards in rational polygons and are deeply connected to flat geometry, complex analysis, and dynamical systems. We begin with a historical overview, highlighting the foundational work of Riemann, followed by important developments by Masur, Veech, and others, who establish the connection between billiard dynamics and geodesic flows on flat surfaces with singularities. The formal definition of a translation surfaces is presented, along with a geometric interpretation using polygonal gluings. We describe how the angle at a point on the surface is defined by summing the interior angles at the corresponding preimages and note that this angle is always a multiple of $2\pi$. We introduce the notions of \emph{singular points}, points with a cone angle greater than $2\pi$, and \emph{marked points}, together with illustrative examples. A translation surface is thus characterized as a Euclidean surface with isolated singularities. We then define \emph{saddle connections}, which are geodesic segments that connect singular (or marked) points without passing through any other singularity in their interior. Special emphasis is given to \emph{systolic connections}, defined as the shortest saddle connections. The chapter concludes with two theorems from Boissy-Geninska that bound the \emph{maximum length} and \emph{number} of systolic connections on a translation surface in terms of its genus and number of singular or marked points. These results characterize optimal surfaces in terms of equilateral triangulations and maximal systolic configurations.

\section{Historical Background}
The origins of translation surfaces lie at the crossroads of classical geometry, complex analysis, and dynamical systems. The idea of equipping Riemann surfaces with flat structures coming from holomorphic $1$-forms can be traced back to the pioneering work of Bernhard Riemann in the mid-19th century~\cite{Riemann1857}, who studied Abelian differentials as part of his theory of complex algebraic curves. These structures implicitly carried the geometry of what we now call translation surfaces, although the terminology and dynamical applications only crystallized much later.

A decisive motivation for the modern theory arose from the study of billiards in polygons with rational angles. In the 1920s, George Birkhoff formulated fundamental questions about billiard dynamics, including the existence of periodic orbits and the ergodic properties of billiard flows~\cite{Birkhoff1927}. Despite this early vision, progress remained limited for several decades. A breakthrough occurred in the late 1970s and early 1980s, when Howard Masur and William Veech developed the analytic and dynamical framework that firmly connected billiard flows in rational polygons to geodesic flows on translation surfaces with conical singularities~\cite{Masur82,Veech89}. Their key insight was the unfolding construction: reflecting a billiard trajectory across the sides of a rational polygon produces a straight-line trajectory on a compact flat surface. This translation surface encodes the entire billiard dynamics, opening the door to the use of Teichmüller theory and ergodic theory.

Masur proved that almost every direction on a translation surface defines a uniquely ergodic straight-line flow~\cite{Masur82}, while Veech introduced the notion of Veech surfaces and showed that they exhibit striking dynamical rigidity: for these surfaces, directional flows are either completely periodic or uniquely ergodic~\cite{Veech89}. This class of surfaces, and the associated Veech groups, provided the first concrete examples where deep connections between geometry, dynamics, and arithmetic could be studied.

Building on these breakthroughs, subsequent research revealed a rich interplay between translation surfaces, interval exchange transformations, and the dynamics of moduli spaces of Abelian differentials. Masur and Tabachnikov provided a comprehensive account of rational billiards and flat structures~\cite{MasurTabachnikov02}, while Anton Zorich’s influential surveys highlighted the connections to ergodic theory and the geometry of moduli spaces~\cite{Zorich06}. Curtis McMullen further demonstrated that translation surfaces arising from Hilbert modular varieties yield Teichmüller curves, providing striking examples of rigidity phenomena at the intersection of number theory and dynamics~\cite{McMullen03}.

Through these developments, translation surfaces have become a central object in modern geometry and dynamical systems. They now play a prominent role in diverse areas, ranging from Teichmüller dynamics and interval exchanges to billiard models in physics and deep questions in algebraic geometry and number theory. The continuing growth of this field underscores the enduring significance of the geometric and dynamical ideas first glimpsed in Riemann’s 19th-century work and crystallized in the breakthroughs of Masur and Veech in the late 20th century.

\section{Definition of translation surfaces}
A translation surface admits several equivalent definitions~\cite{AR15}. In this thesis, we adopt the one given in \cite{Masur06}, as it is both simple and well-suited to our objectives.

\begin{definition}\label{Translation Surface} 
A translation surface is a finite union of Euclidean polygons $\{ P_1, P_2, \dots, P_n \}$, equipped with a side pairing,
 such that:
\begin{enumerate}
    \item The boundary of every polygon is oriented so that the interior of the polygon lies to the left.
    \item For every $1\leq j\leq n$ and every oriented side $s_j$ of $P_j$, there exists $k\in\{1,2,\cdots,n\}$ and an oriented side $s_k$ of $\triangle_k$ such that $s_j$ and $s_k$ are parallel and of the same length. They are glued together in the opposite orientation by a parallel translation.
\end{enumerate}
\end{definition}

Therefore, a translation surface can be described as
\[
S = \left(\bigsqcup_{i=1}^n P_i \right)\bigg/\sim,
\]
where the $P_i$ are polygons in the plane and $\sim$ denotes the identification of edges via translations. 

We now present some well-known examples of translation surfaces.
\begin{example}\label{example 1}
By identifying the opposite sides of a square, one obtains a translation surface of genus one, namely the torus (see Figure~\ref{Torus}).
\end{example}

\begin{figure}[htbp]
\begin{center}
\begin{tikzpicture}[scale=3]
\draw (0,0)--(1,0)--(1,1)--(0,1)--(0,0);
\draw[red] (0,0) to (1,0);
\draw[red] ((1,1) to (0,1);
\draw[green] (0,0) to (0,1) ;
\draw[green] (1,0) to (1,1);
\draw (0,0) node{\tiny$\bullet$};
\draw (1,0) node{\tiny$\bullet$};
\draw (0,1) node{\tiny$\bullet$};
\draw (1,1) node{\tiny$\bullet$};
\draw [-{Stealth[red]}]
(0.55,0)--(0.56,0);
\draw [-{Stealth[red]}]
(0.55,1)--(0.56,1);
\draw [-{Stealth[green]}]
(0,0.55)--(0,0.56);
\draw [-{Stealth[green]}]
(1,0.55)--(1,0.56);
\draw (0.5,-0.1) node{a};
\draw (0.5,1.1) node {a};
\draw (-0.1, 0.5) node{b};
\draw (1.1, 0.5) node{b};
\draw (2.5,0.5) ellipse (0.8cm and 0.4cm);
\draw[red] (2.5,0.5) ellipse (0.5cm and 0.2cm);
\draw [-{Stealth[red]}]
(2.5,0.7)--(2.51,0.7);
\draw (2.3,0.5)..controls(2.5,0.6)..(2.7,0.5);
\draw (2.37,0.6)..controls(2.5,0.4)..(2.64,0.6);
\draw[green] (2.5,0.40) arc
[
start angle =90,
end angle=270,
x radius =0.07cm,
y radius =.15cm
];
\draw[densely dotted, green] (2.5,0.1) arc
[
start angle =270,
end angle=450,
x radius =0.07cm,
y radius =.15cm,
];
\draw (2.45,0.30) node{\tiny$\bullet$};
\draw [-{Stealth[green]}]
(2.446,0.2)--(2.446,0.21);
\draw (2.5,0.82) node{a};
\draw (2.35,0.2) node{b};

\end{tikzpicture}
\end{center}
\caption{}
\label{Torus}
\end{figure}

\begin{example} \label{e2}
A translation surface of genus $2$ can be constructed by identifying the opposite sides of a regular octagon. In this construction, all vertices are identified to a single point (see Figure~\ref{octagonal surface}).
\end{example}

\begin{figure}[htbp]
 \begin{center}
\begin{tikzpicture}[scale=2]
\draw[red] (0,0)--(1,0);
\draw[blue] (1,0)--(1.707,0.707);
\draw[green] (1.707,0.707)--(1.707,1.707);
\draw  (1.707,1.707)--(1,2.414);
\draw[red] (0, 2.414)--(1,2.414);
\draw[blue] (0,2.414)--(-0.707,1.707);
\draw[green] (-0.707,1.707)--(-0.707,0.707);
\draw (-0.707,0.707)--(0,0);
\draw [-{Stealth}]
(0.5,0)--(0.51,0);
\draw [-{Stealth}]
(0.5,2.414)--(0.51,2.414);
\draw (0.5,2.6) node {$e_1$};
\draw (0.5,-0.2) node {$e_1$};

\draw [-{Stealth}]
(1.3535,0.3535)--(1.4,0.4);
\draw [-{Stealth}]
(-0.3535,2.0605)--(-0.34,2.074);
\draw (1.55,0.30) node{$e_4$};
\draw (-0.45,2.2) node{$e_4$};

\draw [-{Stealth}]
(1.707,1.207)--(1.707,1.1);
\draw [-{Stealth}]
(-0.707,1.207)--(-0.707,1.1);
\draw (1.9,1.207) node{$e_3$};
\draw (-0.9,1.207) node{$e_3$};

\draw [-{Stealth}]
(1.3535,2.0605)--(1.34,2.074);
\draw [-{Stealth}]
(-0.3535,0.3535)--(-0.36,0.36);
\draw (1.55,2.1) node{$e_2$};
\draw (-0.5,0.2) node{$e_2$};

\draw (0,0) node{\tiny$\bullet$};
\draw (1,0) node{\tiny$\bullet$};
\draw (1.707,0.707) node{\tiny$\bullet$};
\draw (1.707,1.707) node{\tiny$\bullet$};
\draw (0,2.414) node{\tiny$\bullet$};
\draw (1,2.414) node{\tiny$\bullet$};
\draw (-0.707,1.707) node{\tiny$\bullet$};
\draw (-0.707,0.707) node{\tiny$\bullet$};

\end{tikzpicture}
\end{center} 
\caption{}
\label{octagonal surface}
\end{figure}

\begin{example}\label{e3}
Another surface of interest is obtained from the double pentagon depicted in Figure~\ref{double pentagon}. From the combinatorial viewpoint, this surface is equivalent to the regular octagon. It has genus $2$, with all vertices identified into a single point.

\end{example}

\begin{figure}[htbp]
\begin{center}
\begin{tikzpicture}[scale=2]
\draw[red] (-0.587,-0.809)--(0.587,-0.809);
\draw[blue](0.587,-0.809)--(0.951,0.309);
\draw[green] (0.951,0.309)--(0,1);
\draw[yellow] (0,1)--(-0.951,0.309);
\draw (-0.951,0.309)--(-0.587,-0.809);

\draw[red] (2.587,2.809)--(1.413,2.809);
\draw[blue] (1.413,2.809)--(1.049,1.691);
\draw[green] (1.049,1.691)--(2,1);
\draw[yellow] (2,1)--(2.951,1.691);
\draw (2.951,1.691)--(2.587,2.809);

\draw (-0.587,-0.809) node{\tiny$\bullet$};
\draw (0.587,-0.809) node{\tiny$\bullet$};
\draw (0.951,0.309) node{\tiny$\bullet$};
\draw (0,1) node{\tiny$\bullet$};
\draw (-0.951,0.309) node{\tiny$\bullet$};

\draw (2.587,2.809) node{\tiny$\bullet$};
\draw (1.413,2.809) node{\tiny$\bullet$};
\draw (1.049,1.691) node{\tiny$\bullet$};
\draw (2,1) node{\tiny$\bullet$};
\draw (2.951,1.691) node{\tiny$\bullet$};

\draw [-{Stealth}] (0,-0.809)--(0.1,-0.809);
\draw [-{Stealth}] (2,2.809)--(2.1,2.809);
\draw [-{Stealth}] (0.769,-0.25)--(0.8,-0.155);
\draw [-{Stealth}] (1.231,2.25)--(1.3,2.462);
\draw [-{Stealth}] (0.475,0.654)--(0.5,0.637);
\draw [-{Stealth}] (1.524,1.345)--(1.6,1.290);
\draw [-{Stealth}] (-0.475,0.654)--(-0.5,0.637);
\draw [-{Stealth}] (2.475,1.345)--(2.4,1.290);
\draw [-{Stealth}] (-.769,-0.25)--(-.8,-0.19);
\draw [-{Stealth}] (2.769,2.25)--(2.72,2.38);

\draw (0.05,-.99) node{$e_1$};
\draw (2.05,2.99) node{$e_1$};
\draw (1,-0.3) node{$e_5$};
\draw (1,2.25) node{$e_5$};
\draw (0.53,0.85) node{$e_4$};
\draw (1.3,1.2) node{$e_4$};
\draw (-0.56,0.85) node{$e_3$};
\draw (2.7,1.345) node{$e_3$};
\draw (-1,-0.25) node{$e_2$};
\draw (2.95,2.25) node{$e_2$};
\end{tikzpicture}
\end{center}
\caption{}
\label{double pentagon}
\end{figure}

\begin{example} \label{e4}
A translation surface of genus $2$ can be obtained by identifying the parallel sides of three squares, as illustrated in Figure~\ref{3 square tiled}. Here also all the vertices identified into one point. This surface is called $L$-shaped surface.
\end{example}

\begin{figure}[htbp]
\begin{center}
\begin{tikzpicture}[scale=3]
\draw (0,0)--(2,0)--(2,1)--(1,1)--(1,2)--(0,2)--(0,0);
\draw[red] (1,2) to (0,2);
\draw[red] (0,0) to (1,0);
\draw[green] (0,2) to (0,1);
\draw[green] (1,2) to (1,1);
\draw[blue] (0,1) to (0,0);
\draw[blue] (2,0) to (2,1);
\draw[yellow] (2,1) to (1,1);
\draw[yellow] (2,0) to (1,0);
\draw[dashed] (0,1) to (1,1);
\draw[dashed] (1,1) to (1,0);
\draw (0,0) node{\tiny$\bullet$};
\draw (1,0) node{\tiny$\bullet$};
\draw (2,0) node{\tiny$\bullet$};
\draw (2,1) node{\tiny$\bullet$};
\draw (1,1) node{\tiny$\bullet$};
\draw (0,1) node{\tiny$\bullet$};
\draw (0,2) node{\tiny$\bullet$};
\draw (1,2) node{\tiny$\bullet$};
\draw [-{Stealth[red]}]
(0.5,0)--(0.51,0);
\draw [-{Stealth[red]}]
(0.5,2)--(0.51,2);
\draw (0.5,-0.15) node{a};
\draw (0.5,2.15) node{a};
\draw [-{Stealth[blue]}]
(0,0.5)--(0,0.51);
\draw [-{Stealth[blue]}]
(2,0.5)--(2,0.51);
\draw (-0.15,0.5) node{b};
\draw (2.15,0.5) node{b};
\draw [-{Stealth[green]}]
(0,1.5)--(0,1.51);
\draw [-{Stealth[green]}]
(1,1.5)--(1,1.51);
\draw (-0.15,1.5) node{c};
\draw (1.15,1.51) node{c};
\draw [-{Stealth}]
(1.5,0)--(1.51,0);
\draw [-{Stealth}]
(1.5,1)--(1.51,1);
\draw (1.5,-0.15) node{d};
\draw (1.5,1.15) node{d};
 \end{tikzpicture}
\end{center}    
\caption{}
\label{3 square tiled}
\end{figure}

\begin{example}\label{e5}
Another translation surface can be obtained by arranging four squares and identifying their parallel sides as shown in Figure~\ref{4 square tiled}. In this construction, all vertices are identified into two distinct points.
\end{example}

\begin{figure}[htbp]   
\begin{center}
\begin{tikzpicture}[scale=2]
\draw (0,0)--(2,0)--(2,2)--(1,2)--(1,0);
\draw (0,0)--(0,1)--(3,1)--(3,2)--(2,2);
\draw[red] (3,2)--(2,2);
\draw[red] (2,1)--(3,1);
\draw[green] (1,2)--(2,2);
\draw[green] (1,0)--(2,0);
\draw[blue] (0,1)--(1,1);
\draw[blue] (0,0)--(1,0);
\draw[brown] (0,0)--(0,1);
\draw[brown] (2,0)--(2,1);
\draw[yellow] (1,1)--(1,2);
\draw[yellow] (3,1)--(3,2);
\draw (0,0)  node{\tiny$\bullet$};
\draw  (0,1) node{\tiny$\bullet$};
\draw (2,2) node{\tiny$\bullet$};
\draw (2,1) node{\tiny$\bullet$};
\draw (2,0) node{\tiny$\bullet$};
\draw[red] (1,2) node{\tiny$\bullet$};
\draw[red] (1,1) node{\tiny$\bullet$};
\draw[red] (1,0) node{\tiny$\bullet$};
\draw[red] (3,1) node{\tiny$\bullet$};
\draw[red] (3,2) node{\tiny$\bullet$};
\draw [-{Stealth[green]}]
(1.5,0)--(1.51,0);
\draw [-{Stealth[green]}]
(1.5,2)--(1.51,2);
\draw [-{Stealth[red]}]
(2.5,1)--(2.51,1);
\draw [-{Stealth[red]}]
(2.5,2)--(2.51,2);
\draw [-{Stealth[blue]}]
(0.5,0)--(0.51,0);
\draw [-{Stealth[blue]}]
(0.5,1)--(0.51,1);
\draw [-{Stealth[brown]}]
(0,0.5)--(0,0.51);
\draw [-{Stealth[brown]}]
(2,0.5)--(2,0.51);
\draw [-{Stealth}]
(1,1.5)--(1,1.51);
\draw [-{Stealth}]
(3,1.5)--(3,1.51);
\draw (2.5,.85) node{a};
\draw (2.5,2.15) node{a};
\draw (3.15,1.5) node{b};
\draw (0.85,1.5) node{b};
\draw (0.5,-0.15) node{c};
\draw (0.5,1.15) node{c};
\draw (1.5,-0.15) node{d};
\draw (1.5,2.15) node{d};
\draw (-0.15,0.5) node{e};
\draw (2.15,0.5) node{e};

\end{tikzpicture}
\end{center}
\caption{}
\label{4 square tiled}
\end{figure}

\begin{lemma} (Non-example of a translation surface) \label{Non example}
The sphere is not a translation surface. 
\end{lemma}

\begin{proof}
Suppose, for the sake of contradiction, that the sphere $S^2$ admits a translation surface structure. Let the singular points of $S^2$ be $p_1,\dots,p_s$, and let the cone angle at $p_i$ be $\theta_i=(k_i+1)2\pi$ for integers $k_i \ge 0$, $i=1,2,\dots,s$ by Lemma~\ref{angle at singularity}. 

Now, by Proposition~\ref{Gauss-Bonnet}, for a translation surface of genus $g$ we have
\[
\sum_{i=1}^s k_i = 2g - 2.
\]  
For the sphere, $g = 0$, so this formula gives
\[
\sum_{i=1}^s k_i = 2\cdot 0 - 2 = -2.
\]  
This is impossible, since each $k_i \ge 0$, and therefore their sum cannot be negative. Hence, the sphere cannot admit a translation surface structure.
\end{proof}

A proof of this lemma using a different approach is provided in Section $2.4$, page $13$ of \cite{Mas22}.

\section{Singular points}
Let $P$ be a polygon and let $x \in P$. The angle around $x$, denoted by $\measuredangle(x)$, is defined as follows:
\[
\measuredangle(x)=
\begin{cases} 
      2\pi, & x \in \operatorname{Int}(P), \\[6pt]
      \pi, & x \in \partial P \setminus \text{Vertices}, \\[6pt]
      \text{the interior angle at } x, & x \in \text{Vertices}.
\end{cases}
\]

\begin{example}
   In the regular octagon depicted in Figure~\ref{Regular octagon},  $\measuredangle(x_1)= \frac{3\pi}{4}$, $\measuredangle(x_2)=2\pi$ and $\measuredangle(x_3)=\pi.$
\end{example}

\begin{figure}[htbp]
\begin{center}
\begin{tikzpicture}[scale=2]
\draw (-2,0)--(-1,0);
\draw (-1,0)--(-0.293,0.707);
\draw (-0.293,0.707)--(-0.293,1.707);
\draw  (-0.293,1.707)--(-1,2.414);
\draw (-2, 2.414)--(-1,2.414);
\draw (-2,2.414)--(-2.707,1.707);
\draw(-2.707,1.707)--(-2.707,0.707);
\draw (-2.707,0.707)--(-2,0);

\draw[red] (-2,0) node{\tiny$\bullet$} ;
\draw  (-2, -0.2) node {$x_1$};
\draw[] (-2.2,1) node {\tiny$\bullet$};
\draw (-2,1) node {$x_2$};
\draw[green] (-1.5,0) node {\tiny$\bullet$};
\draw (-1.5,-0.2) node {$x_3$};
\end{tikzpicture}
\end{center}
\caption{}
\label{Regular octagon}
\end{figure}

Now, let $S=\displaystyle{\left(\bigsqcup\limits_{i=1}^nP_i\right)/\sim}$ be a translation surface and $\bar{x} = \{x_1, x_2, \dots, x_n\} \in S$, represented by its pre-images $x_i$ in the polygons. The angle at $\bar{x}$ is defined as follow:
\[
\measuredangle(\bar{x}) := \sum_{i=1}^n \measuredangle(x_i).
\]

\begin{lemma} \label{angle at singularity}\cite{Mas22}
On a translation surface $S$, the total angle at any point $\bar{x} \in S$ is an integer multiple of $2\pi$. More precisely, 
$\measuredangle(\bar{x}) = 2k\pi$ for some $k \in \mathbb{N}$.
\end{lemma}

\begin{definition}

Let $S= \displaystyle{\left(\bigsqcup\limits_{i=1}^nP_i\right)/\sim}$ be a translation surface and $\bar{x}=\{x_1, x_2, \dots, x_m\}\in S$, where $x_j\in \displaystyle\left(\bigsqcup\limits_{i=1}^nP_i\right) $ is  a vertex of some $P_j$. Then $\bar{x}$ is called a singular point of $S$ if $\measuredangle{\bar{x}}>2\pi$. Otherwise, it is called a marked point.
\end{definition}

\begin{example}
    In Example~\ref{example 1}, there is no singular point but one marked point.
\end{example}

\begin{example} In Examples~\ref{e2}, \ref{e3} and \ref{e4} the translation surfaces has only one singular point with cone angle $6\pi$.
\end{example}

\begin{example}
In the Example~\ref{e5}, the translation surface has two singular points, each with cone angle $4\pi$.
\end{example}

\begin{remark}
    A translation surface is an Euclidean surface with singular points.
\end{remark}

We now prove the following proposition, which plays a fundamental role.

\begin{proposition} \label{Gauss-Bonnet}
Let $S$ be a translation surface of genus $g$, with singularities $p_1, \dots, p_s$. 
If the cone angle at $p_i$ is $(k_i + 1) 2\pi$ for some integer $k_i \ge 0$, then
\[
\sum_{i=1}^{s} k_i = 2g - 2.
\]
\end{proposition}

\begin{proof}
The proposition is established through the following steps.

\emph{Step 1: Polygonal decomposition and total angles.} \\
Let the polygonal decomposition consist of $F$ polygons. Denote by $E$ the number of edges in the resulting cell decomposition after identification, and let $V = s$ denote the number of vertices corresponding to the singular points $p_1, \dots, p_s$. Let $n_j$ be the number of sides of the $j$-th polygon, and define
\[
C = \sum_{j=1}^{F} n_j
\]
as the total number of polygon corners. Each polygon has interior angle sum $(n_j - 2)\pi$, so the total sum of all polygon angles is
\[
\sum_{j=1}^F (n_j - 2)\pi = (C - 2F)\pi.
\]

\emph{Step 2: Angles around singularities.} \\
After gluing, the polygon angles assemble around the singularities. By hypothesis, the total angle at $p_i$ is $\theta_i = (k_i + 1)2\pi$, so the sum of all angles around vertices is
\[
\sum_{i=1}^s \theta_i = \sum_{i=1}^s (k_i + 1) 2\pi = 2\pi \left( \sum_{i=1}^s k_i + s \right).
\]
Equating this with the total polygon angle sum gives
\[
2\pi \left( \sum_{i=1}^s k_i + s \right) = (C - 2F)\pi \quad \implies \quad 2 \left( \sum_{i=1}^s k_i + s \right) = C - 2F.
\]

\emph{Step 3: Relation between corners and edges.} \\
Each edge contributes two polygon corners, so $C = 2E$. Substituting this yields
\[
2 \left( \sum_{i=1}^s k_i + s \right) = 2E - 2F \quad \implies \quad \sum_{i=1}^s k_i + s = E - F \quad \implies \quad \sum_{i=1}^s k_i = E - F - s.
\]

\emph{Step 4: Euler characteristic.} \\
The Euler characteristic of $S$ is
\[
\chi(S) = V - E + F = s - E + F,
\]
which can be rearranged to give $E - F - s = -\chi(S)$. Hence
\[
\sum_{i=1}^s k_i = -\chi(S).
\]

\emph{Step 5: Expressing in terms of genus.} \\
For a closed orientable surface of genus $g$, $\chi(S) = 2 - 2g$. Substituting, we obtain
\[
\sum_{i=1}^s k_i = -(2 - 2g) = 2g - 2,
\]
which completes the proof.
\end{proof}

\begin{remark}
Proposition~\ref{Gauss-Bonnet} can be viewed as an application of the
Gauss--Bonnet theorem for surfaces with conical singularities.
Although a translation surface is flat away from its singularities,
the curvature is concentrated at the cone points. The curvature
contribution at $p_i$ is the angle defect
\[
2\pi-\theta_i=-2\pi k_i,
\]
where $\theta_i=2\pi(k_i+1)$. Thus, the Gauss--Bonnet theorem gives
\[
-2\pi\sum_{i=1}^{s}k_i
=2\pi\chi(S)=2\pi(2-2g),
\]
which yields $\sum_{i=1}^{s}k_i=2g-2$. This is analogous to the application of the Gauss--Bonnet theorem to hyperbolic triangles in Section~4.4.
\end{remark}

\section{Saddle connections}
\begin{definition}
    A saddle connection on a translation surface is a geodesic segment connecting two singular points (which may coincide) that contains no singular point in its interior.
\end{definition}
\begin{definition}
A shortest saddle connection is a minimum length saddle connection among all saddle connections on the translation surface.
\end{definition}
In this thesis, we refer shortest saddle connections as \emph{systolic connections}.
\begin{example}
In the translation surface described in Example~\ref{example 1}, $e_1, e_2, e_3$, and $e_4$ are both saddle connections and systolic connections.
\end{example}
\begin{note}
   When a translation surface has no singular points, a marked point replaces a singular point in these definitions.
\end{note}

For a given translation surface $S$, the length of its systolic connection is denoted by $\mathrm{Sys}(S).$
The following results have been developed concerning the maximum length and number of systolic connections on a given translation surface.

\begin{theorem} (\cite{Bo21}, Theorem 3.3)
    Let S be a genus $g\geq1$ translation surface of area one and $r\geq1$ singular points or marked points. Then 
$$\mathrm{Sys}(S)\le \left[\frac{\sqrt{3}}{2}\left(2g-2+r\right)\right]^{-\frac{1}{2}}$$.

This equality is obtained iff S is built with equilateral triangles with side saddle connections of length $\mathrm{Sys}(S)$.
\end{theorem}

\begin{proposition}\label{prop:maxsaddleconnections}
(Proposition~5.1~\cite{Bo21})
The greatest number of shortest saddle connections of a unit-area surface
in the stratum $\mathcal{H}(k_1,\dots,k_r)$ is
\[
\sum_{i=1}^r 3(k_i+1),
\]
and this number is realized if and only if the surface is a global
maximum for the function $\mathrm{Sys}$.
\end{proposition}

\chapter[Systolic Embedding of graphs on translation surfaces]{Embedding of graphs on translation surfaces}\label{chap3}
In this chapter, we investigate the embedding of graphs on translation surfaces, with a particular emphasis on systolic embeddings. We begin by defining systolic cellular and essential embedding of graphs on translation surfaces. In the context of essential-systolic embedding, we prove Theorem~\ref{Main Result}, which states that any finite graph admits an essential-systolic embedding on a translation surface. Next, we compute an upper bound for the minimum genus of translation surfaces on which $G$ admits essential-systolic embedding, stated in Theorem~\ref{upper bound}. After that, we investigate that the wedge of $n$ circles, $\Sigma_n$ for every $n \ge 2$, admits cellular-systolic embedding on some translation surface and also computed the maximum and minimum genus of such translation surfaces in Theorem~\ref{cellular embedding}. In addition to this, we identify two other collections of graphs that can be embedded on translation surfaces in a similar way in Theorem~\ref{another rich collection}.

\section{Systolic Embedding Of Graphs}
First, we introduce the systolic graph of a translation surface. The systolic graph to a translation surface $X$, denoted by $\Gamma_X=(V, E)$, is a graph, where the set $V$ of vertices consists of all singular points (or marked points) of $X$, and the set $E$ of edges consists of all the shortest saddle connections of $X$. Now, we give the definition of systolic embedding of finite graphs below.
\begin{definition}
Let $X$ be a translation surface and $G=(V, E)$ be a graph.
\begin{enumerate}
\item An embedding of $G$ in $X$ is a function $f: G\to X$ where
\begin{enumerate}
\item For every $v\in V$, the image $f(v)$ is a singular point of $X$, 
\item For every edge $e\in E$, its image $f(e)$ is a saddle connection of $X$ and 
\item For $e_1\neq e_2\in E$, $f(e_1)$ and $f(e_2)$ can intersect at their endpoints only.
\end{enumerate}

\item We say that the graph $G$ admits a \emph{cellular embedding} on $X$ if there exists an embedding $f: G\to X$ such that 
\begin{enumerate}
\item $f(V)$ is the set of all singular points of $X$ and 
\item Every component of $X\setminus f(G)$ is a topological disk.
\end{enumerate}

\item An embedding $f: G\to X$ is called an \emph{essential embedding} of $G$ on $X$ if none of the components of $X\setminus f(G)$ is a topological disk.
\item  We say that the graph $G$ admits a \emph{systolic embedding} on $X$ if $G$ is isomorphic to $\Gamma_X$.
\end{enumerate}
\end{definition}

\begin{remark}
    If $f: G\to X$ is an embedding, then $G$ and $f(G)$ are two isomorphic graphs.  We shall not distinguish between $G$ and $f(G)$ whenever the context is clear.
\end{remark}
    
\begin{example}
  The graph formed by the wedge of two circles admits a cellular-systolic embedding on a translation surface of genus one (see Figure~\ref{example: embedding}).
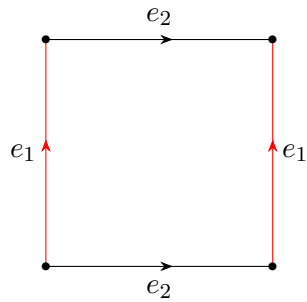
\begin{figure}[htbp]
\begin{center}
\begin{tikzpicture}[scale=3]
\draw (0,0)--(1,0);
\draw[red] (1,0)--(1,1);
\draw (1,1)--(0,1);
\draw[red] (0,1)--(0,0);

\draw (0.5, 1.1) node{$e_2$};
\draw (0.5,-0.1) node {$e_2$};
\draw (-0.1,0.5) node{$e_1$};
\draw (1.1,0.5) node {$e_1$};

\draw [-{Stealth}]
(0.55,0)--(0.56,0);
\draw [-{Stealth}]
(0.55,1)--(0.56,1);
\draw [-{Stealth[red]}]
(0,0.55)--(0,0.56);
\draw [-{Stealth[red]}]
(1,0.55)--(1,0.56);
\draw (0,0) node {\tiny $\bullet$};
\draw (1,0) node {\tiny $\bullet$};
\draw (1,1) node {\tiny $\bullet$};
\draw (0,1) node {\tiny $\bullet$};
\end{tikzpicture}
\end{center}
\caption{Translation surface for wedge of $2$ circles }
\label{example: embedding}
\end{figure}
\end{example}

\begin{example}
    A loop admits an essential-systolic embedding on a translation surface of genus one illustrated in Figure~\ref{example: embedding 2}.
\end{example}

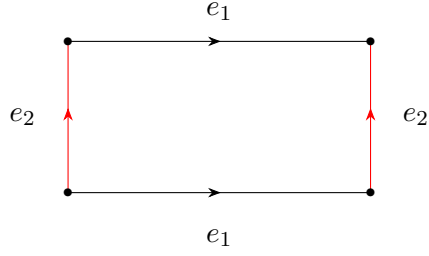
\begin{figure}[H]
\begin{center}
\begin{tikzpicture}[scale=2]
\draw (0,0)--(2,0);
\draw[red] (2,0)--(2,1);
\draw (2,1)--(0,1);
\draw[red] (0,1)--(0,0);

\draw (1, 1.2) node{$e_1$};
\draw (1,-0.3) node {$e_1$};
\draw (-0.3,0.5) node{$e_2$};
\draw (2.3,0.5) node {$e_2$};

\draw [-{Stealth}]
(1,0)--(1.01,0);
\draw [-{Stealth}]
(1,1)--(1.01,1);
\draw [-{Stealth[red]}]
(0,0.55)--(0,0.56);
\draw [-{Stealth[red]}]
(2,0.55)--(2,0.56);

\draw (0,0) node {\tiny $\bullet$};
\draw (2,0) node {\tiny $\bullet$};
\draw (2,1) node {\tiny $\bullet$};
\draw (0,1) node {\tiny $\bullet$};

\end{tikzpicture}
\end{center}
\caption{Translation surface for a loop}
\label{example: embedding 2}
\end{figure}

\section{Essential-Systolic Embedding}
In this section, our goal is to prove Theorem~\ref{Main Result}. In particular, we prove that every graph admits an essential-systolic embedding on a translation surface. We begin with the proofs of Theorem~\ref{main theorem 1} and Proposition~\ref{proposition 1}, which are essential for the proof of Theorem~\ref{Main Result}.

\begin{repthm}[\ref{main theorem 1}]({\cite{MR4985436}})
    Every complete graph $K_n$, with $n$ odd and $n \geq 5$, admits an essential-systolic embedding on a translation surface $S_{g_n}$, where the genus $g_n$ is given by $$ g_n = 1 - n \left( 1 - \frac{n-1}{2} \right).$$
\end{repthm}

\begin{proof}
We prove the theorem by explicitly constructing a translation surface, of genus $g_n$ given in the above equation, on which the graph $K_n$ admits an essential-systolic embedding. We denote the vertices of $K_n$ by $v_0, v_1, \dots, v_{n-1}$.

To construct the desired translation surface, we begin by taking $n$ many regular Euclidean $(n-1)$-gons $R_0$, $R_1$, $\dots$, $R_{n-1}$, each with fixed area $a$ and side length $s$. Next, we introduce $n(n-1)$ rectangles $Q_{ij}$ (see Figure~\ref{Qij}), with sides labeled $a_{ij}, b_{ij}, a_{ij}^\prime, b_{ij}^\prime$, satisfying the condition:
\begin{equation}\label{eq:1}
    \ell(a_{ij}) = \ell(a_{ij}^\prime) = s \quad \text{and} \quad \ell(b_{ij}) = \ell(b_{ij}^\prime) < \frac{s}{2}, \quad \text{for } 0 \leq i \leq n-1,\ 1 \leq j \leq n-1.
\end{equation}

\begin{figure}[htbp]
\begin{center}
\begin{tikzpicture}[scale=2]
\draw (0,0)--(2,0);
\draw[red] (2,0)--(2,1);
\draw (2,1)--(0,1);
\draw[red] (0,1)--(0,0);

\draw (1, 1.2) node{$a_{ij}$};
\draw (1,-0.3) node {$a^\prime_{ij}$};
\draw (-0.3,0.5) node{$b_{ij}$};
\draw (2.3,0.5) node {$b^\prime_{ij}$};
\end{tikzpicture}
\end{center}
\caption{$Q_{ij}$ }
\label{Qij}
\end{figure}

We attach each rectangle $Q_{ij}$ to the polygon $R_i$ along the side $a_{ij}^\prime$ by a translation. Let $P_0, P_1, \dots, P_{n-1},$ be the resulting figures composed of $R_i$ and the corresponding rectangles. In each $P_i$, we identify the sides $b_{ij}$ with $b_{ij}^\prime$ by translation.

To complete the construction of the translation surface, we now glue the remaining sides $a_{ij}$ as described below:\\
We obtain ${\frac{n-1}{2}}$ edge-disjoint Hamiltonian cycles $C_l$ in $K_n$ using Walecki's construction (see~\cite{MR4089838} Lemma 18):
$$C_l= \left(v_0, v_l, v_{l+1}, v_{l-1}, v_{l+2},\ v_{l-2},\dots,  v_{l+\frac{n-1}{2}}\right), \text{ for }1 \leq l \leq \frac{n-1}{2},$$ where subscripts of $v$'s are taken modulo $n-1$. Now, consider the following permutations in $S_n$ corresponding to the above cycles: $$\sigma_l^\prime = \left(0,l, l+1, l-1, l+2, l-2, \dots, l+\frac{n-1}{2} \right) ~mod~(n-1), \text{ for }1 \leq l \leq \frac{n-1}{2}.$$
 For each  $P_i$, re-label the sides $a_{ij}$ as $E_1(i), E_2(i), \dots, E_{\frac{n-1}{2}}(i),\bar{E}_1(i), \bar{E}_2(i), \dots, \bar{E}_{\frac{n-1}{2}}(i)$ in a fixed order such that each \( E_r(i) \) is parallel to \( \bar{E}_r(i) \). Moreover, for any \( 0 \leq i, i' \leq n-1 \) and \( 1 \leq r \leq \frac{n-1}{2} \), we have $E_r(i)$ is parallel to $\bar{E}_r(i')$. 

We glue the sides $E_r(i)$ with $\bar{E}_r(\sigma_r^\prime(i))$, for $i = 0, 1, \dots, n-1$ and $r = 1, 2, \dots, \frac{n-1}{2}$. This process yields a translation surface $S_{g_n}$, with $n$ singular points, each having total angle $2\pi(n - 2)$. If $\sigma_r^\prime(i) = i'$, then the side $E_r(i) = a_{ij}$ is glued to $\bar{E}_r(i') = a_{i'j'}$ for some $j, j' \in \{1, 2, \dots, n-1\}$. Consequently, when the sides $a_{ij}$ and $a_{i'j'}$ are identified by translation, the edges $b_{ij}$ and $b_{i'j'}'$ align and join end-to-end, forming a single straight segment on the surface. We call this joined segment the \emph{concatenation} of 
$b_{ij}$ and $b_{i'j'}'$, denoted by $b_{ij} * b_{i'j'}'$. 
This concatenated segment $b_{ij} * b_{i'j'}'$ forms a saddle connection on $S_{g_n}$ joining the $i$-th singular point to the $i'$-th singular point. From condition~\eqref{eq:1}, the length of each such saddle connection is less than $s$, making it the systolic connection. For instance, in the case of complete graph $K_5$ (Figure~\ref{figure18}), Walecki's
construction decomposes its edges into two edge-disjoint Hamiltonian
cycles,
\[
C_1=(v_0,v_1,v_2,v_4,v_3)
\quad\text{and}\quad
C_2=(v_0,v_2,v_3,v_1,v_4).
\]
These cycles determine the permutations $\sigma_1'=(0\ 1\ 2\ 4\ 3)$
and $\sigma_2'=(0\ 2\ 3\ 1\ 4)$, respectively, which specify the side
identifications in the construction of the corresponding translation
surface shown in Figure~\ref{figure19}.

Therefore, all these $n(n-1)$ saddle connections are systolic connections, and each singular point connects to $(n-1)$ others. The systolic graph $\Gamma_{S_{g_n}}$ is a $(n-1)$-regular graph with $n$ vertices $V_0, V_1,\dots, V_{n-1}$. Now, we map $V_i$ to $v_i$ for each $0\le i\le n-1$ and hence the above construction ensures that $\Gamma_{S_{g_n}}\cong K_n$. Since $\Gamma_{S_{g_n}} \cong K_n$, and the components of $S_{g_n} \setminus K_n$ are not open topological disks, it follows that $K_n$ admits an essential-systolic embedding on $S_{g_n}$.

In this construction, the translation surface \(S_{g_n}\) is composed of \(n\) polygonal regions \(R_i\) together with the attached rectangles \(Q_{ij}\). Each \(R_i\) contributes one vertex and one face to the resulting cell decomposition. Again, by the construction, the rectangles \(Q_{ij}\) collectively contribute an additional \(\frac{n(n-1)}{2}\) vertices and \(n(n-1)\) faces to the resulting cell decomposition. The pairwise gluing of sides along the identifications described above contributes both connecting and internal edges, producing a total of 
\(\frac{n(n-1)}{2} + 2n(n-1)\) edges in the decomposition. Consequently,

\[
V = n + \frac{n(n-1)}{2}, \quad 
E = \frac{n(n-1)}{2} + 2n(n-1), \quad 
F = n + n(n-1).
\]
Therefore, the Euler characteristic of \(S_{g_n}\) is
\[
\chi(S_{g_n}) = V - E + F = 2n - n(n-1),
\]
and since \(\chi(S_{g_n}) = 2 - 2g_n\), it follows that
\[
g_n = \left[1 - \left(1 - \frac{n-1}{2}\right)n\right].
\]

\end{proof}

\begin{figure}[htbp]
\begin{center}
\begin{tikzpicture}[scale=2]
\draw (-0.587,-0.809)--(0.587,-0.809);
\draw(0.587,-0.809)--(0.951,0.309);
\draw (0.951,0.309)--(0,1);
\draw (0,1)--(-0.951,0.309);
\draw (-0.951,0.309)--(-0.587,-0.809);

\draw[red] (-0.587,-0.809) node{\tiny$\bullet$};
\draw (-0.75,-0.809) node{$v_0$};
\draw[blue] (0.587,-0.809) node{\tiny$\bullet$};
\draw (0.75,-0.809) node{$v_1$};
\draw[green] (0.951,0.309) node{\tiny$\bullet$};
\draw[] (1.1,0.309) node{$v_2$};
\draw[yellow] (0,1) node{\tiny$\bullet$};
\draw[] (0,1.1) node{$v_3$};
\draw (-0.951,0.309) node{\tiny$\bullet$};
\draw (-1.1,0.309) node{$v_4$};

\draw (-0.587,-0.809)--(0.951,0.309);
\draw (-0.587,-0.809)--(0,1);
\draw (0.951,0.309)--(-0.951,0.309)--(0.587,-0.809)--(0,1);

\end{tikzpicture}
\caption{$K_5$}
\label{figure18}
\end{center}
\end{figure}

\begin{figure}[htbp]
\begin{center}
\begin{tikzpicture}[scale=1.5]
% for $P_0$
\draw (0,0)--(1,0)--(1,1)--(0,1)--(0,0);
\draw (0,0)--(0,-0.45);
\draw[red] (0,-0.45)--(1,-0.45);
\draw (1,-0.45)--(1,0);
\draw (1,0)--(1.45,0);
\draw[lime] (1.45,0)--(1.45,1);
\draw (1.45,1)--(1,1);
\draw (0,1)--(0,1.45);
\draw[orange] (0,1.45)--(1,1.45);
\draw (1,1.45)--(1,1);
\draw (0,0)--(-0.45,0);
\draw[teal] (-0.45,0)--(-0.45,1);
\draw (-0.45,1)--(0,1);

\draw (-0.2,1.4) node{\tiny$b_{01}$};
\draw (1.2,1.4) node {\tiny$b^\prime_{01}$};
\draw (-0.3,-0.1) node{\tiny$b_{02}$};
\draw (-0.3,1.15) node{\tiny$b_{02}^\prime$};
\draw (1.2,1.15) node{\tiny$b_{04}$};
\draw (1.3,-0.1) node{\tiny$b_{04^\prime}$};
\draw (-0.2,-.4) node{\tiny$b_{03}^\prime$};
\draw (1.2,-.4) node{\tiny$b_{03}$};

\draw (0.5,1.68) node{\tiny$a_{01}$ $(\bar{E_1}(0))$};
\draw (-0.65,0.5) node{\tiny$a_{02}$};
\draw (-1.1,0.5) node {\tiny$(\bar{E_2}(0))$};
\draw (0.5,-0.7) node{\tiny$a_{03}  (E_1(0))$};
\draw (2,0.5) node{\tiny$a_{04} (E_2(0))$};
\draw (0.5,0.5)node{$P_0$};
\draw[red] (0,0) node{\tiny$\bullet$};
\draw[red] (1,0) node{\tiny$\bullet$};
\draw[red] (1,1) node{\tiny$\bullet$};
\draw[red] (0,1) node{\tiny$\bullet$};

% for $P_1$
\draw (5,0)--(6,0)--(6,1)--(5,1)--(5,0);
\draw (5,0)--(5,-0.45);
\draw[cyan] (5,-0.45)--(6,-0.45);
\draw (6,-0.45)--(6,0);
\draw (6,0)--(6.45,0);
\draw[magenta] (6.45,0)--(6.45,1);
\draw (6.45,1)--(6,1);
\draw (5,1)--(5,1.45);
\draw[red] (5,1.45)--(6,1.45);
\draw (6,1.45)--(6,1);
\draw (4.8,1.4) node{\tiny$b_{11}$};
\draw (6.2,1.4) node {\tiny$b^\prime_{11}$};
\draw (5,0)--(4.55,0);
\draw[yellow] (4.55,0)--(4.55,1);
\draw(4.55,1)--(5,1);
\draw (4.7,-0.1) node{\tiny$b_{12}$};
\draw (4.7,1.15) node{\tiny$b_{12}^\prime$};
\draw (6.2,1.15) node{\tiny$b_{14}$};
\draw (6.3,-0.1) node{\tiny$b_{14^\prime}$};
\draw (4.8,-.4) node{\tiny$b_{13}^\prime$};
\draw (6.2,-.4) node{\tiny$b_{13}$};

\draw (5.5,1.7) node{\tiny$a_{11} (\bar{E_1}(1))$};
\draw (4,0.5) node{\tiny$(\bar{E_2}(2)) a_{12}$};

\draw (5.5,-0.7) node{\tiny$a_{13} (E_1(1))$};
\draw (6.65,0.6) node{\tiny$a_{14}$};
\draw (6.8,0.35) node{\tiny$ (E_2(1))$};
\draw (5.5,0.5)node{$P_1$};
\draw[blue] (5,0) node{\tiny$\bullet$};
\draw[blue] (6,0) node{\tiny$\bullet$};
\draw[blue] (6,1) node{\tiny$\bullet$};
\draw[blue] (5,1) node{\tiny$\bullet$};

% for $P_2$
\draw (0,-4)--(1,-4)--(1,-3)--(0,-3)--(0,-4);
\draw (0,-4)--(0,-4.45);
\draw[blue] (0,-4.45)--(1,-4.45);
\draw(1,-4.45)--(1,-4);
\draw (1,-4)--(1.45,-4);
\draw[teal] (1.45,-4)--(1.45,-3);
\draw (1.45,-3)--(1,-3);
\draw (0,-3)--(0,-2.55);
\draw[cyan] (0,-2.55)--(1,-2.55);
\draw(1,-2.55)--(1,-3);
\draw (0,-4)--(-0.45,-4);
\draw[violet] (-0.45,-4)--(-0.45,-3);
\draw[] (-0.45,-3)--(0,-3);
\draw (-0.2,-2.6) node{\tiny$b_{21}$};
\draw (1.2,-2.6) node {\tiny$b^\prime_{21}$};
\draw (-0.3,-4.1) node{\tiny$b_{22}$};
\draw (-0.3,-2.85) node{\tiny$b_{22}^\prime$};
\draw (1.2,-2.85) node{\tiny$b_{24}$};
\draw (1.3,-4.1) node{\tiny$b_{24^\prime}$};
\draw (-0.2,-4.4) node{\tiny$b_{23}^\prime$};
\draw (1.2,-4.4) node{\tiny$b_{23}$};

\draw (0.5,-2.3) node{\tiny$a_{21} (\bar{E_1}(2))$};
\draw (-1,-3.5) node{\tiny$(\bar{E_2}(2)) a_{22} $};
\draw (0.5,-4.7) node{\tiny$a_{23} (E_1(2))$};
\draw (2,-3.5) node{\tiny$a_{24}(E_2(2))$};
\draw (0.5,-3.5)node{$P_2$};
\draw[green] (0,-4) node{\tiny$\bullet$};
\draw[green] (1,-4) node{\tiny$\bullet$};
\draw[green] (1,-3) node{\tiny$\bullet$};
\draw[green] (0,-3) node{\tiny$\bullet$};

%for $P_3$
\draw (5,-4)--(6,-4)--(6,-3)--(5,-3)--(5,-4);
\draw (5,-4)--(5,-4.45);
\draw[orange] (5,-4.45)--(6,-4.45);
\draw(6,-4.45)--(6,-4);
\draw (6,-4)--(6.45,-4);
\draw[violet] (6.45,-4)--(6.45,-3);
\draw (6.45,-3)--(6,-3);
\draw (5,-3)--(5,-2.55);
\draw[green] (5,-2.55)--(6,-2.55);
\draw(6,-2.55)--(6,-3);
\draw (5,-4)--(4.55,-4);
\draw[magenta] (4.55,-4)--(4.55,-3);
\draw(4.55,-3)--(5,-3);
\draw (4.8,-2.6) node{\tiny$b_{31}$};
\draw (6.2,-2.6) node {\tiny$b^\prime_{31}$};
\draw (4.7,-4.1) node{\tiny$b_{32}$};
\draw (4.7,-2.85) node{\tiny$b_{32}^\prime$};
\draw (6.2,-2.85) node{\tiny$b_{34}$};
\draw (6.3,-4.1) node{\tiny$b_{34^\prime}$};
\draw (4.8,-4.4) node{\tiny$b_{33}^\prime$};
\draw (6.2,-4.4) node{\tiny$b_{33}$};

\draw (5.5,-2.3) node{\tiny$a_{31} (\bar{E_1}(3))$};
\draw (4.3,-3.35) node{\tiny$a_{32}$};
\draw (3.8,-3.35) node{\tiny$(\bar{E_2}(3))$};
\draw (5.5,-4.7) node{\tiny$a_{33}$ $ (E_1(3))$};
\draw (6.7,-3.5) node{\tiny$a_{34}$};
\draw (6.8,-3.7) node{\tiny$(E_2(3))$};

\draw (5.5,-3.5)node{$P_3$};
\draw[yellow] (5,-4) node{\tiny$\bullet$};
\draw[yellow] (6,-4) node{\tiny$\bullet$};
\draw[yellow] (6,-3) node{\tiny$\bullet$};
\draw[yellow] (5,-3) node{\tiny$\bullet$};

% for $P_4$
\draw (0,-8)--(1,-8)--(1,-7)--(0,-7)--(0,-8);
\draw (0,-8)--(0,-8.45);
\draw[green] (0,-8.45)--(1,-8.45);
\draw (1,-8.45)--(1,-8);
\draw (1,-8)--(1.45,-8);
\draw[yellow] (1.45,-8)--(1.45,-7);
\draw (1.45,-7)--(1,-7);
\draw (0,-7)--(0,-6.55);
\draw[blue] (0,-6.55)--(1,-6.55);
\draw (1,-6.55)--(1,-7);
\draw (0,-8)--(-0.45,-8);
\draw[lime] (-0.45,-8)--(-0.45,-7);
\draw (-0.45,-7)--(0,-7);
\draw (-0.2,-6.6) node{\tiny$b_{41}$};
\draw (1.2,-6.6) node {\tiny$b^\prime_{41}$};
\draw (-0.3,-8.1) node{\tiny$b_{42}$};
\draw (-0.3,-6.85) node{\tiny$b_{42}^\prime$};
\draw (1.2,-6.85) node{\tiny$b_{44}$};
\draw (1.3,-8.1) node{\tiny$b_{44^\prime}$};
\draw (-0.2,-8.4) node{\tiny$b_{43}^\prime$};
\draw (1.2,-8.4) node{\tiny$b_{43}$};

\draw (0.5,-6.3) node{\tiny$a_{41}$ $(\bar{E_1}(4))$};
\draw (-0.7,-7.5) node{\tiny$a_{42}$};
\draw (-1.2,-7.5) node {\tiny$(\bar{E_2}(4))$};
\draw (0.5,-8.65) node{\tiny$a_{43}$ $ (E_1(4))$};
\draw (2.1,-7.5) node{\tiny$a_{44}$ $(E_2(4))$};
\draw (0.5,-7.5)node{$P_4$};
\draw[] (0,-8) node{\tiny$\bullet$};
\draw[] (1,-8) node{\tiny$\bullet$};
\draw[] (1,-7) node{\tiny$\bullet$};
\draw[] (0,-7) node{\tiny$\bullet$};

\draw[red] (5,-6) node{$\bullet$ $V_0$};
\draw[blue] (5,-6.4) node{$\bullet$ $V_1$};
\draw[green] (5,-6.8) node{$\bullet$ $V_2$};
\draw[yellow] (5,-7.2) node{$\bullet$ $V_3$};
\draw (5,-7.6) node{$\bullet$ $V_4$};
\end{tikzpicture}
\caption{Translation surface for $K_5$}
\label{figure19}
\end{center}
\end{figure}
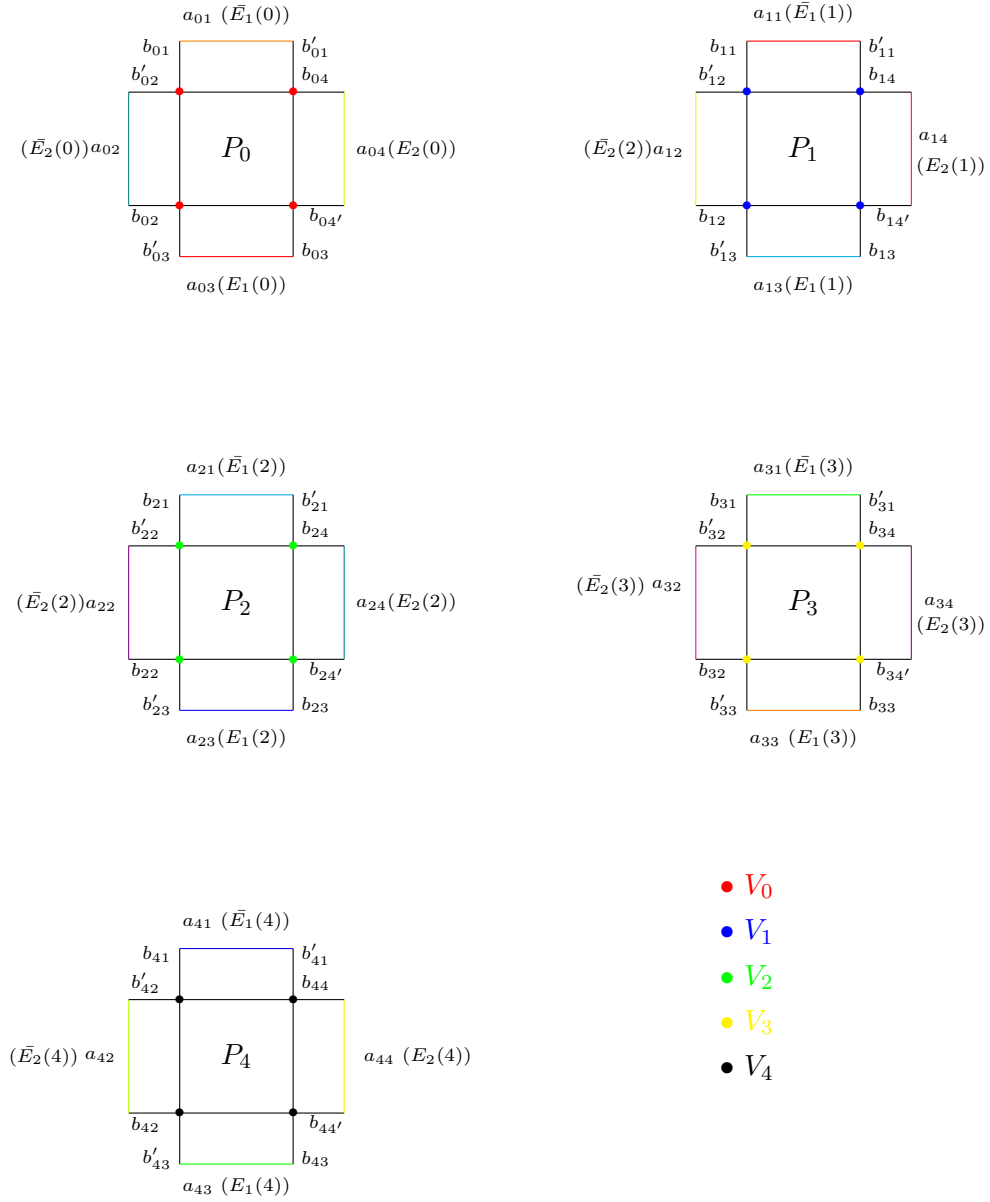

Next, we prove the proposition \ref{proposition 1}.

\begin{repprop}[\ref{proposition 1}](\cite{MR4985436})
  For odd $n$ with $n\geq 5$, every subgraph of $K_n$ admits an essential-systolic embedding on a translation surface of genus $g_n$.
\end{repprop}
\begin{proof}
Let $e$ be an edge of $K_n$. We show that the graph $K_n \setminus \{e\}$ admits an essential-systolic embedding on a translation surface of the same genus $g_n$.

Consider the embedding of $K_n$ on $S_{g_n}$ as given in Theorem~\ref{main theorem 1}. Assume that the edge $e$ is formed by the concatenation of $b_{ij}$ and $b_{ml}^\prime$, denoted by $b_{ij} * b_{ml}^\prime$, for some fixed indices $i, j, m, l$ (as described in Theorem~\ref{main theorem 1}), with $\ell(b_{ij}) = \ell(b_{ml}^\prime) = \ell_0.$ Now, we modify the rectangle $Q_{ij}$ in such a way that
$
\ell(b_{ij}) = \ell(b_{ij}^\prime) > \ell_0$. This ensures that the concatenation $b_{ij} * b_{ml}^\prime$ is no longer a systolic connection.

By performing this modification on $Q_{ij}$ and keeping the rest of the construction unchanged, we obtain a new translation surface $S_{g_n}^*$, of the genus $g_n$, on which $K_n \setminus \{e\}$ is essentially and systolically embedded.

Any subgraph $H$ of $K_n$ can be obtained by deleting finitely many edges from $K_n$, i.e., $$H = K_n \setminus \{e_1, e_2, \dots, e_k\},$$ where $e_i$'s are edges of $K_n$. Hence, by performing the above process finitely many times, it follows that any subgraph $H \subseteq K_n$ admits an essential-systolic embedding on a translation surface of the genus $g_n$.

\end{proof}
In Figure~\ref{figure21}, we show the construction for $K_5\setminus\{e_{v_0v_1}\}$ (Figure~\ref{figure20}). 
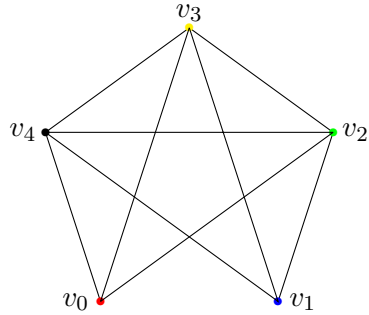
\begin{figure}[htbp]
\begin{center}
\begin{tikzpicture}[scale=2]

\draw(0.587,-0.809)--(0.951,0.309);
\draw (0.951,0.309)--(0,1);
\draw (0,1)--(-0.951,0.309);
\draw (-0.951,0.309)--(-0.587,-0.809);

\draw[red] (-0.587,-0.809) node{\tiny$\bullet$};
\draw (-0.75,-0.809) node{$v_0$};
\draw[blue] (0.587,-0.809) node{\tiny$\bullet$};
\draw (0.75,-0.809) node{$v_1$};
\draw[green] (0.951,0.309) node{\tiny$\bullet$};
\draw[] (1.1,0.309) node{$v_2$};
\draw[yellow] (0,1) node{\tiny$\bullet$};
\draw[] (0,1.1) node{$v_3$};
\draw (-0.951,0.309) node{\tiny$\bullet$};
\draw (-1.1,0.309) node{$v_4$};

\draw (-0.587,-0.809)--(0.951,0.309);
\draw (-0.587,-0.809)--(0,1);
\draw (0.951,0.309)--(-0.951,0.309)--(0.587,-0.809)--(0,1);

\end{tikzpicture}
\caption{$K_5\setminus \{e_{v_0 v_1}\}$}
\label{figure20}
\end{center}
\end{figure}

\begin{figure}[htbp]
\begin{center}
\begin{tikzpicture}[scale=1.5]
% for $P_0$
\draw (0,0)--(1,0)--(1,1)--(0,1)--(0,0);
\draw (0,0)--(0,-0.8);
\draw[red] (0,-0.8)--(1,-0.8);
\draw (1,-0.8)--(1,0);

\draw (1,0)--(1.45,0);
\draw[lime] (1.45,0)--(1.45,1);
\draw (1.45,1)--(1,1);
\draw (0,1)--(0,1.45);
\draw[orange] (0,1.45)--(1,1.45);
\draw (1,1.45)--(1,1);
\draw (0,0)--(-0.45,0);
\draw[teal] (-0.45,0)--(-0.45,1);
\draw (-0.45,1)--(0,1);

\draw (-0.2,1.4) node{\tiny$b_{01}$};
\draw (1.2,1.4) node {\tiny$b^\prime_{01}$};
\draw (-0.3,-0.1) node{\tiny$b_{02}$};
\draw (-0.3,1.15) node{\tiny$b_{02}^\prime$};
\draw (1.2,1.15) node{\tiny$b_{04}$};
\draw (1.3,-0.1) node{\tiny$b_{04^\prime}$};
\draw (-0.2,-0.6) node{\tiny$b_{03}^\prime$};
\draw (1.2,-0.6) node{\tiny$b_{03}$};

\draw (0.5,1.68) node{\tiny$a_{01}$ $(\bar{E_1}(0))$};
\draw (-0.65,0.5) node{\tiny$a_{02}$};
\draw (-1.1,0.5) node {\tiny$(\bar{E_2}(0))$};
\draw (0.5,-1) node{\tiny$a_{03}  (E_1(0))$};
\draw (2,0.5) node{\tiny$a_{04} (E_2(0))$};
\draw (0.5,0.5)node{$P_0$};
\draw[red] (0,0) node{\tiny$\bullet$};
\draw[red] (1,0) node{\tiny$\bullet$};
\draw[red] (1,1) node{\tiny$\bullet$};
\draw[red] (0,1) node{\tiny$\bullet$};

% for $P_1$
\draw (5,0)--(6,0)--(6,1)--(5,1)--(5,0);
\draw (5,0)--(5,-0.45);
\draw[cyan] (5,-0.45)--(6,-0.45);
\draw (6,-0.45)--(6,0);
\draw (6,0)--(6.45,0);
\draw[magenta] (6.45,0)--(6.45,1);
\draw (6.45,1)--(6,1);
\draw (5,1)--(5,1.45);
\draw[red] (5,1.45)--(6,1.45);
\draw (6,1.45)--(6,1);
\draw (4.8,1.4) node{\tiny$b_{11}$};
\draw (6.2,1.4) node {\tiny$b^\prime_{11}$};
\draw (5,0)--(4.55,0);
\draw[yellow] (4.55,0)--(4.55,1);
\draw(4.55,1)--(5,1);
\draw (4.7,-0.1) node{\tiny$b_{12}$};
\draw (4.7,1.15) node{\tiny$b_{12}^\prime$};
\draw (6.2,1.15) node{\tiny$b_{14}$};
\draw (6.3,-0.1) node{\tiny$b_{14^\prime}$};
\draw (4.8,-.4) node{\tiny$b_{13}^\prime$};
\draw (6.2,-.4) node{\tiny$b_{13}$};

\draw (5.5,1.7) node{\tiny$a_{11} (\bar{E_1}(1))$};
\draw (4,0.5) node{\tiny$(\bar{E_2}(2)) a_{12}$};

\draw (5.5,-0.8) node{\tiny$a_{13} (E_1(1))$};
\draw (6.65,0.6) node{\tiny$a_{14}$};
\draw (6.8,0.35) node{\tiny$ (E_2(1))$};
\draw (5.5,0.5)node{$P_1$};
\draw[blue] (5,0) node{\tiny$\bullet$};
\draw[blue] (6,0) node{\tiny$\bullet$};
\draw[blue] (6,1) node{\tiny$\bullet$};
\draw[blue] (5,1) node{\tiny$\bullet$};

% for $P_2$
\draw (0,-4)--(1,-4)--(1,-3)--(0,-3)--(0,-4);
\draw (0,-4)--(0,-4.45);
\draw[blue] (0,-4.45)--(1,-4.45);
\draw(1,-4.45)--(1,-4);
\draw (1,-4)--(1.45,-4);
\draw[teal] (1.45,-4)--(1.45,-3);
\draw (1.45,-3)--(1,-3);
\draw (0,-3)--(0,-2.55);
\draw[cyan] (0,-2.55)--(1,-2.55);
\draw(1,-2.55)--(1,-3);
\draw (0,-4)--(-0.45,-4);
\draw[violet] (-0.45,-4)--(-0.45,-3);
\draw[] (-0.45,-3)--(0,-3);

\draw (-0.2,-2.6) node{\tiny$b_{21}$};
\draw (1.2,-2.6) node {\tiny$b^\prime_{21}$};
\draw (-0.3,-4.1) node{\tiny$b_{22}$};
\draw (-0.3,-2.85) node{\tiny$b_{22}^\prime$};
\draw (1.2,-2.85) node{\tiny$b_{24}$};
\draw (1.3,-4.1) node{\tiny$b_{24^\prime}$};
\draw (-0.2,-4.4) node{\tiny$b_{23}^\prime$};
\draw (1.2,-4.4) node{\tiny$b_{23}$};

\draw (0.5,-2.3) node{\tiny$a_{21} (\bar{E_1}(2))$};
\draw (-1,-3.5) node{\tiny$(\bar{E_2}(2)) a_{22} $};
\draw (0.5,-4.6) node{\tiny$a_{23} (E_1(2))$};
\draw (2,-3.5) node{\tiny$a_{24}(E_2(2))$};
\draw (0.5,-3.5)node{$P_2$};
\draw[green] (0,-4) node{\tiny$\bullet$};
\draw[green] (1,-4) node{\tiny$\bullet$};
\draw[green] (1,-3) node{\tiny$\bullet$};
\draw[green] (0,-3) node{\tiny$\bullet$};

%for $P_3$
\draw (5,-4)--(6,-4)--(6,-3)--(5,-3)--(5,-4);
\draw (5,-4)--(5,-4.45);
\draw[orange] (5,-4.45)--(6,-4.45);
\draw(6,-4.45)--(6,-4);
\draw (6,-4)--(6.45,-4);
\draw[violet] (6.45,-4)--(6.45,-3);
\draw (6.45,-3)--(6,-3);
\draw (5,-3)--(5,-2.55);
\draw[green] (5,-2.55)--(6,-2.55);
\draw(6,-2.55)--(6,-3);
\draw (5,-4)--(4.55,-4);
\draw[magenta] (4.55,-4)--(4.55,-3);
\draw(4.55,-3)--(5,-3);

\draw (4.8,-2.6) node{\tiny$b_{31}$};
\draw (6.2,-2.6) node {\tiny$b^\prime_{31}$};
\draw (4.7,-4.1) node{\tiny$b_{32}$};
\draw (4.7,-2.85) node{\tiny$b_{32}^\prime$};
\draw (6.2,-2.85) node{\tiny$b_{34}$};
\draw (6.3,-4.1) node{\tiny$b_{34^\prime}$};
\draw (4.8,-4.4) node{\tiny$b_{33}^\prime$};
\draw (6.2,-4.4) node{\tiny$b_{33}$};

\draw (5.5,-2.3) node{\tiny$a_{31} (\bar{E_1}(3))$};
\draw (4.3,-3.35) node{\tiny$a_{32}$};
\draw (3.8,-3.35) node{\tiny$(\bar{E_2}(3))$};
\draw (5.5,-4.6) node{\tiny$a_{33}$ $ (E_1(3))$};
\draw (6.7,-3.5) node{\tiny$a_{34}$};
\draw (6.8,-3.7) node{\tiny$(E_2(3))$};

\draw (5.5,-3.5)node{$P_3$};
\draw[yellow] (5,-4) node{\tiny$\bullet$};
\draw[yellow] (6,-4) node{\tiny$\bullet$};
\draw[yellow] (6,-3) node{\tiny$\bullet$};
\draw[yellow] (5,-3) node{\tiny$\bullet$};

% for $P_4$
\draw (0,-8)--(1,-8)--(1,-7)--(0,-7)--(0,-8);
\draw (0,-8)--(0,-8.45);
\draw[green] (0,-8.45)--(1,-8.45);
\draw (1,-8.45)--(1,-8);
\draw (1,-8)--(1.45,-8);
\draw[yellow] (1.45,-8)--(1.45,-7);
\draw (1.45,-7)--(1,-7);
\draw (0,-7)--(0,-6.55);
\draw[blue] (0,-6.55)--(1,-6.55);
\draw (1,-6.55)--(1,-7);
\draw (0,-8)--(-0.45,-8);
\draw[lime] (-0.45,-8)--(-0.45,-7);
\draw (-0.45,-7)--(0,-7);

\draw (-0.2,-6.6) node{\tiny$b_{41}$};
\draw (1.2,-6.6) node {\tiny$b^\prime_{41}$};
\draw (-0.3,-8.1) node{\tiny$b_{42}$};
\draw (-0.3,-6.85) node{\tiny$b_{42}^\prime$};
\draw (1.2,-6.85) node{\tiny$b_{44}$};
\draw (1.3,-8.1) node{\tiny$b_{44^\prime}$};
\draw (-0.2,-8.4) node{\tiny$b_{43}^\prime$};
\draw (1.2,-8.4) node{\tiny$b_{43}$};

\draw (0.5,-6.3) node{\tiny$a_{41}$ $(\bar{E_1}(4))$};
\draw (-0.7,-7.5) node{\tiny$a_{42}$};
\draw (-1.2,-7.5) node {\tiny$(\bar{E_2}(4))$};
\draw (0.5,-8.65) node{\tiny$a_{43}$ $ (E_1(4))$};
\draw (2.1,-7.5) node{\tiny$a_{44}$ $(E_2(4))$};
\draw (0.5,-7.5)node{$P_4$};
\draw[] (0,-8) node{\tiny$\bullet$};
\draw[] (1,-8) node{\tiny$\bullet$};
\draw[] (1,-7) node{\tiny$\bullet$};
\draw[] (0,-7) node{\tiny$\bullet$};

\draw[red] (5,-6) node{$\bullet$ $V_0$};
\draw[blue] (5,-6.4) node{$\bullet$ $V_1$};
\draw[green] (5,-6.8) node{$\bullet$ $V_2$};
\draw[yellow] (5,-7.2) node{$\bullet$ $V_3$};
\draw (5,-7.6) node{$\bullet$ $V_4$};
\end{tikzpicture}
\caption{Translation surface for $K_5\setminus\{e_{v_0v_1}\}$}
\label{figure21}
\end{center}
\end{figure}
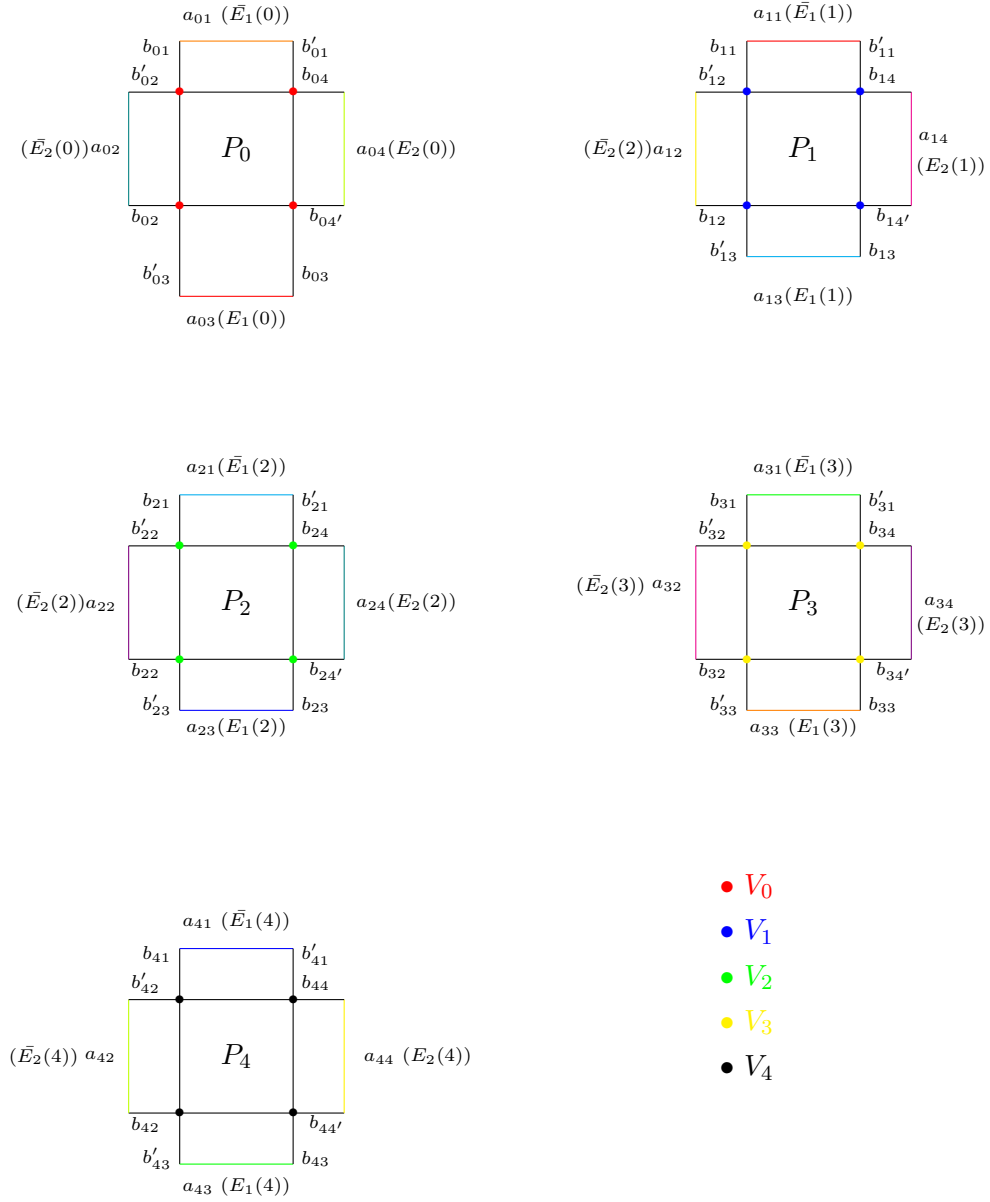

\begin{repthm}[\ref{Main Result}](\cite{MR4985436})
  Every finite graph admits an essential-systolic embedding on a translation surface.  
\end{repthm}
\begin{proof}
Any graph $G$ with $m$ vertices is a subgraph of the complete graph $K_n$, for an odd integer $n$. Now, the theorem follows from Proposition~\ref{proposition 1} and Theorem~\ref{main theorem 1}.
\end{proof}

 We define $P_G$, the set of all genera $g$ such that $G$ admits an essential-systolic embedding on a translation surface of the genus $g$, which is nonempty by Theorem \ref{Main Result}.
 
 \begin{lemma} \label{no maxima} (\cite{MR4985436})
If a graph $G$ admits an essential-systolic embedding on a translation surface of genus $g$, then there exists $g^\prime > g$ such that $G$ also admits such an embedding on a translation surface of genus $g^\prime$. 

 \end{lemma}

 \begin{proof}
   Let the graph \(G\) be a subgraph of a complete graph \(K_n\) for some odd \(n \ge 5\). Hence, \(G\) can also be viewed as a subgraph of the complete graphs \(K_{n+2}, K_{n+4}, K_{n+6}, \dots\). Since the sequence \(\{g_n, g_{n+2}, g_{n+4}, g_{n+6} \dots\}\) is strictly increasing and unbounded, by the Archimedean property of real numbers, there exists \(n_0 \in \{n, n+2, n+4, n+6 \dots\}\) such that \(g_{n_0} > g\). Setting \(g' = g_{n_0}\), we conclude from Proposition~\ref{proposition 1} that \(G\), viewed as a subgraph of \(K_{n_0}\), admits an essential-systolic embedding on a translation surface of the genus \(g' > g\).

\end{proof}

 Consequently, the notion of a maximum genus does not make sense, and therefore we focus on determining the minimum genus for such embeddings. Now, we define
$g^{\min}_G$ := $\min \left\{ g \mid g \in P_G \right\}$. We prove the following theorem that provides an upper bound for $g^{\min}_G$.

\begin{repthm}[\ref{upper bound}] (\cite{MR4985436})
For a graph \( G \) with \( n \) vertices, an upper bound on the minimum genus of a translation surface that admits such an embedding is given by
 \[
g^{\min}_G \leq 
\begin{cases}
1 - n\left(1 - \frac{n - 1}{2}\right), & \text{if } n \geq 5 \text{ and } n \text{ is odd}, \\[0.5em]
1 - (n + 1)\left(1 - \frac{n}{2}\right), & \text{if } n \geq 5 \text{ and } n \text{ is even and} \\[0.5em]
6, & \text{if } n < 5.
\end{cases}
\]

\end{repthm}

\begin{proof}

For $n \geq 5$, we consider the graph $G$ as a subgraph of the complete graph $K_n$ when $n$ is odd, or $K_{n+1}$ when $n$ is even. Then, Theorem~\ref{main theorem 1} and Proposition~\ref{proposition 1} together imply that $G$ admits essential-systolic embedding on a translation surface of genus $g_n$ or $g_{n+1}$, depending on whether $n$ is odd or even, respectively. Therefore, we have: $$
g^{\min}_G \leq g_n \quad \text{if } n \text{ is odd,} \quad \text{and} \quad g^{\min}_G \leq g_{n+1} \quad \text{if } n \text{ is even}.$$
Now, by Theorem~\ref{main theorem 1}, $$g_n = 1 - n\left(1 - \frac{n - 1}{2}\right), \quad \text{and} \quad g_{n+1} = 1 - (n + 1)\left(1 - \frac{n}{2}\right).$$

Similarly, for $n < 5$, we take $G$ as a subgraph of $K_5$. Using the same arguments, we conclude that:$$g^{\min}_G \leq g_{5} = 6.$$

\end{proof}
 
\section{Cellular-systolic embedding}

In this section, we study the cellular-systolic embedding of graphs on translation surfaces. Consider the graph $\Sigma_n$, the wedge of $n$-circles. Note that for $n=1$, the graph $\Sigma_1$ does not admit cellular systolic embedding on any translation surface. To see this, if $\Sigma_1$ is cellularly embedded on a closed surface, then that surface is topologically a sphere, which is not a translation surface (see Lemma~\ref{Non example}). Theorem~\ref{cellular embedding} shows that for every $n \geq 2$, the graph $\Sigma_n$ admits a cellular-systolic embedding on some translation surface.  
We state the result below for clarity.

\begin{repthm}[\ref{cellular embedding}] (\cite{MR4985436})
Let \( \Sigma_n \) denote the wedge of \( n \) circles, where \( n \geq 2 \). Then we have the following:
\begin{enumerate}
    \item $\Sigma_n$ admits a cellular-systolic embedding on a translation surface of genus $\left\lfloor \frac{n}{2} \right\rfloor$. Furthermore, this is the maximum possible genus of such a translation surface.
    
    \item The minimum genus of a translation surface on which $\Sigma_n$ admits a cellular-systolic embedding is \( \left\lceil \frac{n+3}{6} \right\rceil \).
    
    \item For any integer \( g \) satisfying $\left\lceil \frac{n+3}{6} \right\rceil \leq g \leq \left\lfloor \frac{n}{2} \right\rfloor$, the graph \( \Sigma_n \) admits cellular-systolic embedding on a translation surface of genus \( g \).
\end{enumerate}
\end{repthm}
\begin{proof}
\begin{enumerate}
\item We prove this result by explicitly building a translation surface on which $\Sigma_n$ admits a cellular-systolic embedding. If the graph $\Sigma_n$ admits such an embedding with $k\geq 1$ complementary disks, then the Euler characteristic formula gives $$ 1 - n + k = 2 - 2g \implies g = \frac{n + 1 - k}{2}.$$ Therefore, $g$ is maximized when $k = 1$ or $k = 2$, depending on whether $n$ is even or odd, respectively. In the either cases, the maximum genus is given by $\left\lfloor \frac{n}{2} \right\rfloor$.

To prove item (1), it suffices to construct a translation surface of the genus $\left\lfloor \frac{n}{2} \right\rfloor$ on which $\Sigma_n$ admits cellular-systolic embedding. We proceed by considering the following two cases.

Case 1. We assume that $n$ is an even integer, i.e., $n = 2m$, for some $m \in \mathbb{N}$. In this case, we construct the translation surface by identifying opposite parallel sides of a regular $2n=4m$-gon $\mathcal{R}_{2n}$. This yields a translation surface with a single singular point of cone angle $(4m-2)\pi$ and $n$ systolic connections, one for each pair of opposite sides. Each side-pair gives a saddle connection of length equal to the side length of $\mathcal{R}_{2n}$. Moreover, in the polygonal unfolding, any saddle connection lifts to a straight segment joining vertices, and the minimum nonzero distance between such vertices is the side length. Hence, no shorter saddle connection exists. This gives the required cellular-systolic embedding of $\Sigma_n$. The genus of the resulting surface is $m = \left\lfloor \frac{n}{2} \right\rfloor$. For example, the construction of the translation surface for $\Sigma_4$ is shown in Figure~\ref{figure1}(a).

Case 2. In this case, we assume $n$ is an odd integer, i.e., $n = 2m + 1$, where $m \in \mathbb{N}$. To construct the desired translation surface, we begin with a regular $(2m+1)$-gon $P_n$, and label its sides by $s_0, s_1, \dots, s_{2m},$ in counterclockwise order. Then, take the reflection of $P_n$ on the side $s_0$, and denote this reflected polygon by $P_n'$. Let $s_i'$ denote the image of side $s_i$ under the reflection, for $i = 0, 1, \dots, 2m$.
We identify the sides of $P_n$ and $P'_n$ as follows: for each $i = 1, 2, \dots, 2m$, identify $s_i$ with $s_{2m+1 - i}'$ by translation. This yields a translation surface with a single singular point, where all the vertices of the polygons are identified, cone angle $(4m - 2)\pi$. The surface has exactly $2m+1$ systolic connections corresponding to the identified edges. The genus of the resulting surface is $g = m=\left\lfloor \frac{n}{2} \right\rfloor$ and provides a cellular-systolic embedding for $\Sigma_n$. This completes the proof of item (1). 

In the case $n=5$, the construction of the translation surface $\Sigma_5$ is illustrated in Figure~\ref{figure1}(b).

Next, we prove item (2).

 \begin{figure}[htbp]
\begin{center}
\begin{tikzpicture}[scale=1.5]
\draw[red] (-2,0)--(-1,0);
\draw[blue] (-1,0)--(-0.293,0.707);
\draw[green] (-0.293,0.707)--(-0.293,1.707);
\draw  (-0.293,1.707)--(-1,2.414);
\draw[red] (-2, 2.414)--(-1,2.414);
\draw[blue] (-2,2.414)--(-2.707,1.707);
\draw[green] (-2.707,1.707)--(-2.707,0.707);
\draw (-2.707,0.707)--(-2,0);

\draw (-2,0) node{\tiny$\bullet$};
\draw (-1,0) node{\tiny$\bullet$};
\draw (-0.293,0.707) node{\tiny$\bullet$};
\draw (-0.293,1.707) node{\tiny$\bullet$};
\draw (-2,2.414) node{\tiny$\bullet$};
\draw (-1,2.414) node{\tiny$\bullet$};
\draw (-2.707,1.707) node{\tiny$\bullet$};
\draw (-2.707,0.707) node{\tiny$\bullet$};
\draw (-1.5,-0.2) node{\tiny$e_1$};
\draw (-0.4,0.3) node{\tiny$e_2$};
\draw (-0.05,1.3) node{\tiny$e_3$};
\draw (-0.4,2.2) node{\tiny$e_4$};
\draw (-1.5,2.6) node{\tiny$e_1$};
\draw (-2.6,2.2) node{\tiny$e_2$};
\draw (-3,1.2) node{\tiny$e_3$};
\draw (-2.6,0.3) node{\tiny$e_4$};
\draw (-1.5,1.3) node {$\mathcal{R}_{2n}$};
\draw (-1.5,-1) node {(a)};

\draw[orange] (3,1)--(4.174,1);
\draw[green] (4.174,1)--(4.538,2.118);
\draw[yellow] (4.538,2.118)--(3.587,2.809);
\draw[red] (3.587,2.809)--(2.636,2.118);
\draw[blue] (2.636,2.118)--(3,1);
\draw[blue] (4.174,1)--(4.538,-.118);
\draw[red] (4.538,-.118)--(3.587,-0.809);
\draw[yellow] (3.587,-0.809)--(2.636,-0.118);
\draw[green] (2.636,-0.118)--(3,1);
\draw (3,1) node{\tiny$\bullet$};
\draw (4.174,1) node{\tiny$\bullet$};
\draw (4.538,2.118) node{\tiny$\bullet$};
\draw (3.587,2.809) node{\tiny$\bullet$};
\draw (2.636,2.118) node{\tiny$\bullet$};
\draw (4.538,-0.118) node{\tiny$\bullet$};
\draw (3.587,-0.809) node{\tiny$\bullet$};
\draw (2.636,-0.118) node{\tiny$\bullet$};
\draw (3.5,-1.3) node {(b)};
\draw (3.7,1.2 ) node {\tiny$s_0$};
\draw (3.7,2.1) node {$P_n$};
\draw (3.7,-0.1) node {$P_n^\prime$};
\draw (3.7,0.75 ) node {\tiny$s_0^\prime$};
\draw (4.6,1.559 ) node {\tiny$s_1$};
\draw (4.2,2.6 ) node {\tiny$s_2$};
\draw (2.9,2.6) node {\tiny$s_3$};
\draw (2.6,1.559) node {\tiny$s_4$};
\draw (4.6,.441) node {\tiny$s_1^\prime$};
\draw (4.2,-0.6) node {\tiny$s_2^\prime$};
\draw (3,-0.7) node {\tiny$s_3^\prime$};
\draw (2.5,.441) node {\tiny$s_4^\prime$};
\end{tikzpicture}
\end{center}
\caption{Translation Surfaces for (a) $\Sigma_4$ and (b) $\Sigma_5.$}
\label{figure1}
\end{figure}

\item To establish this result, we first recall the following proposition.

\begin{repprop}[Restatement of Proposition~\ref{prop:maxsaddleconnections}]
The greatest number of shortest saddle connections of a unit-area surface
in the stratum $\mathcal{H}(k_1,\dots,k_r)$ is
\[
\sum_{i=1}^r 3(k_i+1),
\]
and this number is realized if and only if the surface is a global
maximum for the function $\mathrm{Sys}$.
\end{repprop}

Now, if a translation surface in \( \mathcal{H}(k_1, \dots, k_r) \) has genus \(g\), then 
\(\sum_{i=1}^r k_i = 2g - 2\) (see Section 3.1 in~\cite{Mas22}). Consequently, the maximum number of shortest saddle connections for a translation surface of genus \(g\) with \(r\) singular points is 
\[
\sum_{i=1}^r 3(k_i + 1) = 3(2g - 2 + r).
\]

So, the maximum number of systolic connections on a translation surface of genus $g$ and a single singular point is $3(2g-2+1)=3(2g-1)$. If $\Sigma_n$ admits a cellular-systolic embedding on a translation surface of genus $g$, then we have $$ n \leq 3(2g - 1)  \implies \frac{n + 3}{6}\leq g.$$ To prove the result, it is sufficient to construct a translation surface of genus $\lceil \frac{n + 3}{6} \rceil$ on which $\Sigma_n$ admits a cellular-systolic embedding. 

For $n = 2, 3$, we have $\lceil \frac{n + 3}{6} \rceil = 1$. The desired translation surfaces for $\Sigma_2$ and $\Sigma_3$ are obtained from the square torus and the equilateral torus (formed by two equilateral triangles), respectively (see Figure~\ref{figure3}). Now, for $n \geq 4$, let $g = \lceil \frac{n + 3}{6} \rceil$. Then $g \geq 2$, and we have $6g - 9 < n \leq 6g - 3.$
Hence, $n \in P_g$ where $$P_g = \{6g - 8, 6g - 7, 6g - 6, 6g - 5, 6g - 4, 6g - 3\}.$$

%Observe that:
%\begin{equation}\label{eq:2}
%    \{n \in \mathbb{N} : n \geq 4\} = \bigcup_{g \ge 2} P_g. 
%\end{equation}
Our objective is to construct, for every integer $g \geq 2$, a translation surface of genus $g$ that realises the graph $\Sigma_k$ for each $k \in P_g$.  
In the case of genus $g = 2$, we have $P_2 = \{4, 5, 6, 7, 8, 9\}$. The corresponding translation surfaces of genus $2$ for $\Sigma_k$, with $k \in P_2$, are depicted in Figures~\ref{figure sigma4} and~\ref{figure Sigma8}.

For any $g\geq 3$, to construct translation surfaces for $\Sigma_{6g - 8}, \Sigma_{6g - 7}, \Sigma_{6g - 6}, \dots,$ and $ \Sigma_{6g - 3}$, we take the polygons used in the construction of translation surfaces for $\Sigma_4, \Sigma_5, \Sigma_6, \dots,$ and $\Sigma_9$, respectively, and attach $4(g - 2)$ equilateral triangles as shown in Figures~\ref{figure Sigma(6g-8)} to \ref{figure Sigma(6g-3)}. This gives the desired translation surfaces.
In these constructions, each side of the polygons represents a systolic connection of the surface.
Therefore, for every $k \in P_g$ and for all $g \geq 2$, we obtain translation surfaces of genus $g$ on which $\Sigma_k$ admits a cellular-systolic embedding. %Hence, by the observation~\ref{eq:2} we conclude the proof.

\begin{figure}[htbp]
\begin{center}
\begin{tikzpicture}[scale=2]
\draw[red] (0,0)--(1,0);
\draw[blue] (1,0)--(1,1);
\draw[red] (1,1)--(0,1);
\draw[blue]  (0,0)--(0,1);

\draw (0,0) node{\tiny$\bullet$};
\draw (1,0) node{\tiny$\bullet$};
\draw (1,1) node{\tiny$\bullet$};
\draw (0,1) node{\tiny$\bullet$};
\draw (0.5,-0.1) node{$e_1$};
\draw (0.5,1.1) node{$e_1$};
\draw (-0.15,0.5) node{$e_2$};
\draw (1.15,0.5) node{$e_2$};

\draw[red] (3,0)--(4,0);
\draw[blue] (3,0)--(3.5,0.866);
\draw[green] (3.5,0.866)--(4,0);
\draw[red] (3.5,0.866)--(4.5,0.866);
\draw[blue] (4,0)--(4.5,0.866);
\draw (3,0) node{\tiny$\bullet$};
\draw (4,0) node{\tiny$\bullet$};
\draw (3.5,0.866) node{\tiny$\bullet$};
\draw (4.5,0.866) node{\tiny$\bullet$};
\draw (3.5,-0.1) node{$e_1$};
\draw (4,1) node{$e_1$};
\draw (3.1,0.433) node{$e_2$};
\draw (4.4,0.433) node{$e_2$};
\draw (3.9,0.433) node{$e_3$};
\end{tikzpicture}
\caption{Translation surfaces for $\Sigma_2$ and $\Sigma_3$.}
\label{figure3}
\end{center}
\end{figure}

\begin{figure}[htbp]
\begin{center}
\begin{tikzpicture}[scale=1.5]
\draw[red] (0,0)--(1,0);
\draw[blue] (1,0)--(1.707,0.707);
\draw[green] (1.707,0.707)--(1.707,1.707);
\draw[yellow]  (1.707,1.707)--(1,2.414);
\draw[red] (0, 2.414)--(1,2.414);
\draw[blue] (0,2.414)--(-0.707,1.707);
\draw[green] (-0.707,1.707)--(-0.707,0.707);
\draw[yellow] (-0.707,0.707)--(0,0);
\draw (0,0) node{\tiny$\bullet$};
\draw (1,0) node{\tiny$\bullet$};
\draw (1.707,0.707) node{\tiny$\bullet$};
\draw (1.707,1.707) node{\tiny$\bullet$};
\draw (1,2.414) node{\tiny$\bullet$};
\draw (0,2.414) node{\tiny$\bullet$};
\draw (-0.707,1.707) node{\tiny$\bullet$};
\draw (-0.707,0.707) node{\tiny$\bullet$};
\draw (0.5, -1) node{(a)};
\draw (3.5,-1.5) node{(b)};
\draw (0.5,-0.2) node{\tiny$e_1$};
\draw (1.6,0.3) node{\tiny$e_2$};
\draw (1.95,1.3) node{\tiny$e_3$};
\draw (1.6,2.2) node{\tiny$e_4$};
\draw (0.5,2.6) node{\tiny$e_1$};
\draw (-0.6,2.2) node{\tiny$e_2$};
\draw (-1,1.2) node{\tiny$e_3$};
\draw (-0.6,0.3) node{\tiny$e_4$};

\draw[orange] (3,1)--(4.174,1);
\draw[green] (4.174,1)--(4.538,2.118);
\draw[yellow] (4.538,2.118)--(3.587,2.809);
\draw[red] (3.587,2.809)--(2.636,2.118);
\draw[blue] (2.636,2.118)--(3,1);
\draw[blue] (4.174,1)--(4.538,-.118);
\draw[red] (4.538,-.118)--(3.587,-0.809);
\draw[yellow] (3.587,-0.809)--(2.636,-0.118);
\draw[green] (2.636,-0.118)--(3,1);
\draw (3,1) node{\tiny$\bullet$};
\draw (4.174,1) node{\tiny$\bullet$};
\draw (4.538,2.118) node{\tiny$\bullet$};
\draw (3.587,2.809) node{\tiny$\bullet$};
\draw (2.636,2.118) node{\tiny$\bullet$};
\draw (4.538,-0.118) node{\tiny$\bullet$};
\draw (3.587,-0.809) node{\tiny$\bullet$};
\draw (2.636,-0.118) node{\tiny$\bullet$};
\draw (3.7,1.2 ) node {\tiny$e_5$};
\draw (4.6,1.559 ) node {\tiny$e_1$};
\draw (4.2,2.6 ) node {\tiny$e_2$};
\draw (2.9,2.6) node {\tiny$e_3$};
\draw (2.6,1.559) node {\tiny$e_4$};
\draw (4.6,.441) node {\tiny$e_4$};
\draw (4.2,-0.6) node {\tiny$e_3$};
\draw (3,-0.7) node {\tiny$e_2$};
\draw (2.5,.441) node {\tiny$e_1$};

\draw[red] (6,0)--(7,0);
\draw[blue] (7,0)--(7,1);
\draw[green] (7,1)--(7,2);
\draw[yellow] (7,2)--(7,3);
\draw[red] (7,3)--(6,3);
\draw[blue] (6,3)--(6,2);
\draw[green] (6,2)--(6,1);
\draw[yellow](6,1)--(6,0);
\draw[orange]{} (6,2)--(7,2);
\draw[violet] (6,1)--(7,1);
\draw (6.5,-0.15) node {\tiny$e_1$};
\draw (7.2,0.5) node {\tiny$e_2$};
\draw (7.2,1.5) node {\tiny$e_3$};
\draw (7.2,2.5) node {\tiny$e_4$};
\draw (6.5,3.15) node {\tiny$e_1$};
\draw (5.8,2.5) node {\tiny$e_2$};
\draw (5.8,1.5) node {\tiny$e_3$};
\draw (5.8,0.5) node {\tiny$e_4$};
\draw (6.5,2.15) node {\tiny$e_5$};
\draw (6.5,1.15) node {\tiny$e_6$};

\draw (6,0) node{\tiny$\bullet$};
\draw (7,0) node{\tiny$\bullet$};
\draw (7,1) node{\tiny$\bullet$};
\draw (7,2) node{\tiny$\bullet$};
\draw (7,3) node{\tiny$\bullet$};
\draw (6,3) node{\tiny$\bullet$};
\draw (6,2) node{\tiny$\bullet$};
\draw (6,1) node{\tiny$\bullet$};
\draw (6.5,-0.8) node {(c)};

\draw[yellow] (9.5,2)--(9.5,3);
\draw[red] (9.5,3)--(8.5,3);
\draw[blue] (8.5,3)--(8.5,2);
\draw[violet] (8.5,2)--(9.5,2);
\draw[orange] (8.5,2)--(9,1.134);
\draw[green] (9,1.134)--(9.5,2);
\draw (9,1.134)--(8,1.134);
\draw[green] (8,1.134)--(8.5,2);
\draw[yellow] (8,1.134)--(8,0.134);
\draw[red] (8,0.134)--(9,0.134);
\draw[blue] (9,0.134)--(9,1.134);

\draw (8.5,0) node {\tiny$e_1$};
\draw (9.2,0.634) node {\tiny$e_2$};
\draw (8.08,1.6) node {\tiny$e_3$};
\draw (9.4,1.5) node {\tiny$e_3$};
\draw (9.7,2.5) node {\tiny$e_4$};
\draw (9,3.2) node {\tiny$e_1$};
\draw (8.3,2.5) node {\tiny$e_2$};
\draw (7.8,0.6) node {\tiny$e_4$};
\draw (9,2.1) node {\tiny$e_5$};
\draw (8.95,1.56) node {\tiny$e_6$};
\draw (8.5,1) node {\tiny$e_7$};

\draw (9.5,2) node{\tiny$\bullet$};
\draw (9.5,3) node{\tiny$\bullet$};
\draw (8.5,3) node{\tiny$\bullet$};
\draw (8.5,2) node{\tiny$\bullet$};
\draw (9,1.134) node{\tiny$\bullet$};
\draw (8,1.134) node{\tiny$\bullet$};
\draw (9,0.134) node{\tiny$\bullet$};
\draw (8,0.134) node{\tiny$\bullet$};
\draw (8.5,-1) node {(d)};

\end{tikzpicture}

\caption{Translation surfaces for (a) $\Sigma_4$, (b) $\Sigma_5$, (c) $\Sigma_6$ and (d) $\Sigma_7.$}
\label{figure sigma4}
\end{center}
\end{figure}
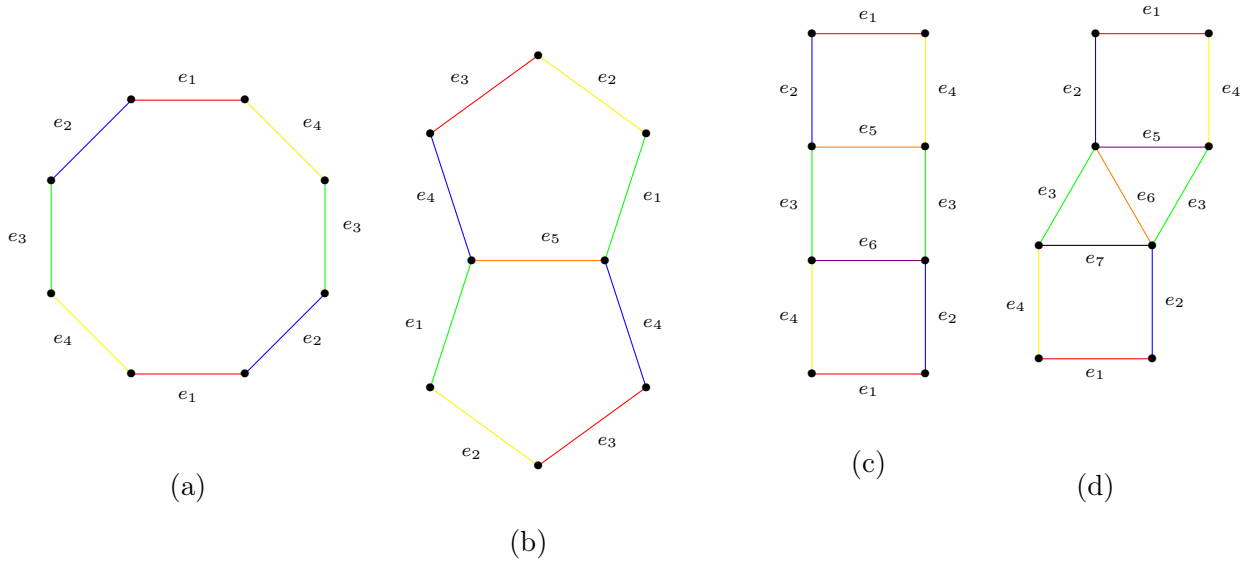

\begin{figure}[htbp]
\begin{center}
\rotatebox{0}{
\begin{tikzpicture}[scale=2]
\draw[] (0,1)--(0.5,0.134);
\draw[blue] (0.5,.134)--(1,1);
\draw[red] (0.5,0.134)--(-0.5,0.134);
\draw[yellow] (-0.5,0.134)--(0,1);
\draw[green] (1,1)--(1,2);
\draw[blue] (0,2)--(0.5,2.866);
\draw[violet] (0.5,2.866)--(1,2);
\draw[red] (0.5,2.866)--(1.5,2.866);
\draw[yellow] (1.5,2.866)--(1,2);
\draw[green] (0,2)--(0,1);
\draw[orange]{} (0,2)--(1,2);
\draw[violet] (0,1)--(1,1);
\draw (0,0) node {\tiny$e_1$};
\draw (0.9,0.5) node {\tiny$e_2$};
\draw (1.2,1.5) node {\tiny$e_3$};
\draw (1.5,2.433) node {\tiny$e_4$};
\draw (1,3) node {\tiny$e_1$};
\draw (0,2.433) node {\tiny$e_2$};
\draw (-0.2,1.5) node {\tiny$e_3$};
\draw (-0.5,0.6) node {\tiny$e_4$};
\draw (0.95,2.433) node {\tiny$e_5$};
\draw (0.5,2.1) node {\tiny$e_6$};
\draw (0.5,1.1) node {\tiny$e_7$};
\draw (0.15,0.4) node {\tiny$e_8$};

\draw (0,1) node{\tiny$\bullet$};
\draw (0.5,0.134) node{\tiny$\bullet$};
\draw (1,1) node{\tiny$\bullet$};
\draw (-0.5,0.134) node{\tiny$\bullet$};
\draw (0,2) node{\tiny$\bullet$};
\draw (1,2) node{\tiny$\bullet$};
\draw (0.5,2.866) node{\tiny$\bullet$};
\draw (1.5,2.866) node{\tiny$\bullet$};
\draw (0,-0.8) node {(e)};

\draw[orange] (3,1.732)--(3.5,0.866);
\draw[green] (3.5,0.866)--(4,1.732);
\draw[] (3.5,0.866)--(2.5,0.866);
\draw[green] (2.5,0.866)--(3,1.714);
\draw[magenta] (2.5,0.866)--(3,0);
\draw[blue] (3,0)--(3.5,0.866);
\draw[red] (3,0)--(2,0);
\draw[yellow] (2,0)--(2.5,0.866);
\draw[blue] (3,1.714)--(3.5,2.58);
\draw[violet] (3.5,2.58)--(4,1.714);
\draw[red] (3.5,2.58)--(4.5,2.58);
\draw[yellow] (4.5,2.58)--(4,1.714);
\draw[cyan] (3,1.714)--(4,1.714);

\draw (2.5,-0.2) node {\tiny$e_1$};
\draw (3.4,0.3) node {\tiny$e_2$};
\draw (4,1.299) node {\tiny$e_3$};
\draw (4.5,2.2) node {\tiny$e_4$};
\draw (4,2.7) node {\tiny$e_1$};
\draw (3,2.2) node {\tiny$e_2$};
\draw (2.5,1.3) node {\tiny$e_3$};
\draw (2,0.433) node {\tiny$e_4$};
\draw (3.9,2.2) node {\tiny$e_5$};
\draw (3.5,1.6) node {\tiny$e_6$};
\draw (3.15,1.2) node {\tiny$e_7$};
\draw (3,0.7) node {\tiny$e_8$};
\draw (2.6,0.3) node {\tiny$e_9$};

\draw (3,1.714) node{\tiny$\bullet$};
\draw (3.5,0.866) node{\tiny$\bullet$};
\draw (4,1.714) node{\tiny$\bullet$};
\draw (2.5,0.866) node{\tiny$\bullet$};
\draw (3,0) node{\tiny$\bullet$};
\draw (3.5,0.866) node{\tiny$\bullet$};
\draw (2,0) node{\tiny$\bullet$};
\draw (3.5,2.58) node{\tiny$\bullet$};
\draw (4.5,2.58) node{\tiny$\bullet$};
\draw (2.5,-1) node {(f)};

\end{tikzpicture}
}
\caption{Translation surfaces for (e) $\Sigma_8$ and (f) $\Sigma_9.$}
\label{figure Sigma8}
\end{center}
\end{figure}
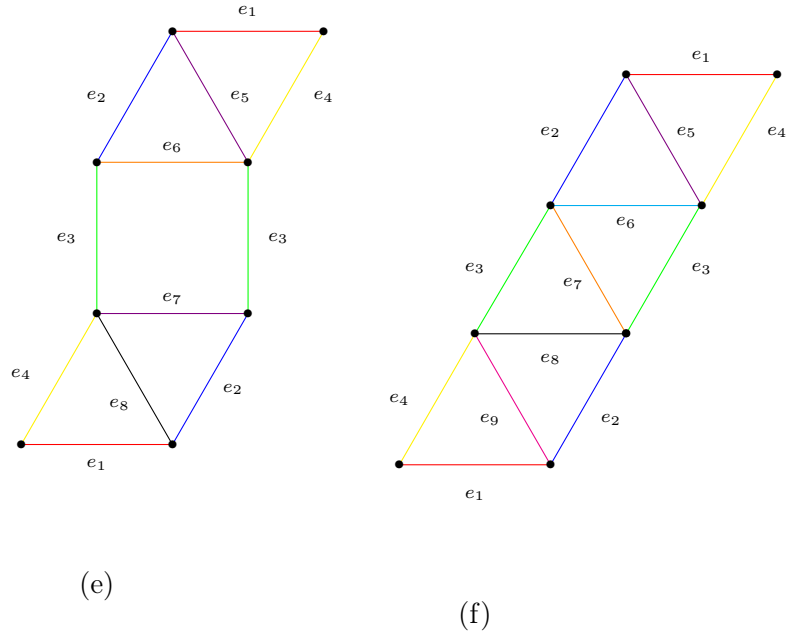
%%%%%%%%%%%%%%%%%%%%%%%%%%%%%%%%%%%%%%%%%%%%%%%%%%%%%%%%%%%%%%%%%%%%%%%%%%%%
\begin{figure}[htbp]
\begin{center}
\rotatebox{0}{
\begin{tikzpicture}[scale=1]
\draw[] (0,0)--(1,0);
\draw[blue] (1,0)--(1.707,0.707);
\draw[green] (1.707,0.707)--(1.707,1.707);
\draw[red]  (1.707,1.707)--(1,2.414);
\draw[] (0, 2.414)--(1,2.414);
\draw[blue] (0,2.414)--(-0.707,1.707);
\draw[green] (-0.707,1.707)--(-0.707,0.707);
\draw[red] (-0.707,0.707)--(0,0);
\draw (1.5,0.3) node{\tiny$e_1$};
\draw (1.9,1.1) node{\tiny$e_2$};
\draw (1.2,2) node{\tiny$e_3$};
\draw (-0.5,2.2) node{\tiny$e_1$};
\draw (-0.9,1.1) node{\tiny$e_2$};
\draw (-0.1,0.35) node{\tiny$e_3$};
\draw (0,2.414)--(0.5,3.28)--(1,2.414);
\draw (0.5,3.28)--(1.5,3.28)--(1,2.414);
\draw (1,2.414)--(2,2.414)--(1.5,3.28);
\draw (1.5,3.28)--(2.5,3.28)--(2,2.414);
\draw[dashed] (2,2.414)--(3,2.414);
\draw[dashed] (2.5,3.28)--(3.5,3.28);
\draw (3.5,3.28)--(4.5,3.28)--(4,2.414)--(3.5,3.28);
\draw (3.5,3.28)--(3,2.414)--(4,2.414);
\draw (1,3.4) node{\tiny$a_1$};
\draw (2,3.4) node{\tiny$a_2$};
\draw (4,3.4) node{\tiny$a_{g-2}$};
\draw (1.7,2.25) node{\tiny$b_1$};
\draw (3.5,2.28) node{\tiny$b_{g-3}$};

\draw (0,0)--(0.5,-0.866)--(1,0);
\draw (0.5,-0.866)--(-0.5,-0.866)--(0,0);
\draw (0,0)--(-1,0)--(-0.5,-0.866);
\draw(-1,0)--(-1.5,-0.866)--(-0.5,-0.866);
\draw (-1,0)--(-2,0)--(-1.5,-0.866);
\draw (0.1,2.914) node{\tiny$c$};
\draw (0.9,-0.5) node {\tiny$c$};
\draw (4.3,2.7) node{\tiny$d$};
\draw (-4.4,-0.366) node{\tiny$d$};
\draw[dashed] (-2,0)--(-3,0);
\draw[dashed] (-1.5,-0.866)--(-3.5,-0.866);
\draw (-3.5,-0.866)--(-4.5,-0.866)--(-4,0)--(-3.5,-0.866);
\draw (-3,0)--(-4,0);
\draw (-3,0)--(-3.5,-0.866);
\draw (0,-1) node{\tiny$a_1$};
\draw (-1,-1) node {\tiny$a_2$};
\draw (-4,-1) node {\tiny$a_{g-2}$};
\draw (-0.5,0.2) node {\tiny$b_1$};
\draw (-1.5,0.2) node {\tiny$b_2$};
\draw (-3.5,0.2) node {\tiny$b_{g-3}$};
\end{tikzpicture}
}
\caption{Translation surface for $\Sigma_{6g-8}$}
\label{figure Sigma(6g-8)}
\end{center}
\end{figure}

\begin{figure}[htbp]
\begin{center}
\begin{tikzpicture}[scale=1]
\draw[orange] (0,0)--(1.174,0);
\draw[green] (1.174,0)--(1.538,1.118);
\draw[] (1.538,1.118)--(0.587,1.809);
\draw[red] (0.587,1.809)--(-0.364,1.118);
\draw[blue] (-0.364,1.118)--(0,0);
\draw[blue] (1.174,0)--(1.538,-1.118);
\draw[red] (1.538,-1.118)--(0.587,-1.809);
\draw[] (0.587,-1.809)--(-0.364,-1.118);
\draw[green] (-0.364,-1.118)--(0,0);

\draw (1.3,-1.6) node{\tiny$e_1$};
\draw (1.56, -0.55) node{\tiny$e_2$};
\draw (0.1,1.7) node{\tiny$e_1$};
\draw (-0.4,0.55) node{\tiny$e_2$};
\draw (-.4,-0.55) node{\tiny$e_3$};
\draw (1.55,0.5) node{\tiny$e_3$};

\draw (0.587,1.809)--(1.5003,2.2161)--(1.538,1.118)--(2.309,1.634)--(1.5003,2.2161);
\draw[dashed] (2.309,1.634)--(3.118,1.041);
\draw[dashed] (1.538,1.118)--(2.205,0.634);
\draw (3.118,1.041)--(2.205,0.634);
\draw (2.205,0.634)--(3.014,0.046);
\draw (3.118,1.041)--(3.928,0.454)--(3.014,0.046)--(3.118,1.041);

\draw(-0.364,-1.118)--(-0.467,-2.112)--(0.587,-1.809);
\draw (-0.467,-2.112)--(-1.277,-1.525)--(-0.364,-1.118);
\draw[dashed] (-0.364,-1.118)--(-1.173,-0.530);
\draw[dashed] (-1.277,-1.525)--(-2.086,-0.937);
\draw (-1.173,-0.530)--(-2.086,-0.937);
\draw (-1.173,-0.530)--(-1.982,0.056);
\draw (-2.086,-0.937)--(-2.895,-0.350);
\draw (-2.895,-0.350)--(-1.982,0.056)--(-2.086,-0.937);

\draw (2,2) node{\tiny$a_1$};
\draw (-0.9,-2) node{\tiny$a_1$};
\draw (3.75,0.9) node{\tiny$a_{g-2}$};
\draw (-2.8,-0.8) node {\tiny$a_{g-2}$};
\draw (2.5,0.1) node {\tiny$b_{g-3}$};
\draw (-1.4,0) node {\tiny$b_{g-3}$};
\draw (1,2.2) node {\tiny$c$};
\draw (0,-2.2) node {\tiny$c$};
\draw (3.6,0.15) node {\tiny$d$};
\draw (-2.6,0) node {\tiny$d$};

\end{tikzpicture}

\caption{Translation surface for $\Sigma_{6g-7}$}
\label{figure Sigma(6g-7)}
\end{center}
\end{figure}

\begin{figure}[htbp]
\begin{center}
\begin{tikzpicture}[scale=1]
\draw[] (0,0)--(1,0);
\draw[blue] (1,0)--(1,1);
\draw[green] (1,1)--(1,2);
\draw[yellow] (1,2)--(1,3);
\draw[] (1,3)--(0,3);
\draw[blue] (0,3)--(0,2);
\draw[green] (0,2)--(0,1);
\draw[yellow](0,1)--(0,0);
\draw[orange]{} (0,2)--(1,2);
\draw[violet] (0,1)--(1,1);

\draw (1.2,.5) node{\tiny$e_1$};
\draw (1.2,1.5) node{\tiny$e_2$};
\draw (1.2,2.4) node{\tiny$e_3$};
\draw (-0.2,.5) node{\tiny$e_3$};
\draw (-0.2,1.5) node{\tiny$e_2$};
\draw (-0.2,2.4) node{\tiny$e_1$};

\draw (0,0)--(0.5,-0.866)--(1,0);
\draw (0.5,-0.866)--(-0.5,-0.866)--(0,0);
\draw (0,0)--(-1,0)--(-0.5,-0.866);
\draw (-1,0)--(-1.5,-0.866)--(-0.5,-0.866);
\draw[dashed] (-1,0)--(-2,0);
\draw[dashed] (-1.5,-0.866)--(-2.5,-0.866);
\draw (-3,0)--(-2.5,-0.866)--(-3.5,-0.866)--(-3,0);
\draw (-3,0)-- (-2,0)--(-2.5,-0.866);

\draw (0,3)--(0.5,3.866)--(1,3);
\draw (0.5,3.866)--(1.5,3.866)--(1,3);
\draw (1,3)--(2,3)--(1.5,3.866)--(2.5,3.866)--(2,3);
\draw[dashed] (2.5,3.866)--(3.5,3.866);
\draw[dashed] (2,3)--(3,3);
\draw (3.5,3.866)--(4,3)--(4.5,3.866)--(3.5,3.866);
\draw (4,3)--(3,3)--(3.5,3.866);
\draw (1,4) node {\tiny$a_1$};
\draw (2,4) node {\tiny$a_2$};
\draw (4,4) node {\tiny$a_{g-2}$};

\draw (0,-1) node{\tiny$a_1$};
\draw (-1,-1) node{\tiny$a_2$};
\draw (-3,-1) node{\tiny$a_{g-2}$};
\draw (1.5,2.8) node{\tiny$b_1$};
\draw (3.5,2.8) node{\tiny$b_{g-3}$};
\draw (-0.5,0.2) node {\tiny$b_1$};
\draw (-2.5,0.2) node{\tiny$b_{g-3}$};
\draw (0.1,3.5) node {\tiny$c$};
\draw (0.8,-.5) node {\tiny$c$};
\draw (4.35,3.4) node {\tiny$d$};
\draw (-3.38,-.35) node {\tiny$d$};
\end{tikzpicture}
\caption{ Translation surface for $\Sigma_{6g-6}$}
\label{figure sigma(6g-6)}
\end{center}
\end{figure}

\begin{figure}[htbp]
\begin{center}
\begin{tikzpicture}[scale=1]
\draw[red] (1,2)--(1,3);
\draw[] (1,3)--(0,3);
\draw[blue] (0,3)--(0,2);
\draw[violet] (0,2)--(1,2);
\draw[orange] (0,2)--(0.5,1.134);
\draw[green] (0.5,1.134)--(1,2);
\draw (0.5,1.134)--(-.5,1.134);
\draw[green] (-.5,1.134)--(0,2);
\draw[red] (-0.5,1.134)--(-0.5,0.134);
\draw[] (-0.5,0.134)--(0.5,0.134);
\draw[blue] (0.5,0.134)--(0.5,1.134);

\draw (1.2,2.4) node{\tiny$e_3$};
\draw (-.7,0.4) node{\tiny$e_3$};
\draw (-0.4,1.7) node{\tiny$e_2$};
\draw (1,1.5) node{\tiny$e_2$};
\draw (-0.2,2.6) node{\tiny$e_1$};
\draw (0.7,0.5) node{\tiny$e_1$};

\draw (0,3)--(0.5,3.866)--(1.5,3.866)--(1,3)--(0.5,3.866);
\draw (1,3)--(2,3)--(1.5,3.866)--(2.5,3.866)--(2,3);
\draw[dashed] (2,3)--(4,3);
\draw[dashed] (2.5,3.866)--(4.5,3.866);
\draw (4.5,3.866)--(5.5,3.866)--(5,3)--(4,3)--(4.5,3.866)--(5,3);
\draw (0.5,0.134)--(0,-0.732)--(-1,-0.732)--(-0.5,0.134)--(0,-0.732);
\draw (-0.5,0.134)--(-1.5,0.134)--(-1,-0.732)--(-2,-0.732)--(-1.5,0.134);
\draw[dashed] (-1.5,0.134)--(-3.5,0.134);
\draw[dashed] (-2,-0.732)--(-4,-0.732);
\draw (-4,-0.732)--(-3.5,0.134)--(-4.5,0.134)--(-4,-0.732)--(-5,-0.732)--(-4.5,0.134);

\draw (0.9,4) node {\tiny$a_1$};
\draw (1.9,4) node {\tiny$a_2$};
\draw (4.9,4) node {\tiny$a_{g-2}$};
\draw (4.7,2.8) node {\tiny$b_{g-3}$};
\draw (1.6,2.8) node {\tiny$b_1$};
\draw (0.1,3.5) node{\tiny$c$};
\draw (5.5,3.5) node{\tiny$d$};
\draw (0.3,-0.5) node{\tiny$c$};
\draw (-5,-0.3) node{\tiny$d$};

\draw (-0.5,-1) node {\tiny$a_1$};
\draw (-1.5,-1) node {\tiny$a_2$};
\draw (-4.5,-1) node {\tiny$a_{g-2}$};
\draw (-4.3,0.35) node {\tiny$b_{g-3}$};
\draw (-1.3,0.35) node {\tiny$b_{1}$};

\end{tikzpicture}
\caption{Translation surface for $\Sigma_{6g-5}$}
\label{figure Sigma 6g-5}
\end{center}
\end{figure}

\begin{figure}[htbp]
\begin{center}
\begin{tikzpicture}[scale=1]
\draw[] (0,1)--(0.5,0.134);
\draw[blue] (0.5,.134)--(1,1);
\draw[] (0.5,0.134)--(-0.5,0.134);
\draw[red] (-0.5,0.134)--(0,1);
\draw[green] (1,1)--(1,2);
\draw[blue] (0,2)--(0.5,2.866);
\draw[violet] (0.5,2.866)--(1,2);
\draw[] (0.5,2.866)--(1.5,2.866);
\draw[red] (1.5,2.866)--(1,2);
\draw[green] (0,2)--(0,1);
\draw[orange]{} (0,2)--(1,2);
\draw[violet] (0,1)--(1,1);

\draw (0.5,2.866)--(1,3.866)--(2,3.866)--(1.5,2.866)--(1,3.866);
\draw (1.5,2.866)--(2.5,2.866)--(3,3.866)--(2,3.866)--(2.5,2.866);
\draw[dashed] (3,3.866)--(5,3.866);
\draw[dashed] (2.5,2.866)--(4.5,2.866);
\draw (5,3.866)--(4.5,2.866)--(5.5,2.866)--(6,3.866)--(5,3.866)--(5.5,2.866);

\draw (0.5,0.134)--(0,-0.732)--(-1,-0.732)--(-0.5,0.134)--(0,-0.732);
\draw (-0.5,0.134)--(-1.5,0.134)--(-1,-0.732)--(-2,-0.732)--(-1.5,0.134);
\draw[dashed] (-1.5,0.134)--(-3.5,0.134);
\draw[dashed] (-2,-0.732)--(-4,-0.732);
\draw (-4,-0.732)--(-3.5,0.134)--(-4.5,0.134)--(-4,-0.732)--(-5,-0.732)--(-4.5,0.134);

\draw (1.5,4) node{\tiny$a_1$};
\draw (2.5,4) node{\tiny$a_2$};
\draw (5.5,4) node{\tiny$a_{g-2}$};

\draw (5.1,2.7) node{\tiny$b_{g-3}$};
\draw (2.1,2.7) node{\tiny$b_{1}$};
\draw (-1,0.3) node{\tiny$b_{1}$};
\draw (-4,0.3) node{\tiny$b_{g-3}$};

\draw (-0.5,-0.9) node{\tiny$a_{1}$};
\draw (-1.5,-0.9) node{\tiny$a_{2}$};
\draw (-4.5,-0.9) node{\tiny$a_{g-2}$};

\draw (0.4,-0.4) node{\tiny$c$};
\draw (-5,-0.3) node{\tiny$d$};
\draw (0.6,3.5) node{\tiny$c$};
\draw (5.9,3.4) node{\tiny$d$};

\draw (0.9,0.5) node{\tiny$e_1$};
\draw (1.2,1.5) node{\tiny$e_2$};
\draw (1.5,2.5) node{\tiny$e_3$};

\draw (0.1,2.5) node{\tiny$e_1$};
\draw (-0.2,1.5) node{\tiny$e_2$};
\draw (-0.5,0.6) node{\tiny$e_3$};
\end{tikzpicture}
\caption{Translation surface for $\Sigma_{6g-4}.$}
\label{figure8}
\end{center}
\end{figure}

\begin{figure}[htbp]
\begin{center}
\rotatebox{0}{
\begin{tikzpicture}[scale=1]
\draw[orange] (0,2)--(0.5,1.134);
\draw[green] (0.5,1.134)--(1,2);
\draw[] (0.5,1.134)--(-0.5,1.134);
\draw[green] (-0.5,1.134)--(0,2);
\draw[magenta] (-0.5,1.134)--(0,0.268);
\draw[blue] (0,0.268)--(0.5,1.134);
\draw  (0,0.268)--(-1,0.268);
\draw[yellow] (-1,0.268)--(-0.5,1.134);
\draw[blue] (0,2)--(0.5,2.866);
\draw[violet] (0.5,2.866)--(1,2);
\draw[] (0.5,2.866)--(1.5,2.866);
\draw[yellow] (1.5,2.866)--(1,2);
\draw[cyan] (0,2)--(1,2);

\draw (0.4,0.65) node{\tiny$e_1$};
\draw (1,1.5) node{\tiny$e_2$};
\draw (1.5,2.43) node{\tiny$e_3$};
\draw (0.1,2.5) node{\tiny$e_1$};
\draw (-0.4,1.7) node{\tiny$e_2$};
\draw (-1,0.65) node{\tiny$e_3$};

\draw (0.5,2.866)--(1,3.732)--(1.5,2.866);
\draw (1,3.732)--(2,3.732)--(1.5,2.866);
\draw (1.5,2.866)--(2.5,2.866)--(2,3.732);
\draw (2.5,2.866)--(3,3.732)--(2,3.732);
\draw[dashed] (2.5,2.866)--(3.5,2.866);
\draw[dashed] (3,3.732)--(4,3.732);
\draw (4,3.732)--(4.5,2.866)--(5,3.732)--(4,3.732);
\draw (4.5,2.866)--(3.5,2.866)--(4,3.732);

\draw (-1,0.268)--(-0.5,-0.598)--(0,0.268);
\draw (-1,0.268)--(-1.5,-.598)--(-0.5,-.598);
\draw (-1,0.268)--(-2,0.268)--(-1.5,-.598);
\draw (-2,0.268)--(-2.5,-.598)--(-1.5,-.598);
\draw[dashed] (-2,0.268)--(-3,0.268);
\draw[dashed] (-2.5,-.598)--(-3.5,-.598);
\draw (-3.5,-.598)--(-4,.268)--(-4.5,-.598)--(-3.5,-.598);
\draw (-4,0.268)--(-3,0.268)--(-3.5,-.598);

\draw (1.5,3.866) node{\tiny$a_1$};
\draw (2.5,3.866) node{\tiny$a_2$};
\draw (4.5,3.866) node{\tiny$a_{g-2}$};
\draw (-1,-0.8) node{\tiny$a_1$};
\draw (-2,-0.8) node{\tiny$a_2$};
\draw (-4,-0.8) node{\tiny$a_{g-2}$};

\draw (0.65,3.366) node{\tiny$c$};
\draw (-0.1,-0.2) node{\tiny$c$};

\draw (4.8,3.1) node{\tiny$d$};
\draw (-4.4,-0.1) node{\tiny$d$};

\draw (2,2.7) node{\tiny$b_1$};
\draw (4,2.7) node{\tiny$b_{g-3}$};
\draw (-1.5,0.45) node{\tiny$b_1$};
\draw (-3.5,0.45) node{\tiny$b_{g-3}$};
\end{tikzpicture}
}
\caption{Translation surface for $\Sigma_{6g-3}.$}
\label{figure Sigma(6g-3)}
\end{center}
\end{figure}

\item
Let $g$ be an integer that satisfies $\left\lceil \frac{n+3}{6} \right\rceil\le g\le\left\lfloor \frac{n}{2} \right\rfloor$. We aim to show that \( \Sigma_n \) admits cellular-systolic embedding on a translation surface of genus $g$.
To construct such a translation surface, we require \( (n - 2g + 1) \) polygonal faces. This follows from the Euler characteristic formula:
$$\chi = V - E + F = 2 - 2g,$$
where \( V \), \( E \), and \( F \) denote the numbers of vertices, edges, and faces, respectively. Now, \( \Sigma_n \) has \( 1 \) vertex and an $n$ edges to yield the desired genus. From that we find \( F \) = \(n-2g+1\).
We consider two cases based on the parity of \( (n - 2g + 1) \):

Case 1. Assume \( n - 2g + 1 = 2k + 1 \), for some \( k \in \mathbb{N} \). In this case, we construct the translation surface using \( 2k \) equilateral triangles and one regular \( (2n - 6k) \)-gon. These polygons are glued together placing \( k \) equilateral triangles on each of the two opposite ends of the \( (2n - 6k) \)-gon, as shown in Figure~\ref{intermediate genus 1}. This configuration yields a translation surface of genus \( g \) that supports a cellular and systolic embedding of \( \Sigma_n \).

Case 2. Assume \( n - 2g + 1 = 2k \), for some \( k \in \mathbb{N} \). In this case, we use \( 2k - 2 \) equilateral triangles and two regular \( (n - 3k + 3) \)-gons. These polygons are glued to form the desired surface, as illustrated in Figure~\ref{intermediate genus 2}. Again, the resulting translation surface supports a cellular-systolic embedding of \( \Sigma_n \).

In both cases, the number of polygonal faces is exactly \( n - 2g + 1 \), and the resulting surface has genus \( g \), as required by the Euler characteristic. The embeddings are constructed to satisfy the cellular and systolic conditions. Hence, for every integer \( g \) satisfying $\left\lceil \frac{n+3}{6} \right\rceil\le g\le\left\lfloor \frac{n}{2} \right\rfloor$, such a translation surface exists.
\end{enumerate}
\end{proof}

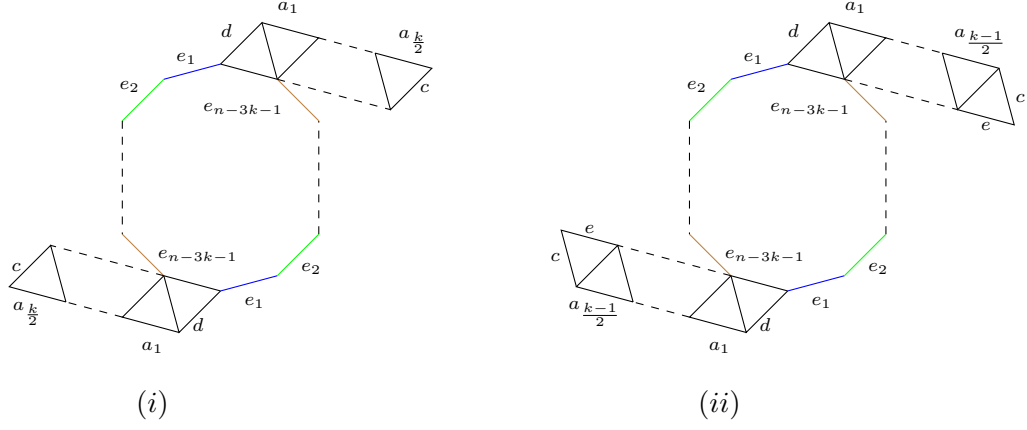
\begin{figure}[htbp]
\begin{center}
\begin{tikzpicture}[scale=1.5]

% --- Original figure ---
\begin{scope}[xshift=0 cm]
\draw[blue] (0,1)--(-0.5,0.866);
\draw[green] (-0.5,0.866)--(-0.866,0.5);
\draw[brown] (-0.866,-0.5)--(-0.5,-0.866);
\draw(-0.5,-0.866)--(0,-1);
\draw[dashed] (-0.866,0.5)--(-0.866,-0.5);
\draw[dashed] (0.866,-0.5)--(0.866,0.5);
\draw[blue] (0,-1)--(0.5,-0.866);
\draw[green] (0.5,-0.866)--(0.866,-0.5);
\draw[brown] (0.866,0.5)--(0.5,0.866);

\draw (-0.3,1.05) node{\tiny$e_1$};
\draw (-0.8,0.8) node{\tiny$e_2$};
\draw (0.3,-1.1) node{\tiny$e_1$};
\draw (0.8,-0.8) node{\tiny$e_2$};
\draw (-0.2,-0.7) node{\tiny\tiny$e_{n-3k-1}$};
\draw (0.2,0.6) node{\tiny\tiny$e_{n-3k-1}$};

\draw[] (0.5,0.866)--(0,1);
\draw[] (0,1)--(0.365,1.365)--(0.5,0.866);
\draw (0.365,1.365)--(0.864,1.231)--(0.5,0.866);
\draw (-0.5,-0.866)--(-0.366,-1.366)--(0,-1);
\draw (-0.5,-0.866)--(-0.865,-1.231)--(-0.366,-1.366);
\draw (0.6,1.5) node{\tiny$a_1$};
\draw (-0.6,-1.5) node {\tiny$a_1$};
\draw[dashed] (0.5,0.866)--(1.498,0.598);
\draw[dashed] (0.864,1.231)--(1.363,1.097);
\draw (1.363,1.097)--(1.498,0.598)--(1.863,0.963)--(1.363,1.097);
\draw (1.7, 1.2) node{\tiny$a_{\scriptstyle\frac{k}{2}}$};
\draw (-1.7, -1.2) node{\tiny$a_{\scriptstyle\frac{k}{2}}$};
\draw (1.8,0.79) node{\tiny$c$};
\draw (0.05,1.3) node{\tiny$d$};
\draw (-1.8,-0.79) node{\tiny$c$};
\draw (-0.2,-1.3) node{\tiny$d$};
\draw (-0.6,-2) node {$(i)$};
\draw[dashed] (-0.5,-0.866)--(-1.498,-0.598);
\draw[dashed] (-0.864,-1.231)--(-1.363,-1.097);
\draw (-1.363,-1.097)--(-1.498,-0.598)--(-1.863,-0.963)--(-1.363,-1.097);
\end{scope}

% --- Shifted copy ---
\begin{scope}[xshift=5 cm]
\draw[blue] (0,1)--(-0.5,0.866);
\draw[green] (-0.5,0.866)--(-0.866,0.5);
\draw[brown] (-0.866,-0.5)--(-0.5,-0.866);
\draw(-0.5,-0.866)--(0,-1);
\draw[dashed] (-0.866,0.5)--(-0.866,-0.5);
\draw[dashed] (0.866,-0.5)--(0.866,0.5);
\draw[blue] (0,-1)--(0.5,-0.866);
\draw[green] (0.5,-0.866)--(0.866,-0.5);
\draw[brown] (0.866,0.5)--(0.5,0.866);

\draw (-0.3,1.05) node{\tiny$e_1$};
\draw (-0.8,0.8) node{\tiny$e_2$};
\draw (0.3,-1.1) node{\tiny$e_1$};
\draw (0.8,-0.8) node{\tiny$e_2$};
\draw (-0.2,-0.7) node{\tiny\tiny$e_{n-3k-1}$};
\draw (0.2,0.6) node{\tiny\tiny$e_{n-3k-1}$};

\draw[] (0.5,0.866)--(0,1);
\draw[] (0,1)--(0.365,1.365)--(0.5,0.866);
\draw (0.365,1.365)--(0.864,1.231)--(0.5,0.866);
\draw (-0.5,-0.866)--(-0.366,-1.366)--(0,-1);
\draw (-0.5,-0.866)--(-0.865,-1.231)--(-0.366,-1.366);
\draw (0.6,1.5) node{\tiny$a_1$};
\draw (-0.6,-1.5) node {\tiny$a_1$};
\draw[dashed] (0.5,0.866)--(1.498,0.598);
\draw[dashed] (0.864,1.231)--(1.363,1.097);
\draw (1.363,1.097)--(1.498,0.598)--(1.863,0.963)--(1.363,1.097);

\draw (1.498,0.598)--(1.997,0.464)--(1.863,0.963);
\draw (1.7475,0.45) node{\tiny$e$};
\draw (-1.7475,-0.45) node{\tiny$e$};
\draw (2.05,0.7) node{\tiny$c$};
\draw (-2.05,-0.7) node{\tiny$c$};
\draw (1.7, 1.2) node{\tiny$a_{\scriptstyle\frac{k-1}{2}}$};
\draw (-1.7, -1.2) node{\tiny$a_{\scriptstyle\frac{k-1}{2}}$};

\draw (0.05,1.3) node{\tiny$d$};
\draw (-0.2,-1.3) node{\tiny$d$};

\draw (-0.6,-2) node {$(ii)$};
\draw[dashed] (-0.5,-0.866)--(-1.498,-0.598);
\draw[dashed] (-0.864,-1.231)--(-1.363,-1.097);
\draw (-1.363,-1.097)--(-1.498,-0.598)--(-1.863,-0.963)--(-1.363,-1.097);
\draw (-1.498,-0.598)--(-1.997,-0.464)--(-1.863,-0.963);

\end{scope}

\end{tikzpicture}
\caption{$(i)$ For $k$ is even \quad $(ii)$ For $k$ is odd}
\label{intermediate genus 1}
\end{center}
\end{figure}

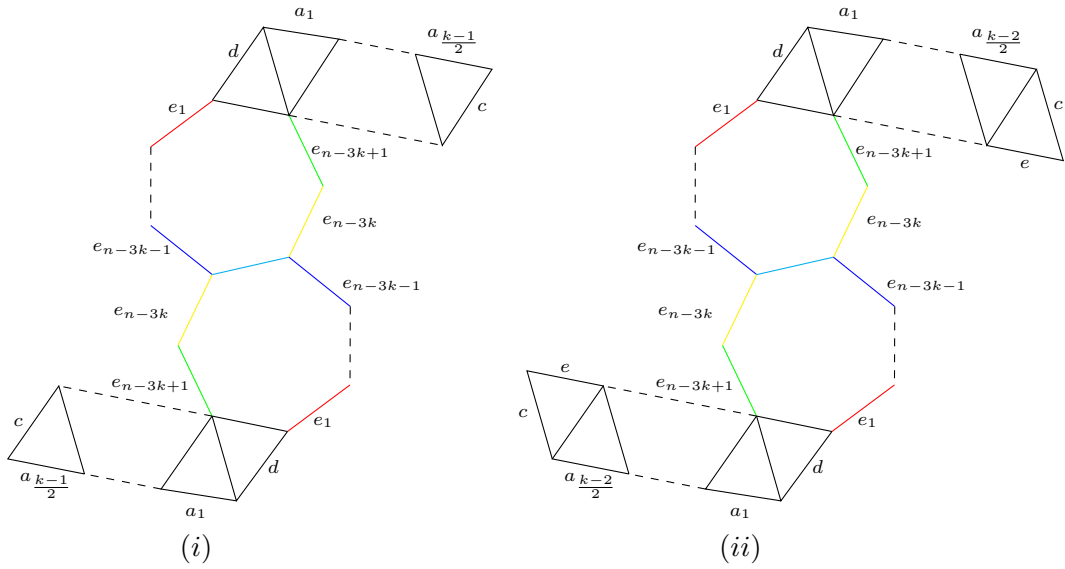
\begin{figure}[htbp]
\begin{center}
\begin{tikzpicture}[scale=1.2]

% --- Original Figure (i) ---
\begin{scope}[xshift=0 cm]
\draw (4.623,0.781)--(3.788,0.947);
\draw[red] (3.788,0.947)--(3.1,0.434);
\draw[dashed] (3.1,0.434)--(3.099,-0.433);
\draw[blue] (3.099,-0.433)--(3.777,-0.974);
\draw[cyan] (3.777,-0.974)--(4.622,-0.782);
\draw[yellow] (4.622,-0.782)--(4.999,0.001);
\draw[green] (4.999,0.001)--(4.623,0.781);
\draw[red] (5.295,-2.190)--(4.606,-2.703);
\draw (4.606,-2.703)--(3.775,-2.533);
\draw[green]  (3.775,-2.533)--(3.4,-1.753);
\draw[yellow] (3.4,-1.753)--(3.777,-0.974);
\draw[dashed] (5.298,-1.325)--(5.295,-2.190);
\draw[blue] (5.298,-1.325)--(4.622,-0.782);

\draw (3.4,0.85) node{\tiny$e_1$};
\draw (5,-2.6) node{\tiny$e_1$};
\draw (5.3,0.37) node{\tiny$e_{n-3k+1}$};
\draw (3.1,-2.2) node{\tiny$e_{n-3k+1}$};
\draw (5.3,-0.4) node{\tiny$e_{n-3k}$};
\draw (3,-1.4) node{\tiny$e_{n-3k}$};
\draw (2.9,-0.7) node{\tiny$e_{n-3k-1}$};
\draw (5.65,-1.1) node{\tiny$e_{n-3k-1}$};
\draw (3.778,0.947)--(4.340,1.751)--(4.623,0.781)--(5.171,1.621)--(4.340,1.751);
\draw (3.775,-2.533)--(4.043,-3.467)--(4.606,-2.703);
\draw (3.775,-2.533)--(3.213,-3.337)--(4.043,-3.467);
\draw[dashed] (4.623,0.781)--(6.312,0.449);
\draw[dashed] (5.171,1.621)--(6.015,1.455);

\draw (6.015, 1.455)--(6.860,1.289)--(6.312,0.449)--(6.015,1.455);
\draw[dashed] (3.775,-2.533)--(2.086,-2.201);
\draw[dashed] (3.213,-3.337)--(2.369, -3.171);
\draw (2.369,-3.171)--(2.086,-2.201)--(1.524,-3.005)--(2.369,-3.171);
\draw (6.4375, 1.6) node {\tiny$a_{\scriptstyle\frac{k-1}{2}}$};
\draw (1.9465,-3.3) node {\tiny$a_{\scriptstyle\frac{k-1}{2}}$};
\draw (6.75,0.869) node {\tiny$c$};
\draw (1.64, -2.603) node {\tiny$c$};
\draw (4.01,1.5) node {\tiny$d$};
\draw (4.45,-3.085) node {\tiny$d$};
\draw (4.8,1.9) node {\tiny$a_1$};
\draw (3.6,-3.6) node{\tiny$a_1$};
\draw (3.6,-4) node{$(i)$};
\end{scope}

% --- Shifted Copy (ii) ---
\begin{scope}[xshift=6 cm]
\draw (4.623,0.781)--(3.788,0.947);
\draw[red] (3.788,0.947)--(3.1,0.434);
\draw[dashed] (3.1,0.434)--(3.099,-0.433);
\draw[blue] (3.099,-0.433)--(3.777,-0.974);
\draw[cyan] (3.777,-0.974)--(4.622,-0.782);
\draw[yellow] (4.622,-0.782)--(4.999,0.001);
\draw[green] (4.999,0.001)--(4.623,0.781);
\draw[red] (5.295,-2.190)--(4.606,-2.703);
\draw (4.606,-2.703)--(3.775,-2.533);
\draw[green]  (3.775,-2.533)--(3.4,-1.753);
\draw[yellow] (3.4,-1.753)--(3.777,-0.974);
\draw[dashed] (5.298,-1.325)--(5.295,-2.190);
\draw[blue] (5.298,-1.325)--(4.622,-0.782);

\draw (3.4,0.85) node{\tiny$e_1$};
\draw (5,-2.6) node{\tiny$e_1$};
\draw (5.3,0.37) node{\tiny$e_{n-3k+1}$};
\draw (3.1,-2.2) node{\tiny$e_{n-3k+1}$};
\draw (5.3,-0.4) node{\tiny$e_{n-3k}$};
\draw (3,-1.4) node{\tiny$e_{n-3k}$};
\draw (2.9,-0.7) node{\tiny$e_{n-3k-1}$};
\draw (5.65,-1.1) node{\tiny$e_{n-3k-1}$};
\draw (3.778,0.947)--(4.340,1.751)--(4.623,0.781)--(5.171,1.621)--(4.340,1.751);
\draw (3.775,-2.533)--(4.043,-3.467)--(4.606,-2.703);
\draw (3.775,-2.533)--(3.213,-3.337)--(4.043,-3.467);
\draw[dashed] (4.623,0.781)--(6.312,0.449);
\draw[dashed] (5.171,1.621)--(6.015,1.455);
\draw (6.015, 1.455)--(6.860,1.289)--(6.312,0.449)--(6.015,1.455);

\draw (6.312,0.449)--(7.155,0.282)--(6.860,1.289);

\draw[dashed] (3.775,-2.533)--(2.086,-2.201);
\draw (2.086,-2.201)--(1.245,-2.034)--(1.524,-3.005);
\draw[dashed] (3.213,-3.337)--(2.369, -3.171);
\draw (2.369,-3.171)--(2.086,-2.201)--(1.524,-3.005)--(2.369,-3.171);

\draw (6.7335,0.25) node {\tiny$e$};
\draw (1.6655,-2) node {\tiny$e$};
\draw (7.1,0.9) node {\tiny$c$};
\draw (1.2,-2.5) node {\tiny$c$};
\draw (6.4375, 1.6) node {\tiny$a_{\scriptstyle\frac{k-2}{2}}$};
\draw (1.9465,-3.3) node {\tiny$a_{\scriptstyle\frac{k-2}{2}}$};

\draw (4.01,1.5) node {\tiny$d$};
\draw (4.45,-3.085) node {\tiny$d$};
\draw (4.8,1.9) node {\tiny$a_1$};
\draw (3.6,-3.6) node{\tiny$a_1$};
\draw (3.6,-4) node{$(ii)$};
\end{scope}

\end{tikzpicture}

\caption{$(i)$ For $k$ is odd \quad $(ii)$ For $k$ is even }
\label{intermediate genus 2}
\end{center}
\end{figure}

%As for now we find a class of graphs, \emph{wedge of $n$ circles}, $n\geq2$ that are embedded in such kind of configuration on translation surfaces and also give a complete picture on this.
%Now the question is whether there are some other class of graphs more than one vertex. We have obtained this answer and found a nice collection of non simple graphs. We find two such special classes as follows:

As of now, we have identified that the graphs $\Sigma_n$, for $n \geq 2$, admit cellular-systolic embedding on translation surfaces. It is natural to ask whether there exist further classes of graphs with more than one vertex that admit such kind of embedding on some translation surface. We answer this affirmatively and find another rich collection of non-simple graphs admitting cellular-systolic embedding. In particular, we define two special classes as follows:

\begin{enumerate}
    \item $\mathcal{P} = \left\{\, G_{n,m} = (V_n, E_{n,m}) \mid n \geq 2,\; m \in \mathbb{N} \right\}$, where 
    \begin{align*}
        V_n &= \{ v_0, v_1, \dots, v_{n-1} \} \text{ and}\\  E_{n,m} &= \{ e_{ij} \mid 0 \leq i \leq n-1,\; 1 \leq j \leq 2m \}.
    \end{align*}
    
     Note that $e_{ij}$ are the $2m$ parallel edges between~$v_i$ and~$v_{i+1}$, $j=1,2, \dots, 2m$ (indices of $v$'s taken modulo $n$). For example, Figure~\ref{P_family} illustrates the graphs $G_{3,1}$, $G_{4,1}$, $G_{5,2}$, and $G_{6,3}$.
\item $\mathcal{Q} = \left\{\, G_{n,m}^\prime = (V_n, E_{n,m}) \;\mid \; n \geq 2,\; m \in \mathbb{N}\right\}$, where 
\begin{align*}
V_n &= \left\{ v_0, v_1, \dots, v_{n-1}\right \}\text{ and}\\ E_{n,m} &= L \cup P, \text{ where}\\ L &= \{\, \ell_{ij} \mid 0 \leq i \leq n-1,\; 1 \leq j \leq m \} \text{ and}\\ P &= \{\, e_{ij} \mid 0 \leq i \leq n-1,\; 1 \leq j \leq m \},
\end{align*}
such that each $\ell_{ij}$ is a loop at vertex $v_i$ and each $e_{ij}$ connects vertex $v_i$ to vertex $v_{i+1}$ (indices taken modulo $n$). That is, every vertex $v_i$ has exactly $m$ loops and every pair $(v_i, v_{i+1})$ is connected by $m$ parallel edges. In Figure~\ref{Q_family}, $G_{3,1}^\prime, G_{4,2}^\prime, G_{5,3}^\prime$ and $G_{6,4}^\prime$ are depicted.
    
 \end{enumerate} 

\begin{figure}[htbp]
\begin{center}
\resizebox{\textwidth}{!}{%
\begin{tikzpicture}[thick, every node/.style={circle, fill=black, inner sep=2pt}]

% ---------------- (a) G_{3,1} ----------------
\begin{scope}[shift={(0,0)}]
  \foreach \i in {1,2,3} {
    \node (v\i) at (90+\i*120:2.5) {};
  }
  \foreach \i/\j in {1/2, 2/3, 3/1} {
    \draw (v\i) to[bend left=15] (v\j);
    \draw (v\i) to[bend right=15] (v\j);
  }
  \node[fill=none, inner sep=0pt] at (0,-4) {$(a)~G_{3,1}$};
\end{scope}

% ---------------- (b) G_{4,1} ----------------
\begin{scope}[shift={(7.5,0)}]
  \foreach \i in {1,2,3,4} {
    \node (v\i) at (45+\i*90:2.5) {};
  }
  \foreach \i/\j in {1/2, 2/3, 3/4, 4/1} {
    \draw (v\i) to[bend left=15] (v\j);
    \draw (v\i) to[bend right=15] (v\j);
  }
  \node[fill=none, inner sep=0pt] at (0,-4) {$(b)~G_{4,1}$};
\end{scope}

% ---------------- (c) G_{5,2} ----------------
\begin{scope}[shift={(0,-9)}]
  \foreach \i in {1,...,5} {
    \node (v\i) at (90+\i*72:2.5) {};
  }
  \foreach \i/\j in {1/2, 2/3, 3/4, 4/5, 5/1} {
    \draw (v\i) to[bend left=15] (v\j);
    \draw (v\i) to[bend right=15] (v\j);
    \draw (v\i) to[bend left=45] (v\j);
    \draw (v\i) to[bend right=45] (v\j);
  }
  \node[fill=none, inner sep=0pt] at (0,-4) {$(c)~G_{5,2}$};
\end{scope}

% ---------------- (d) G_{6,3} ----------------
\begin{scope}[shift={(7.5,-9)}]
  \foreach \i in {1,...,6} {
    \node (v\i) at (\i*60:2.5) {};
  }
  \foreach \i/\j in {1/2, 2/3, 3/4, 4/5, 5/6, 6/1} {
    \draw (v\i) to[bend left=15] (v\j);
    \draw (v\i) to[bend right=15] (v\j);
    \draw (v\i) to[bend left=35] (v\j);
    \draw (v\i) to[bend right=35] (v\j);
    \draw (v\i) to[bend left=55] (v\j);
    \draw (v\i) to[bend right=55] (v\j);
  }
  \node[fill=none, inner sep=0pt] at (0,-4) {$(d)~G_{6,3}$};
\end{scope}

\end{tikzpicture}%
}
\caption{Graphs $(a)~G_{3,1}$, $(b)~G_{4,1}$, $(c)~G_{5,2}$, and $(d)~G_{6,3}$ from the family $\mathcal{P}$.}
\label{P_family}
\end{center}
\end{figure}

\begin{figure}[htbp]
\begin{center}
\resizebox{\textwidth}{!}{%
\begin{tikzpicture}[thick, every node/.style={circle, fill=black, inner sep=2.5pt}]

% ---------------- (a) G'_{3,1} ----------------
\begin{scope}[shift={(0,0)}]
  \foreach \i in {1,2,3} {
    \node (v\i) at (90+\i*120:2.5) {};
  }
  \foreach \i/\j in {1/2, 2/3, 3/1} {
    \draw (v\i) -- (v\j);
  }
  \foreach \i in {1,2,3} {
    \pgfmathsetmacro{\ang}{90+\i*120}
    \draw (v\i) to[out=\ang-25, in=\ang+25, distance=1.5cm] (v\i);
  }
  \node[fill=none, inner sep=0pt] at (0,-4.5) {$(a)~G_{3,1}^\prime$};
\end{scope}

% ---------------- (b) G'_{4,2} ----------------
\begin{scope}[shift={(9,0)}]
  \foreach \i in {1,2,3,4} {
    \node (v\i) at (45+\i*90:2.5) {};
  }
  \foreach \i/\j in {1/2, 2/3, 3/4, 4/1} {
    \draw (v\i) to[bend left=15] (v\j);
    \draw (v\i) to[bend right=15] (v\j);
  }
  \foreach \i in {1,2,3,4} {
    \pgfmathsetmacro{\ang}{45+\i*90}
    \draw (v\i) to[out=\ang-20, in=\ang+20, distance=1.0cm] (v\i);
    \draw (v\i) to[out=\ang-30, in=\ang+30, distance=2.0cm] (v\i);
  }
  \node[fill=none, inner sep=0pt] at (0,-5.5) {$(b)~G_{4,2}^\prime$};
\end{scope}

% ---------------- (c) G'_{5,3} ----------------
\begin{scope}[shift={(0,-12)}]
  \foreach \i in {1,...,5} {
    \node (v\i) at (90+\i*72:2.5) {};
  }
  \foreach \i/\j in {1/2, 2/3, 3/4, 4/5, 5/1} {
    \draw (v\i) to[bend left=20] (v\j);
    \draw (v\i) -- (v\j);
    \draw (v\i) to[bend right=20] (v\j);
  }
  \foreach \i in {1,...,5} {
    \pgfmathsetmacro{\ang}{90+\i*72}
    \draw (v\i) to[out=\ang-15, in=\ang+15, distance=1.0cm] (v\i);
    \draw (v\i) to[out=\ang-25, in=\ang+25, distance=2.0cm] (v\i);
    \draw (v\i) to[out=\ang-35, in=\ang+35, distance=3.0cm] (v\i);
  }
  \node[fill=none, inner sep=0pt] at (0,-6.5) {$(c)~G_{5,3}^\prime$};
\end{scope}

% ---------------- (d) G'_{6,4} ----------------
\begin{scope}[shift={(9,-12)}]
  \foreach \i in {1,...,6} {
    \node (v\i) at (\i*60:2.5) {};
  }
  \foreach \i/\j in {1/2, 2/3, 3/4, 4/5, 5/6, 6/1} {
    \draw (v\i) to[bend left=15] (v\j);
    \draw (v\i) to[bend right=15] (v\j);
    \draw (v\i) to[bend left=45] (v\j);
    \draw (v\i) to[bend right=45] (v\j);
  }
  \foreach \i in {1,...,6} {
    \pgfmathsetmacro{\ang}{\i*60}
    \draw (v\i) to[out=\ang-15, in=\ang+15, distance=1.0cm] (v\i);
    \draw (v\i) to[out=\ang-25, in=\ang+25, distance=2.0cm] (v\i);
    \draw (v\i) to[out=\ang-35, in=\ang+35, distance=3.0cm] (v\i);
    \draw (v\i) to[out=\ang-45, in=\ang+45, distance=4.0cm] (v\i);
  }
  \node[fill=none, inner sep=0pt] at (0,-7.5) {$(d)~G_{6,4}^\prime$};
\end{scope}

\end{tikzpicture}%
}
\caption{Graphs $(a)~G_{3,1}^\prime$, $(b)~G_{4,2}^\prime$, $(c)~G_{5,3}^\prime$, and $(d)~G_{6,4}^\prime$ from the family $\mathcal{Q}$.}
\label{Q_family}
\end{center}
\end{figure}

   These graphs are nothing but $4m$-regular multigraphs. Now, we prove the following result.

\begin{repthm}[\ref{another rich collection}] (\cite{MR4985436})
Every graph $G\in \mathcal{P} \sqcup \mathcal{Q}$ admits a cellular-systolic embedding on some translation surface.   
\end{repthm}
\begin{proof}
We explicitly construct the corresponding translation surface for $G\in\mathcal{P} \sqcup \mathcal{Q}$.
Now, there are two cases here.

Case 1.
Assume $G\in \mathcal{P}$.
So, $G=G_{n,m}$ for some $n\ge 2, m\in \N$. Then, the vertices of $G$ are labeled $v_0, v_1, v_2, \dots, v_{n-1}$ so that there are $2m$ parallel edges joining the pair of vertices $(v_i, v_{i+1})$ mod $n$, $i= 0,1,2, \dots, n-1$. Consider the permutation $\sigma$ = ($0$ $1$ $2$ \dots $n-1$) in $S_n$. Now, we form our desired translation surface as follows:

First, we take $n$ numbers of regular $4m$ gons denoted by $P_0$, $P_1$, $P_2$, \dots, $P_{n-1}$. Then we label the sides of $P_i$ as $E_1(i), E_2(i), \dots, E_{2m}(i), \bar{E_1}(i), \bar{E_2}(i), \dots \bar{E}_{2m}(i)$ such that $E_k(i)$ is parallel to $\bar{E_k}(i)$, $1\leq k\leq 2m$ for each $i\in \{0,1,2 \dots, n-1\}$. Now for a fixed $i\in\{0, 1, 2, \dots, n-1\}$ we glue the sides $E_k(i)$ and $\bar{E_k}(\sigma(i))$, $1\leq k\leq 2m$  by translation. 

In this way we get a translation surface $S$ with $n$ singular points (or marked points) and each side of each $P_i$ becomes a systolic connection of $S$. So, it is clear that $\Gamma_S \cong G$. Also by the construction $S\setminus G$ is nothing but union of the $n$ number of regular $4m$ gons i.e. topological disks and hence $G$ admits cellular-systolic embedding on $S$.
 For example, Figure~\ref{figure16} illustrates the construction of a translation surface for $G_{4,1}$.

 \begin{figure}[htbp]
\begin{center}
\begin{tikzpicture}[scale=1]

\draw (-3,-4)--(-1,-4)--(-1,-2)--(-3,-2)--(-3,-4);
\draw (3,-4)--(1,-4)--(1,-2)--(3,-2)--(3,-4);

\draw (-3,-8)--(-1,-8)--(-1,-6)--(-3,-6)--(-3,-8);
\draw (3,-8)--(1,-8)--(1,-6)--(3,-6)--(3,-8);

\draw (-2,-4.2) node{\tiny$E_1(0)$};
\draw (-0.6, -3 ) node{\tiny$E_2(0)$};
\draw (-2,-1.8) node{\tiny$\bar{E_1}(0)$};
\draw (-3.4,-3) node{\tiny$\bar{E_2}(0)$};

\draw (2,-1.8) node{\tiny$\bar{E_1}(1)$};
\draw (0.59,-3) node{\tiny$\bar{E_2}(1)$};
\draw (-2,-8.2) node{\tiny$E_1(3)$};
\draw (-0.6,-7) node{\tiny$E_2(3)$};

\draw (2,-4.2) node{\tiny$E_1(1)$};
\draw (3.4,-3) node{\tiny$E_2(1)$};
\draw (2,-5.8) node{\tiny$\bar{E_1}(2)$};
\draw (0.6,-7) node{\tiny$\bar{E_2}(2)$};
\draw (2,-8.2) node{\tiny$E_1(2)$};
\draw (3.4,-7) node{\tiny$E_2(2)$};
\draw (-2,-5.8) node{\tiny$\bar{E_1}(3)$};
\draw (-3.4,-7) node{\tiny$\bar{E_2}(3)$};

\draw[green] (-3,-4) node{\tiny$\bullet$};
\draw[green] (-1,-2) node{\tiny$\bullet$};
\draw[green] (1,-2) node{\tiny$\bullet$};
\draw[green] (-1,-8) node{\tiny$\bullet$};

\draw[blue] (-1,-4) node{\tiny$\bullet$};
\draw[blue] (1,-4) node{\tiny$\bullet$};
\draw[blue] (3,-2) node{\tiny$\bullet$};
\draw[blue] (1,-6) node{\tiny$\bullet$};

\draw (-3,-2) node{\tiny$\bullet$};
\draw (-3,-8) node{\tiny$\bullet$};
\draw (-1,-6) node{\tiny$\bullet$};
\draw (3,-8) node{\tiny$\bullet$};

\draw[red] (-3,-6) node{\tiny$\bullet$};
\draw[red] (1,-8) node{\tiny$\bullet$};
\draw[red] (3,-6) node{\tiny$\bullet$};
\draw[red] (3,-4) node{\tiny$\bullet$};

\end{tikzpicture}
\caption{Translation surface for $G_{4,1}$.}
\label{figure16}
\end{center}
\end{figure}
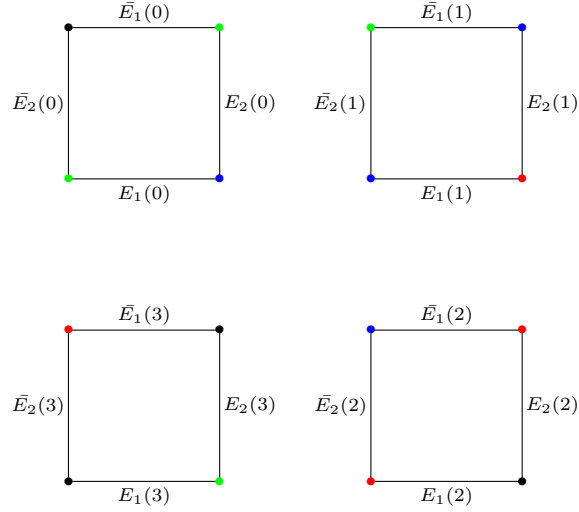

Case 2.
Assume $G\in \mathcal{Q}$.
This implies that $G=G_{n,m}^\prime$ for some $n\ge2, m\in \N$.
 So, label the vertices of $G$ as $v_0$, $v_1$, $v_2$, \dots, $v_{n-1}$ such that,
 (1) each vertex $v_i$, $i= 0,1,\dots, n-1$ has $m$ loops and  
 (2) there are $m$ parallel edges joining the pair of vertices ($v_i$, $v_{i+1}$) mod $n$, $0\le i\le n-1$.
Also, in this case we  consider the permutation $\sigma$ = ($0$ $1$ $2$ \dots $n-1$) from $S_n$. Similarly to case $1$, we take $n$ numbers of regular $4m$ gons denoted by $P_0$, $P_1$, $P_2$, \dots, $P_{n-1}$. Label the sides of $P_i$ as $E_1(i), E_2(i), \dots, E_{2m}(i), \bar{E_1}(i), \bar{E_2}(i), 
\dots \bar{E}_{2m}(i)$ such that $E_k(i)$ is parallel to $\bar{E_k}(i)$, $1\leq k\leq 2m$ for each $i\in \{0,1,2 \dots, n-1\}$. Here, we form the translation surface as follows:

For a fixed $i\in \{0,1,2 \dots, n-1\}$ we identify the sides $E_k(i)$ and $\bar{E}_k(i)$  for $1\le k\le m$ and the sides $E_k(i)$ and $\bar{E}_k(\sigma(i))$ for $m+1\le k\le 2m$ via translation. So, we get a translation surface $S$ from this with $n$ number of singular points (or marked points) and each side of $P_i$'s is systolic connection of $S$.
\begin{figure}[htbp]
\begin{center}
\begin{tikzpicture}[scale=1.5]
\draw (0,-5)--(1,-5);
\draw (1,-5)--(1.707,-4.293);
\draw (1.707,-4.293)--(1.707,-3.293);
\draw  (1.707,-3.293)--(1,-2.586);
\draw (0, -2.586)--(1,-2.586);
\draw (0,-2.586)--(-0.707,-3.293);
\draw (-0.707,-3.293)--(-0.707,-4.293);
\draw (-0.707,-4.293)--(0,-5);

\draw (0.5,-5.2) node{\tiny$E_1(0)$};
\draw (1.8,-4.7) node{\tiny$E_2(0)$};
\draw (2.1,-3.7) node{\tiny$E_3(0)$};
\draw (1.7,-2.8) node{\tiny$E_4(0)$};
\draw (0.5,-2.4) node{\tiny$\bar{E_1}(0)$};
\draw (-0.7,-2.8) node{\tiny$\bar{E_2}(0)$};
\draw (-1.1,-3.8) node{\tiny$\bar{E_3}(0)$};
\draw (-0.7,-4.8) node{\tiny$\bar{E_4}(0)$};

\draw[green] (0,-5) node{\tiny$\bullet$};
\draw[green] (1,-5) node{\tiny$\bullet$};
\draw[green] (1.707,-4.293) node{\tiny$\bullet$};
\draw[blue] (1.707,-3.293) node{\tiny$\bullet$};
\draw[green] (1,-2.586) node{\tiny$\bullet$};
\draw[green] (0,-2.586) node{\tiny$\bullet$};
\draw[green] (-0.707,-3.293) node{\tiny$\bullet$};
\draw (-0.707,-4.293) node{\tiny$\bullet$};

\draw (5,-5)--(6,-5);
\draw (6,-5)--(6.707,-4.293);
\draw (6.707,-4.293)--(6.707,-3.293);
\draw  (6.707,-3.293)--(6,-2.586);
\draw (5, -2.586)--(6,-2.586);
\draw (5,-2.586)--(4.293,-3.293);
\draw (4.293,-3.293)--(4.293,-4.293);
\draw (4.293,-4.293)--(5,-5);

\draw (5.5,-5.2) node{\tiny$E_1(1)$};
\draw (6.8,-4.7) node{\tiny$E_2(1)$};
\draw (7.1,-3.7) node{\tiny$E_3(1)$};
\draw (6.7,-2.8) node{\tiny$E_4(1)$};
\draw (5.5,-2.4) node{\tiny$\bar{E_1}(1)$};
\draw (4.3,-2.8) node{\tiny$\bar{E_2}(1)$};
\draw (3.9,-3.8) node{\tiny$\bar{E_3}(1)$};
\draw (4.3,-4.8) node{\tiny$\bar{E_4}(1)$};

\draw[blue] (5,-5) node{\tiny$\bullet$};
\draw[blue] (6,-5) node{\tiny$\bullet$};
\draw[blue] (6.707,-4.293) node{\tiny$\bullet$};
\draw[red] (6.707,-3.293) node{\tiny$\bullet$};
\draw[blue] (6,-2.586) node{\tiny$\bullet$};
\draw[blue] (5,-2.586) node{\tiny$\bullet$};
\draw[blue] (4.293,-3.293) node{\tiny$\bullet$};
\draw[green] (4.293,-4.293) node{\tiny$\bullet$};

\draw (5,-10)--(6,-10);
\draw (6,-10)--(6.707,-9.293);
\draw (6.707,-9.293)--(6.707,-8.293);
\draw  (6.707,-8.293)--(6,-7.586);
\draw (5, -7.586)--(6,-7.586);
\draw (5,-7.586)--(4.293,-8.293);
\draw (4.293,-8.293)--(4.293,-9.293);
\draw (4.293,-9.293)--(5,-10);

\draw (5.5,-10.2) node{\tiny$E_1(2)$};
\draw (6.8,-9.7) node{\tiny$E_2(2)$};
\draw (7.1,-8.7) node{\tiny$E_3(2)$};
\draw (6.7,-7.8) node{\tiny$E_4(2)$};
\draw (5.5,-7.4) node{\tiny$\bar{E_1}(2)$};
\draw (4.3,-7.8) node{\tiny$\bar{E_2}(2)$};
\draw (3.9,-8.8) node{\tiny$\bar{E_3}(2)$};
\draw (4.3,-9.8) node{\tiny$\bar{E_4}(2)$};

\draw[red] (5,-10) node{\tiny$\bullet$};
\draw[red] (6,-10) node{\tiny$\bullet$};
\draw[red] (6.707,-9.293) node{\tiny$\bullet$};
\draw (6.707,-8.293) node{\tiny$\bullet$};
\draw[red] (6,-7.586) node{\tiny$\bullet$};
\draw[red] (5,-7.586) node{\tiny$\bullet$};
\draw[red] (4.293,-8.293) node{\tiny$\bullet$};
\draw[blue] (4.293,-9.293) node{\tiny$\bullet$};

\draw (0,-10)--(1,-10);
\draw (1,-10)--(1.707,-9.293);
\draw (1.707,-9.293)--(1.707,-8.293);
\draw  (1.707,-8.293)--(1,-7.586);
\draw (0, -7.586)--(1,-7.586);
\draw (0,-7.586)--(-0.707,-8.293);
\draw (-0.707,-8.293)--(-0.707,-9.293);
\draw (-0.707,-9.293)--(0,-10);

\draw (0.5,-10.2) node{\tiny$E_1(3)$};
\draw (1.8,-9.7) node{\tiny$E_2(3)$};
\draw (2.1,-8.7) node{\tiny$E_3(3)$};
\draw (1.7,-7.8) node{\tiny$E_4(3)$};
\draw (0.5,-7.4) node{\tiny$\bar{E_1}(3)$};
\draw (-0.7,-7.8) node{\tiny$\bar{E_2}(3)$};
\draw (-1.1,-8.8) node{\tiny$\bar{E_3}(3)$};
\draw (-0.7,-9.8) node{\tiny$\bar{E_4}(3)$};

\draw (0,-10) node{\tiny$\bullet$};
\draw (1,-10) node{\tiny$\bullet$};
\draw (1.707,-9.293) node{\tiny$\bullet$};
\draw[green] (1.707,-8.293) node{\tiny$\bullet$};
\draw (1,-7.586) node{\tiny$\bullet$};
\draw (0,-7.586) node{\tiny$\bullet$};
\draw (-0.707,-8.293) node{\tiny$\bullet$};
\draw[red] (-0.707,-9.293) node{\tiny$\bullet$};

\end{tikzpicture}
\caption{Translation surface for $G_{4,2}^\prime$.}
\label{figure17}
\end{center}
\end{figure}
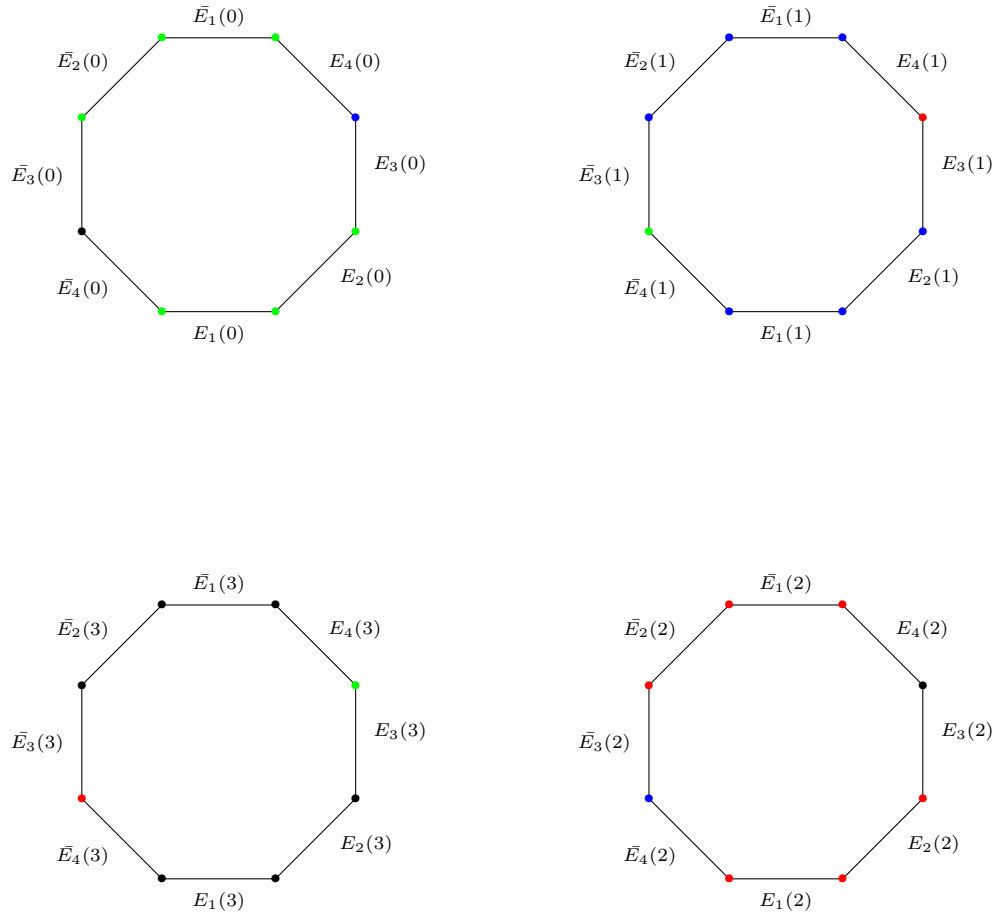

So, $G$ admits cellular-systolic embedding on $S$. In Figure \ref{figure17}, a translation surface is explicitly constructed for $G_{4,2}^\prime$. 

The genus of the translation surface \(S\) constructed above can be computed from its induced cell decomposition. In both Cases~$1$ and $2$, we have  \(V = n\),  polygonal faces \(F = n\), and since each vertex has degree \(4m\), the number of edges is \(E =\frac{4mn}{2} =2mn\). Hence, the Euler characteristic is 
\[
\chi(S) = V - E + F = 2n(1 - m).
\]
Using \(\chi(S) = 2 - 2g\) for a closed orientable surface, we obtain 
\[
g = n(m - 1) + 1.
\]
So, the genus of the translation surface for both cases is $ [n(m - 1) + 1]$.

Therefore, from both the cases we conclude that any $G\in \mathcal{P} \sqcup \mathcal{Q}$ admits cellular-systolic embedding on a translation surface.

\end{proof}

\chapter[The Geometry of the Hyperbolic Plane]{The Geometry of the Hyperbolic Plane}

\label{Chapter 4}

This chapter presents a concise overview of hyperbolic geometry, which lays the groundwork for the topics developed in the subsequent chapters. We begin by outlining the historical background and foundational development of hyperbolic geometry. Various equivalent formulations, referred to as models of hyperbolic geometry, have been proposed over time. Among these, we introduce a selection of standard models, with particular emphasis on the upper half-plane model. 
Following this, we examine fundamental geometric notions within this model, including the arc-length function, the hyperbolic distance, the hyperbolic area, and the Gauss-Bonnet theorem. We also explore the isometry group associated with the upper half-plane model. The chapter concludes with a summary of important identities from hyperbolic trigonometry, which will be instrumental in later discussions.

\section{Origins and Evolution of Hyperbolic Geometry}

In antiquity, Euclidean geometry played a central role in various practical domains such as land surveying, architectural design, boundary demarcation, and navigation. Around 300 BCE, the Greek mathematician Euclid was the first to systematically develop planar geometry using a rigorous axiomatic approach. His influential treatise \emph{Elements} formalized geometric principles through a set of five foundational postulates (see~\cite{euclid2002elements}):

\begin{enumerate}
    \item Any two distinct points can be connected by a unique straight line segment.
    \item A straight line segment can be extended indefinitely in both directions.
    \item Given any point and any radius, a circle can be constructed with that point as the center.
    \item All right angles are congruent to one another.
    \item If a straight line intersects two other straight lines $L$ and $L'$ such that the sum of the interior angles $\alpha$ and $\beta$ on the same side is less than $\pi$, then the lines $L$ and $L'$ will intersect on that side when extended sufficiently.
\end{enumerate}

Among Euclid’s five postulates, the fifth—commonly known as the \emph{parallel postulate}—has historically been regarded as the most intricate and least self-evident. Over the centuries, it has been reformulated in various equivalent ways. One well-known equivalent statement, derived from the first four postulates, is the following:

\begin{enumerate}
    \item[5$'$.] Given a line and a point not lying on it, there exists exactly one line through the point that does not intersect the given line; in other words, there is a unique line parallel to the given one through the external point.
\end{enumerate}

A significant portion of Euclidean plane geometry can be developed without reference to the parallel postulate. In fact, this postulate does not appear until Proposition 29 in Book I of Euclid’s \emph{Elements}~\cite{euclid2002elements}. This observation led many to suspect that the fifth postulate might be deducible from the other four. For centuries, countless mathematicians pursued such a derivation—though all attempts ultimately failed.

A major turning point came in the nineteenth century, when mathematicians began to consider the possibility of rejecting the parallel postulate altogether. Rather than trying to prove it, they examined the consequences of its negation. This novel perspective led to the birth of entirely new geometric systems, now referred to as \emph{non-Euclidean geometries}.

Carl Friedrich Gauss (1777–1855), Nikolai Lobachevsky (1793–1856), and János Bolyai (1802–1860) were instrumental in this development. In a letter dated 1824 to F. A. Taurinus, Gauss noted that assuming the sum of the angles in a triangle is less than $180^\circ$ results in a geometry distinct from the Euclidean framework. Indeed, the parallel postulate is logically equivalent to the statement that the sum of the angles in any triangle equals $180^\circ$. Although Gauss explored these ideas in private, he chose not to publish his findings.

Independently, Lobachevsky (in 1829) and Bolyai (in 1832) succeeded in constructing a consistent axiomatic system that excluded the parallel postulate. Their formulations gave rise to what we now call \emph{hyperbolic geometry}—a name derived from one of its canonical models, the \emph{hyperboloid model} (see~\cite{CannonHypGeom} for details).

%Hyperbolic geometry was created in the first half of the nineteenth century amid attempts to understand Euclid’s axiomatic basis for geometry.

\section{Standard Models of Hyperbolic Plane}

Hyperbolic geometry, an example of non-Euclidean geometry characterized by constant negative curvature, admits multiple distinct but equivalent models that facilitate its visualization and analysis. In this section, we explore four classical representations: the upper half-plane model, the Poincar\'e disk model, the Klein disk model, and the hyperboloid model. For each, we examine the underlying structure, the associated metric, the nature of geodesics, and the relationships between the models.

\subsection{The Upper Half-Plane Model}

Consider the upper half of the complex plane, defined below:
\[
\mathbb{H} = \{ z = x + iy \in \mathbb{C} \mid y > 0 \}.
\]
Equipped with the Riemannian metric
\[
\mathrm{d}s^2 = \frac{\mathrm{d}x^2 + \mathrm{d}y^2}{y^2}=\frac{4dz d\bar{z}}{|z-\bar{z}|^2},
\]
the set $\mathbb{H}$ defines the \emph{upper half-plane model} of the hyperbolic plane. This metric endows $\mathbb{H}$ with the structure of a complete, simply connected Riemannian surface of constant Gaussian curvature equal to $-1$.

In the upper half-plane model, geodesics are represented by vertical lines and semicircles orthogonal to the real axis.  The boundary at infinity is given by $\partial \mathbb{H} = \mathbb{R} \cup \{\infty\}$. A fundamental distinction from Euclidean geometry arises in the behavior of parallel lines: given a point $P$ not lying on a geodesic $L$, there exist infinitely many geodesics through $P$ that do not intersect $L$, illustrating the breakdown of Euclidean parallelism in hyperbolic geometry. Hence, the hyperbolic geometry is a non-Euclidean geometry.

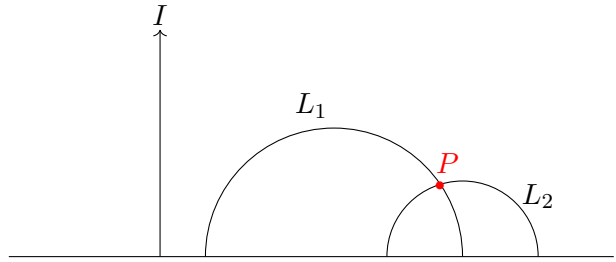
\begin{figure}[htbp]
    \centering
    \begin{tikzpicture}
         \draw (-4,0)--(4,0);
         \draw [->](-2,0)--(-2,3);
         \draw (3,0) arc [start angle=0, end angle=180, x radius=1 cm, y radius=1 cm];
         \draw (2,0) arc [start angle=0, end angle=180, x radius=1.7 cm, y radius=1.7 cm];
         \draw (-2,3.2) node {$I$};
         \draw (3,0.8) node {$L_2$}; 
         \draw (0,2) node {$L_1$};
         \draw [red] (1.7, 0.94) node {\tiny $\bullet$}; 
         \draw [red] (1.8, 0.97) node [above] {$P$};
    \end{tikzpicture}
    \caption{Geodesics $L_1$ and $L_2$ in $\mathbb{H}$ are parallel to the vertical geodesic $I$ through $P$.}
\end{figure}

\subsection{The Poincar\'e Disk Model}

The \emph{Poincar\'e disk model} represents hyperbolic geometry inside the unit disk
\[
\mathbb{D} = \{ z \in \mathbb{C} : |z| < 1 \},
\]
endowed with the metric
\[
\mathrm{d}s = \frac{2 |\mathrm{d}z|}{1 - |z|^2}.
\]
Here, geodesics are either diameters or circular arcs intersecting the boundary circle orthogonally. 

The upper half-plane model and the Poincar\'e disk model are isometric. 
An explicit isometry 
\[
f:\left(\mathbb{H},\,\frac{4\,\mathrm{d}z\, \mathrm{d}\bar{z}}{|z-\bar{z}|^2}\right)\longrightarrow 
\left(\mathbb{D},\,\frac{4\,\mathrm{d}z\, \mathrm{d}\bar{z}}{(1 - |z|^2)^2}\right)
\]

is given by
\[
f(z) = \frac{z - i}{\,z + i\,}, \quad \forall z \in \mathbb{H}.
\]

To prove this claim, it suffices to show that 
\begin{equation}\label{eqn 4.2.1}
    \frac{\mathrm{d}z\, \mathrm{d}\bar{z}}{|z-\bar{z}|^2}
    = \frac{\mathrm{d}f(z)\, \overline{\mathrm{d}f(z)}}{(1 - |f(z)|^2)^2}.
\end{equation}

\vspace{0.2cm}

Now, observe that
\[
\mathrm{d}f(z)=\frac{2i}{(z+i)^2}\,\mathrm{d}z, 
\qquad 
\overline{\mathrm{d}f(z)}=-\frac{2i}{(\bar{z}-i)^2}\,\mathrm{d}\bar{z},
\qquad \text{and} \qquad
1 - |f(z)|^2=\frac{2|z-\bar{z}|}{|z+i|^2}.
\]
Substituting these expressions into the right-hand side of 
\eqref{eqn 4.2.1} yields the desired identity.

\begin{figure}[htbp]
    \centering
    \begin{tikzpicture}
        \draw (0,0) circle [radius=2cm];
        \draw ({2*cos(140)},{2*sin(140)})--({-2*cos(140)},{-2*sin(140)});
        \draw ({2*cos(0)},{2*sin(0)}) to [bend left=55] ({2*cos(70)},{2*sin(70)});
        \draw ({2*cos(40)},{2*sin(40)}) to [bend left=60] ({2*cos(120)},{2*sin(120)});
        \draw [red] (.76,.85) node {\tiny$\bullet$};
        \draw [red] (.6,.6) node {$P$};
        \draw (1.1,.1) node {$L_1$}; 
        \draw (-.3,.8) node {$L_2$}; 
        \draw (-.3,0) node {$L$};
    \end{tikzpicture}
    \caption{Geodesics in the Poincar\'e disk: arcs and diameters orthogonal to the boundary.}
\end{figure}
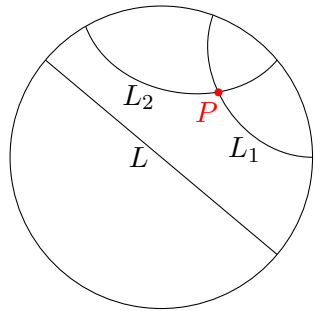

\subsection{The Klein Disk Model}

Like the Poincar\'e model, the Klein model is defined on the open unit disk $\mathbb{D}$ and is denoted by $\mathbb{D}_K$. We define the Klein metric on this model by pushing forward the metric of Poincar\'e disk model via the homeomorphism $
K : \mathbb{D} \to \mathbb{D}_K$ given by $$ K(z) = \frac{2z}{1+|z|^2}, \quad \forall z\in\mathbb{D}.
$$
that is, 
\begin{align}
ds_K^2 &= \frac{4\,dK(z)\,d\overline{K(z)}}{\bigl(1-|K(z)|^2\bigr)^2},  \\
 &= \frac{(1-|z|^2)\,dz\,d\bar z + \frac{1}{4}(\bar z\,dz + z\,d\bar z)^2}{(1-|z|^2)^2}.
\end{align}

In this model, distances are computed using the cross-ratio. Let us first recall the definition of cross-ratio as follow.

\begin{definition}[Cross ratio on the Riemann sphere]
Let $z_1,z_2,z_3,z_4 \in \widehat{\C} = \C \cup \{\infty\}$ be four distinct points. 
The \emph{cross ratio} of $(z_1,z_2;z_3,z_4)$ is defined by
\[
[z_1,z_2;z_3,z_4] = 
\begin{cases}
\dfrac{(z_1-z_4)(z_2-z_3)}{(z_1-z_3)(z_2-z_4)}, & \text{if } z_i \in \C, \\[2ex]
\dfrac{z_1-z_4}{z_1-z_3}, & \text{if } z_2 = \infty, \\[2ex]
\dfrac{z_2-z_3}{z_2-z_4}, & \text{if } z_1 = \infty, \\[2ex]
\dfrac{z_2-z_3}{z_1-z_3}, & \text{if } z_4 = \infty, \\[2ex]
\dfrac{z_1-z_4}{z_2-z_4}, & \text{if } z_3 = \infty.
\end{cases}
\]
\end{definition}

\begin{lemma} (Proposition $6.2$  in \cite{bridson1999metric})
Given two points $z_1, z_2$ in $\mathbb{D}_K$, let $w_1$ and $w_2$ be the ideal endpoints of the chord through them (in order $w_1, z_1,z_2$ and $w_2).$ Then the distance between $z_1$ and $z_2$ is
\[
d_{\mathbb{K}}(z_1, z_2) = \frac{1}{2} \log \left(| [z_1,z_2;w_1,w_2]| \right).
\]
\end{lemma}

In this model, geodesics are straight-line segments (chords) (see Figure~\ref{Klein model}).
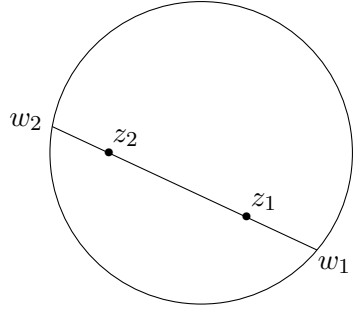
\begin{figure}[H]
    \centering
    \begin{tikzpicture}
        \draw (0,0) circle [radius=2cm];
        \draw ({2*cos(320)},{2*sin(320)})--({2*cos(170)},{2*sin(170)});
        \draw (.6,-.85) node {\tiny$\bullet$}; 
        \draw (-1.22,0) node {\tiny$\bullet$};
        \draw (.83,-.65) node {$z_1$}; 
        \draw (-1,.2) node {$z_2$}; 
        \draw ({2.3*cos(320)},{2.3*sin(320)}) node {$w_1$}; 
        \draw ({2.35*cos(170)},{2.35*sin(170)}) node {$w_2$};
    \end{tikzpicture}
    \caption{Straight-line geodesics in the Klein model.}
    \label{Klein model}
\end{figure}

\subsection{The Hyperboloid Model}

Let $\mathbb{R}^{2,1}$ be $\mathbb{R}^3$ with the Lorentzian inner product
\[
\langle \mathbf{x}, \mathbf{y} \rangle =  x_1 y_1 + x_2 y_2- x_3 y_3.
\]
Define the upper sheet of the hyperboloid:
\[
\mathcal{H} = \{ \mathbf{x} = (x_1, x_2, x_3) \in \mathbb{R}^3 : \langle \mathbf{x}, \mathbf{x} \rangle = -1,\ x_3 > 0 \}.
\]
Equipped with the induced metric
\[
\mathrm{d}s^2 = \mathrm{d}x_1^2 + \mathrm{d}x_2^2 - \mathrm{d}x_3^2,
\]
this forms the hyperboloid model of the hyperbolic plane. Geodesics are intersections of $\mathcal{H}$ with planes through the origin.

The map $h\colon\mathcal{H}  \to \mathbb{D}$ defined by

\[
h(x_1, x_2, x_3) = \left( \frac{x_1}{1 + x_3}, \frac{x_2}{1 + x_3} \right)
\]
defines an isometry from the hyperboloid to the Poincar\'e disk (see Theorem 3.7.5 in \cite{Beardon}). The isometry group of this model is the subgroup of $\mathrm{GL}(3, \mathbb{R})$ preserving the Lorentz form and the upper sheet.

\begin{figure}[htbp]
    \centering
    \begin{tikzpicture}
        \fill[gray!20!white] (-3-1,-1.5)--(1+1,-1.5)--(3+1,1.5)--(-1-1,1.5)--cycle;
        \draw (-3-1,-1.5)--(1+1,-1.5)--(3+1,1.5)--(-1-1,1.5)--cycle;
        \draw[very thick] [->-=1] (0,0)--(0,5);
        \draw[->-=1] [very thick] (0,0)--(4,0);
        \draw [->-=1, very thick] (0,0)--(-2,-2);
        \draw (-2.8,3.5)..controls(0,.5)..(2.8,3.5);
        
        % \draw (-2.5,3) arc
        % [
        %     start angle=180,
        %     end angle=360,
        %     x radius=2.5cm,
        %     y radius =2cm
        % ] ;
        \draw  (-2.3,3) arc
        [
            start angle=180,
            end angle=360,
            x radius=2.3cm,
            y radius =.25cm
        ] ;
        \draw [ dashed] (2.3,3) arc
        [
            start angle=0,
            end angle=180,
            x radius=2.3cm,
            y radius =.25cm
        ] ;
        \draw (3.5,-.3) node {$x_2$} (-1.5,-1.9) node {$x_1$} (.35,4.7) node {$x_3$} (.7,.8) node {$(0,0,1)$};
    \end{tikzpicture}
    
    \caption{Hyperboloid model}
    \label{fig: Hyperboloid model}
\end{figure}
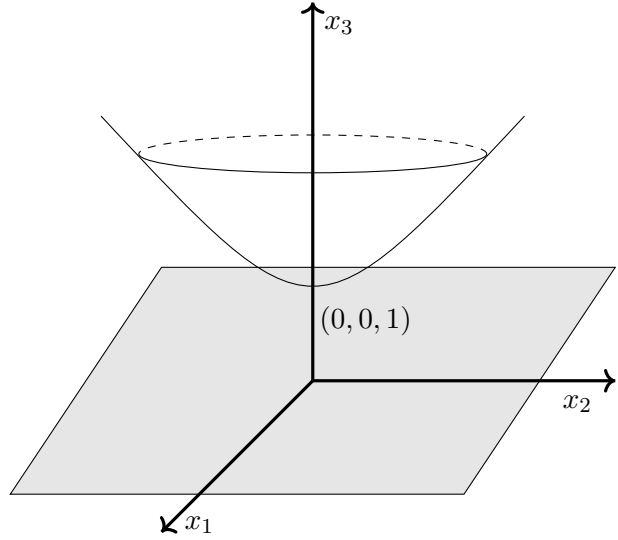

\section{Insights into the Upper Half-Plane Model}

This section presents a comprehensive study of the upper half-plane model of hyperbolic geometry. We begin by defining the hyperbolic arc length, followed by a classification of isometries and an analysis of their actions. Building on the notion of arc length, we then introduce the hyperbolic distance function. Finally, we describe the geodesics, the curves that minimize hyperbolic distance, and provide explicit formulas for calculating distances between points.

\subsection{Measuring Arc-Length in Hyperbolic Geometry}

Let $I=[a,b]$ be a closed interval in $\mathbb{R}.$ 
The hyperbolic length of a piecewise smooth arc \( \alpha : I \to \mathbb{H}\) where $\alpha(t)=x(t)+iy(t), \forall t\in I$, is denoted by \( l_\mathbb{H}(\alpha) \) or simply \( l(\alpha) \) and is defined by
\[
l_\mathbb{H}(\alpha) = \int_a^b \frac{\sqrt{\left( \frac{dx}{dt} \right)^2 + \left( \frac{dy}{dt} \right)^2}}{y(t)}\, dt.
\]

\begin{lemma}
This length is invariant under reparameterization, meaning that if \( \alpha: [a, b] \to \mathbb{H} \) and \( \beta: [a', b'] \to \mathbb{H} \) are such that \( \alpha = \beta \circ h \) for a diffeomorphism \( h: [a, b] \to [a', b'] \), then \( l_\mathbb{H}(\alpha) = l_\mathbb{H}(\beta) \).
\end{lemma}

\begin{proof}
Suppose  $\alpha(t) = x(t) + i y(t), \quad \forall\, t \in [a,b],$ and $\beta(t) = x^\prime(t) + i y^\prime(t), \quad \forall\, t \in [a^\prime, b^\prime].$
It is given that $\alpha(t) = (\beta \circ h)(t), \quad \forall\, t \in [a,b].$
This implies 
\[
x(t) = x^\prime(h(t)) \quad \text{and} \quad y(t) = y^\prime(h(t)).
\]

Now,
\begin{align*} l_\mathbb{H}(\alpha) &= \int_a^b \frac{\sqrt{\left( \frac{dx}{dt} \right)^2 + \left( \frac{dy}{dt} \right)^2}}{y(t)}\, dt \\ &=\int_a^b \frac{\sqrt{\left( \frac{dx^\prime(h(t)}{dt} \right)^2 + \left( \frac{dy^\prime(h(t))}{dt} \right)^2}}{y^\prime(h(t))}\, \frac{dh(t)}{dt}dt \\ &= \int_a^b \frac{\sqrt{\left( \frac{dx^\prime(h(t)}{dt} \right)^2 + \left( \frac{dy^\prime(h(t))}{dt} \right)^2}}{y^\prime(h(t))}\, dh(t)\\ &= \int_{a^\prime}^{b^\prime} \frac{\sqrt{\left( \frac{dx^\prime}{dt} \right)^2 + \left( \frac{dy^\prime}{dt} \right)^2}}{y^\prime(t)}\, dt\\ &= l_\mathbb{H}(\beta). \end{align*}
\end{proof}

\subsection{Isometries in \texorpdfstring{$\mathbb{H}$}{H}}

\begin{definition}
An \emph{isometry} of $\mathbb{H}$ is a map $f: \mathbb{H} \to \mathbb{H}$ such that
\[
d_{\mathbb{H}}(f(x), f(y)) = d_{\mathbb{H}}(x, y), \quad \text{for all } x, y \in \mathbb{H}.
\]
\end{definition}

\begin{definition}
A differentiable map $f: \mathbb{H} \to \mathbb{H}$ is said to be \emph{orientation-preserving} if
\[
\det(df_p) > 0 \quad \text{for all } p \in \mathbb{H}.
\]
\end{definition}

The set of all isometries of $\mathbb{H}$ forms a group, denoted by $\mathrm{Isom}(\mathbb{H})$. The subgroup consisting of orientation-preserving isometries is denoted by $\mathrm{Isom}^+(\mathbb{H})$. 

Next, we characterize the structure of these groups. In particular, we show that the group $\mathrm{PSL}(2,\mathbb{R})$ describes all orientation-preserving isometries of the upper half-plane, while the full isometry group includes both orientation-preserving and orientation-reversing transformations.

We first recall the definition of the special linear group:
\[
\mathrm{SL}(2,\mathbb{R}) \;=\; \left\{ 
\begin{pmatrix}
a & b \\ c & d
\end{pmatrix} : a,b,c,d \in \mathbb{R}, \; ad-bc=1
\right\}.
\]

Every element 
\[
g = \begin{pmatrix} a & b \\ c & d \end{pmatrix} \in \mathrm{SL}(2,\mathbb{R})
\]
acts on the upper half-plane 
\[
\mathbb{H} = \{ z \in \mathbb{C} : \Im(z) > 0 \}
\]
via a Möbius transformation,
\[
g(z) \;=\; \frac{az+b}{cz+d}, \qquad \forall z \in \mathbb{H}.
\]

This action is not faithful as the matrices $g$ and $-g$ induce the same transformation of $\mathbb{H}$. 
The kernel of the action is therefore the subgroup $\{\pm I_2\}$, where $I_2$ is the $2 \times 2$ identity matrix. 
Passing to the quotient removes this redundancy, and one defines the projective special linear group
\[
\mathrm{PSL}(2,\mathbb{R}) \;=\; \mathrm{SL}(2,\mathbb{R}) / \{\pm I_2\},
\]
which is precisely the group of distinct Möbius transformations with real coefficients.

 Every Möbius transformation corresponds to a pair of matrices $\pm g \in \mathrm{SL}(2,\mathbb{R})$, and composition of transformations corresponds to matrix multiplication.

Hence, $\mathrm{PSL}(2,\mathbb{R})$ is isomorphic to the group of all orientation-preserving isometries of the upper half-plane $\mathbb{H}$.

\begin{theorem}
    $\p=\mathrm{Isom}^{+}(\mathbb{H})$.
\end{theorem}
\begin{proof}
    The hyperbolic metric for the upper half-plane can be expressed as 
    $$ds^2=\frac{4dzd\overline{z}}{|z-\overline{z}|^2}.$$
    Thus, to prove that a transformation $T\in \p$ is an isometry of $\mathbb{H}$, it is sufficient to show that
    \begin{equation}\label{eqn 2.3.1}
        \frac{dT(z)d\overline{T(z)}}{|T(z)-\overline{T(z)}|^2}=\frac{dzd\overline{z}}{|z-\overline{z}|^2}.
    \end{equation}
    Let $T(z) = \frac{az+b}{cz+d}, 
    \quad a,b,c,d \in \mathbb{R} \quad \text{and}\quad \; ad-bc=1.$
 Then  
\begin{align*}
\overline{T(z)} &= \frac{a\overline{z}+b}{c\overline{z}+d}, \\[6pt]
dT(z) &= \frac{dz}{(cz+d)^2}, \\[6pt]
d\overline{T(z)} &= \frac{d\overline{z}}{(c\overline{z}+d)^2} \quad \text{and} \\[6pt]
T(z) - \overline{T(z)} &= \frac{z-\overline{z}}{|cz+d|^2}.
\end{align*}

    Using these relations, Equation \eqref{eqn 2.3.1} now easily follows.

    For the reverse direction, let $\phi \in \mathrm{Isom}^{+}(\mathbb{H})$ and $I=\{ti:t>0\}$. Then $\phi(I)$ is a geodesic in the upper half plane. Take $A\in \p$ such that $A\circ \phi(I)=I$. Suppose $A\circ \phi(i)=\lambda i$, for some $\lambda>0$. Then $B\circ A \circ \phi(i)=i$, where $B\in \p$ is defined by $B(z)=\frac{1}{\lambda}z$. Suppose $B\circ A \circ \phi$ fixes $0,i,\infty$. If not, we compose it by the map $T\in \p$, where $T(z) =-\frac{1}{z}$ , for all $z\in \h$. As $B\circ A\circ\phi$ is an isometry, it fixes $I$ pointwise. Now, we show that $B\circ A\circ\phi$ is the identity map. Let $z$ be an arbitrary point of $\h$. Consider the geodesic $\gamma$ passing through $z$, which meets $I$ perpendicularly. As $B\circ A\circ\phi$ is angle-preserving and orientation-preserving, $B\circ A\circ\phi(\gamma)=\gamma$ and in particular $B\circ A\circ\phi(z)=z$. Therefore, $B\circ A\circ\phi=id\implies \phi= A^{-1}B^{-1} \in \p$.
\end{proof}

\begin{theorem}[Theorem 7.4.1 \cite{Beardon}]
    The group $\mathrm{Isom}(\h)$ of the upper half-plane model is generated by $\p$ and $z\mapsto -\overline{z}$.
\end{theorem}

\subsection{Classification of isometries}
In this section, we study the behavior of the action of the elements $\mathrm{Isom}^+(\h)$ on $\h$. Our initial classification is based on the fixed points of the elements of $\mathrm{Isom}^+(\h)$. 

Let $T\in \mathrm{Isom}^+(\h)$, where $T(z)=\frac{az+b}{cz+d}$. We can extend $T$ continuously to $\overline{\h}$ , where $\overline{\h}= \h\cup \partial\h$  by following:

When $c=0$,
\[ T(z)=\begin{cases} 
      \frac{a}{d}z+\frac{c}{d}, & \text{ for }z\in \overline{\h}\setminus \{\infty\} \quad \text{and} \\
      \infty, & \text{ for }z=\infty, 
   \end{cases}
\]
and in case $c\neq 0$, we have 
\[ T(z)=\begin{cases} 
      \frac{az+b}{cz+d}, & \text{ for }z\in \overline{\h}\setminus \{\infty, -\frac{d}{c}\}, \\
      \infty, & \text{ for }z=-\frac{d}{c} \quad \text{and} \\
      \frac{a}{c}, & \text{ for }z=\infty.
   \end{cases}
\]

A point $z\in \overline{\h}$ is a fixed point of $T$ if $T(z)=z$. By Brouwer's fixed point theorem (\cite{hatcher2002}, Theorem $1.9)$, it follows that $T$ always has at least one fixed point in $\overline{\h}$.

Let $T\in \mathrm{Isom}^+(\h)$ with $T(z)=\frac{az+b}{cz+d}, ad-bc=1$. Then $z_0=\infty$ is a fixed point of $T$ if and only if $c=0$. In this case, $T(z)=\frac{a}{d}z+\frac{b}{d}$. If $a=d$, then $z_0=\infty$ is the only fixed point of $T$. If $a\neq d$, then $T$ has two fixed points in $\overline{\h}$, namely $\infty$ and $-\frac{c}{a}$.

If $c\neq 0$, the fixed points of $T$ are the solution of the equation $T(z)=z$ or equivalently $cz^2-(a-d)z-b=0$ which are given below:
$$z_0=\frac{a-d+\sqrt{tr(T)^2-4}}{2c},\  \frac{a-d-\sqrt{tr(T)^2-4}}{2c},$$
where $tr(T)=|a+d|$. 

\begin{definition}
     Two transformations $T_1, T_2$ are called conjugate in $\mathrm{Isom}^+(\h)$, if there exists an element $p\in\mathrm{Isom}^+(\h)$ such that $T_2=pT_1p^{-1}$.
\end{definition}
\begin{theorem}[Theorem 2.31 \cite{Anderson}]
\begin{enumerate}
\item $T$ has exactly one fixed point in $\partial\h$ and none in $\h$ if and only if $tr(T)=2$. In this case, $T$ is conjugate to a translation $z\mapsto z+b$, where $b\in \{1, -1\}$

\item $T$ has exactly one fixed point in $\h$ and none in $\partial\h$ if and only if $tr(T)<2.$ In this case, $T$ is conjugate to a hyperbolic rotation $z\mapsto \frac{\cos\theta z-\sin\theta}{\sin\theta z+\cos\theta}$. 
    
\item The transformation $T$ has exactly two fixed point in $\partial\h$ and none in $\h$ if and only if $tr(T)>2.$ In this case, $T$ is conjugate to a dilation $z\mapsto \lambda^2z$, for some $\lambda\in \mathbb{R}_{>1}.$
    \end{enumerate}
\end{theorem}
Now, we have the following classification.

\begin{definition}
    Let $T\in \mathrm{Isom}^+(\h)\setminus \{\mathrm{Id}\}$, where $T(z)=\frac{az+b}{cz+d}$, $a, b, c, d\in \mathbb{R}$, with $ad-bc=1$. Then
    \begin{enumerate}
        \item $T$ is called elliptic if $tr(T)<2.$
        \item $T$ is called parabolic if $tr(T)=2.$
        \item $T$ is called hyperbolic if $tr(T)>2.$
    \end{enumerate}
\end{definition}

\subsection{ Hyperbolic Distance Function}
Based on the arc length, the hyperbolic distance between two points is defined as follows.

\begin{definition}
The hyperbolic distance \( d_\mathbb{H}(z, w) \) between two points \( z, w \in \mathbb{H} \) is defined by
\[
d_\mathbb{H}(z, w) = \inf \left\{ l_\mathbb{H}(\alpha) \,\middle|\, \alpha: [0,1] \to \mathbb{H} \text{ is a piece-wise smooth path from } z \text{ to } w \right\}.
\]
\end{definition}

\begin{lemma}
    $d_\mathbb{H}$ is a distance function on $\mathbb{H}.$ 
\end{lemma}

\begin{proof}
Here, $
\ell_{\mathbb H}(\alpha) \;=\int_0^1 \frac{|\alpha'(t)|}{\Im(\alpha(t))}\,dt$, where $\Im(\alpha(t)$ is the imaginary part of $\alpha(t).$

Now, we verify the metric axioms.

\medskip
\noindent\textbf{(i) Non-negativity.}
For every path $\alpha$ we have $|\alpha'(t)|\ge 0$ and $\Im(\alpha(t))>0$, hence
$\ell_{\mathbb H}(\alpha)\ge 0$. Taking the infimum gives $d_{\mathbb H}(z,w)\ge 0$.

\medskip
\noindent\textbf{(ii) Symmetry.}
If $\alpha$ joins $z$ to $w$, then the reversed path $\tilde\alpha(t)=\alpha(1-t)$
joins $w$ to $z$ and satisfies $\ell_{\mathbb H}(\tilde\alpha)=\ell_{\mathbb H}(\alpha)$.
Thus $d_{\mathbb H}(z,w)=d_{\mathbb H}(w,z)$.

\medskip
\noindent\textbf{(iii) Triangle inequality.}
Let $z,w,u\in\mathbb H$ and fix $\varepsilon>0$.
Choose paths $\alpha$ from $z$ to $w$ and $\beta$ from $w$ to $u$ with
\[
\ell_{\mathbb H}(\alpha) \le d_{\mathbb H}(z,w)+\tfrac\varepsilon2,
\qquad
\ell_{\mathbb H}(\beta) \le d_{\mathbb H}(w,u)+\tfrac\varepsilon2.
\]
Concatenating $\alpha$ and $\beta$ gives a path $\gamma$ from $z$ to $u$ with
\[
\ell_{\mathbb H}(\gamma)
= \ell_{\mathbb H}(\alpha)+\ell_{\mathbb H}(\beta)
\le d_{\mathbb H}(z,w)+d_{\mathbb H}(w,u)+\varepsilon.
\]
Taking the infimum and letting $\varepsilon\to 0$ yields
\[
d_{\mathbb H}(z,u) \le d_{\mathbb H}(z,w)+d_{\mathbb H}(w,u).
\]

\medskip
\noindent\textbf{(iv) Definiteness.}
 Suppose $z,w\in\mathbb H$ and $z\neq w$. We show $d_{\mathbb H}(z,w)>0$.

The group $\mathrm{PSL}(2,\mathbb R)$ of real Möbius transformations acts transitively on $\mathbb H$ by isometries for the Poincaré line element $|dz|/\Im z$. Hence there exists $T\in\mathrm{PSL}(2,\mathbb R)$ with $T(z)=i$. Put $w':=T(w)$. Since $T$ is an isometry,
\[
d_{\mathbb H}(z,w)=d_{\mathbb H}(T(z),T(w))=d_{\mathbb H}(i,w').
\]

If $w'=i$ then $w=z$ (because $T$ is injective), contrary to $z\neq w$. Thus $w'\neq i$.
Consider the stabilizer subgroup of $i$ in $\mathrm{PSL}(2,\mathbb R)$; this subgroup is isomorphic to $\mathrm{SO}(2)$ and acts by rotations about $i$ (in the hyperbolic sense). In particular it preserves the hyperbolic circle
\[
\{u\in\mathbb H:\ d_{\mathbb H}(u,i)=d_{\mathbb H}(w',i)\},
\]
and acts transitively on that circle. Therefore we can find $R$ in the stabilizer of $i$ such that the image $w'':=R(w')$ satisfies $\Im w''\neq 1$. (Intuitively: the orbit of $w'$ under the stabilizer is a nontrivial circle centered at $i$, and some point of that circle has imaginary part different from $1$.)

Now $S:=R\circ T\in\mathrm{PSL}(2,\mathbb R)$ is an isometry sending $z$ to $i$ and $w$ to $w''$, so
\[
d_{\mathbb H}(z,w)=d_{\mathbb H}(i,w'').
\]

Let $\gamma:[0,1]\to\mathbb H$ be any piecewise-smooth path from $i$ to $w''$. Writing $\gamma(t)=x(t)+i y(t)$ with $y(t)=\Im\gamma(t)>0$, we have
\[
|\gamma'(t)| \;\ge\; |y'(t)| \quad\text{for all }t,
\]
hence
\[
\frac{|\gamma'(t)|}{\Im\gamma(t)} \;\ge\; \frac{|y'(t)|}{y(t)}.
\]
Integrating and using the triangle inequality for integrals,
\[
\ell_{\mathbb H}(\gamma)
=\int_0^1\frac{|\gamma'(t)|}{\Im\gamma(t)}\,dt
\ge \int_0^1\frac{|y'(t)|}{y(t)}\,dt
\ge \Big|\int_0^1\frac{y'(t)}{y(t)}\,dt\Big|
= \big|\ln y(1)-\ln y(0)\big|.
\]
Since $\gamma(0)=i$ and $\gamma(1)=w''$ we get
\[
\ell_{\mathbb H}(\gamma) \ge \big|\ln(\Im w'')-\ln(\Im i)\big|
= \big|\ln(\Im w'')\big|.
\]
By choice $\Im w''\neq 1$, so the right-hand side is strictly positive. Because the inequality holds for every admissible path $\gamma$ from $i$ to $w''$, taking the infimum yields
\[
d_{\mathbb H}(i,w'') \ge \big|\ln(\Im w'')\big| >0.
\]
Thus $d_{\mathbb H}(z,w)=d_{\mathbb H}(i,w'')>0$. This proves that if $z\neq w$ then $d_{\mathbb H}(z,w)>0$.

Conversely, if $z=w$ the constant path shows $d_{\mathbb H}(z,z)=0$. Combining both directions, we obtain
\[
d_{\mathbb H}(z,w)=0 \iff z=w,
\]
as required.

\medskip
Combining (i)–(iv), we conclude that $d_{\mathbb H}$ is a metric on $\mathbb H$.
\end{proof}

Now, it is straightforward to see that \( d_\mathbb{H} \) satisfies the properties of a metric: positive definiteness, symmetry, and the triangle inequality. Moreover, the space \( (\mathbb{H}, d_\mathbb{H}) \) is a path metric space. In fact, for any two points \( z, w \in \mathbb{H} \), there exists a unique path \( \alpha: [0,1] \to \mathbb{H} \) connecting them such that \( l_\mathbb{H}(\alpha) = d_\mathbb{H}(z, w) \). Such a path is called a geodesic segment joining \( z \) and \( w \), and \( (\mathbb{H}, d_\mathbb{H}) \) is thus a uniquely geodesic metric space.

\subsection{Geodesics in \texorpdfstring{$\mathbb{H}$}{H}}
We now describe the geodesics in \( \mathbb{H} \), as stated in Theorem~\ref{geod}. We begin with an illustrative example.

\begin{example}
Let \( z = ia \) and \( w = ib \in \mathbb{H} \) with \( b \geq a>0 \). Define the path \( \sigma: [a, b] \to \mathbb{H} \) by \( \sigma(t) = it \), for all $t\in[a,b]$. This path connects \( z \) and \( w \), and its hyperbolic length is
\[
l_{\mathbb{H}}(\sigma) = \int_a^b \frac{1}{t} \, dt = \ln\left( \frac{b}{a} \right).
\]
Now let \( \gamma: [0, 1] \to \mathbb{H} \), be any path from \( z \) to \( w \) given by \( \gamma(t) = x(t) + i y(t) \), $\forall t\in[0,1]$ i.e., \( \gamma(0) = z \), \( \gamma(1) = w \). Then
\begin{align*}
l_\mathbb{H}(\gamma) &= \int_0^1 \frac{\sqrt{x'(t)^2 + y'(t)^2}}{y(t)} \, dt \\
&\geq \int_0^1 \frac{|y'(t)|}{y(t)} \, dt \\
&\geq \left| \int_0^1 \frac{y'(t)}{y(t)} \, dt \right| = \left| \ln\left( \frac{b}{a} \right) \right| = \ln\left( \frac{b}{a} \right).
\end{align*}

Hence, \( d_\mathbb{H}(z, w) = \ln\left( \frac{b}{a} \right) \). Equality holds if and only if \( x'(t) = 0 \) and \( y'(t) > 0 \), which implies \( \gamma \) and \( \sigma \) represent the same vertical path (up to reparameterization). Therefore, the positive imaginary axis is a geodesic in \( \mathbb{H} \).
\end{example}

Now consider a transformation \( T: \mathbb{H} \to \mathbb{H} \) defined by
\[
T(z) = \frac{az + b}{cz + d}, \quad \text{where } a, b, c, d \in \mathbb{R} \text{ and } ad - bc = 1.
\]
Such transformations preserve hyperbolic lengths, that is, \( l_\mathbb{H}(\alpha) = l_\mathbb{H}(T \circ \alpha) \). Thus, they also preserve hyperbolic distances and map geodesics to geodesics.

The group of such transformations acts transitively on the set of vertical lines and semicircles orthogonal to the real axis. Consequently, the complete set of geodesics in \( \mathbb{H} \) is described by the following result.

\begin{theorem}[Theorem 1.2.1, \cite{Katok}] \label{geod}
The geodesics in \( \mathbb{H} \) are precisely the vertical straight lines and the semicircles orthogonal to the real axis.
\end{theorem}

The theorem follows from the fact that the group of all transformations $z\mapsto \frac{az+b}{cz+d}$, where $a,b,c,d\in \mathbb{R}$ satisfying $ad-bc=1$ acts transitively on the set of all vertical lines and semi-circles with center on real line and the positive imaginary axis is a hyperbolic geodesic.

\subsection{Hyperbolic Distance Formulas}

The Theorem below provides several explicit formulas for computing the distance between two points in the upper half-plane model of hyperbolic geometry.

\begin{theorem}[Theorem 1.2.6 \cite{Katok}]
Let $z, w \in \mathbb{H}$. Then the hyperbolic distance $d_{\mathbb{H}}(z, w)$ satisfies the following identities:
\begin{enumerate}
    \item $d_{\mathbb{H}}(z, w) = \ln\left( \frac{|z - \overline{w}| + |z - w|}{|z - \overline{w}| - |z - w|} \right)$.
    \item $\cosh d_{\mathbb{H}}(z, w) = 1 + \dfrac{|z - w|^2}{2\, \mathrm{Im}(z)\, \mathrm{Im}(w)}$.
    \item $\sinh\left( \dfrac{1}{2} d_{\mathbb{H}}(z, w) \right) = \dfrac{|z - w|}{2\sqrt{\mathrm{Im}(z)\, \mathrm{Im}(w)}}$.
    \item $\cosh\left( \dfrac{1}{2} d_{\mathbb{H}}(z, w) \right) = \dfrac{|z - \overline{w}|}{2\sqrt{\mathrm{Im}(z)\, \mathrm{Im}(w)}}$.
    \item $\tanh\left( \dfrac{1}{2} d_{\mathbb{H}}(z, w) \right) = \left| \dfrac{z - w}{z - \overline{w}} \right|$.
\end{enumerate}
\end{theorem}

\section{Area of hyperbolic triangles}
In this section, our objective is to compute the area of hyperbolic triangles, which will serve as a foundation for determining the area of more general hyperbolic polygons. We start by introducing the formal definition of area for a region in the hyperbolic plane. 
\begin{definition}
    The area of a region $A\subset \h$, denoted by $\mathrm{Area}_\h(A)$, is defined by
    $$\mathrm{Area}_\h(A)=\int_A \frac{dxdy}{y^2},$$
    if the integral exists.
\end{definition}
\begin{theorem}[Theorem 1.4.2 \cite{Katok}]\label{thm: Gauss Bonnet}
    Let $\Delta$ be a hyperbolic triangles with angles $\alpha, \beta$ and $\gamma$. Then 
    $$\mathrm{Area}_\h(\Delta)=\pi -(\alpha+\beta+\gamma).$$
\end{theorem}

\begin{corollary}
    Let $P$ be an $n$-sided hyperbolic polygon with interior angles $\theta_1, \dots, \theta_n$. Then 
    $$\mathrm{Area}_\h(P)=(n-2)\pi-(\theta_1+ \dots + \theta_n).$$
\end{corollary}
\begin{proof}
    Decompose the polygon into triangles and use Theorem \ref{thm: Gauss Bonnet}.
\end{proof}

\section{Trigonometry in Hyperbolic Geometry}
In this section, we study hyperbolic trigonometric identities for hyperbolic triangles which will be frequently used in the upcoming chapters. We begin with the definition of hyperbolic trigonometric functions. Then we discuss some identities satisfied by the hyperbolic functions which are easy to prove. 
\subsection{Hyperbolic trigonometric functions}
The hyperbolic cosine, sine and tan functions are defined as follows:
\begin{align*}
\cosh{x}&=\frac{e^x+e^{-x}}{2},\\  \sinh{x}&=\frac{e^x-e^{-x}}{2} \quad\text{and}\\
\tanh{x}&=\frac{\sinh{x}}{\cosh{x}}=\frac{e^x-e^{-x}}{e^x+e^{-x}}
\end{align*}
As in the Euclidean geometry other hyperbolic trigonometric functions $\csch, \sech, \coth$ and their inverses are defined in a similar manner.  
\subsection{Hyperbolic angles}
Let  $\alpha:(-\epsilon,\epsilon)\to \h$ and $\beta:(-\epsilon,\epsilon)\to \h$ be two differentiable curves in $\h$ with $\alpha(0)=\beta(0)=p$ and $\dot{\alpha}(0), \dot{\beta}(0)\neq0$, where $p$ is a point of $\h$. The angle between $\alpha$ and $\beta$ at the point $p$ is defined by
$$\angle_p(\alpha,\beta)=\frac{\langle\dot{\alpha}(0),\dot{\beta}(0)\rangle_p}{||\dot{\alpha}(0)||_p\cdot ||\dot{\beta}(0)||_p},$$ where $\langle(u_1, v_1)),(u_2, v_2)\rangle_p= \frac{1}{\mathrm{Im}(p)}(u_1v_1+u_2v_2)$, for $(u_1, v_1)),(u_2, v_2)\in T_p(\h)\cong \mathbb{R}^2$, is the hyperbolic inner product. 
It is easy to see that the angle between two curves in the upper half-plane model is same as the Euclidean angle.

\subsection{Hyperbolic trigonometric formula}
Consider a hyperbolic triangle with interior angles $\alpha,\beta$ and $\gamma$. Let the lengths of the side opposite to $\alpha,\beta$ and $\gamma$ be $a,b$ and $c$, respectively (see Figure~\ref{fig:triangle}$(i)$)

\begin{figure}[htbp]
      \centering
      \begin{tikzpicture}
          \draw (-.5,1) [bend left=20] to (-1.9,-1);
          \draw (-1.9,-1) [bend left=20] to (1.9,-1);
          \draw (1.9,-1) [bend left=20] to (-.5,1);
          \draw (-.45,.55) node {\tiny$\gamma$};
          \draw (-1.2,-.6) node {\tiny$\alpha$};
          \draw (.5,-.53) node {\tiny$\beta$};
          \draw (-.5,-.85) node {\small$c$};
          \draw (-1.2,-0) node {\small$b$};
          \draw (.4,.1) node {\small$a$};

          \draw (3+1.5,-1)--(5.5+1.5,-1);
          \draw (3+1.5,1)--(3+1.5,-1);
          \draw (3+1.5,1) to [bend right] (5.5+1.5,-1);
          \draw (2.8+1.5,0) node {\small$a$} (4.25+1.5,-1.2) node {\small$b$} (3.85+1.5,0) node{\small$c$} (4.65,.3) node {\tiny$\beta$} (5.9,-.85) node {\tiny$\alpha$};

          \draw (0,-1.8) node {$(i)$} (5.75,-1.8) node {$(ii)$};
      \end{tikzpicture}
      \caption{Hyperbolic triangles}
      \label{fig:triangle}
  \end{figure}
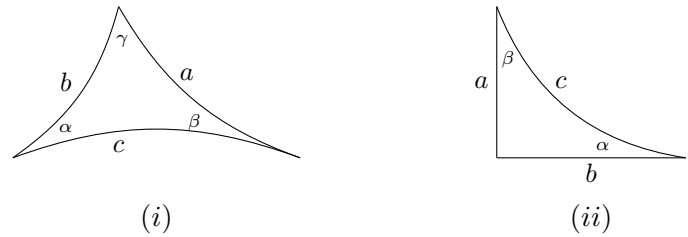

 %(for more details, see Section 2.2 \cite{buser2010geometry}).

 Analogous to Euclidean geometry, we have the hyperbolic sine and cosine rule in the hyperbolic plane.
\begin{theorem}[Theorem 2.2.1 \cite{buser2010geometry}] \label{Cosine Rule}
    \begin{enumerate}
        \item The Sine Rule: $\frac{\sinh{a}}{\sin{\alpha}}=\frac{\sinh{b}}{\sin{\beta}}=\frac{\sinh{c}}{\sin{\gamma}}$.
        \item The Cosine Rule I: $\cosh c= \cosh a \cosh b-\sinh a \sinh b \cos \gamma$
        \item The Cosine Rule II: $\cos \gamma=\sin \alpha \sin \beta \cosh c-\cos \alpha \cos \beta$
    \end{enumerate}
\end{theorem}

  Next, we consider a right-angled hyperbolic triangle, where $\gamma=\frac{\pi}{2}$ in Figure \ref{fig:triangle}$(ii)$ (for details, see Theorem 2.2.2 \cite{buser2010geometry}). Then the following hold.
  \begin{theorem}\label{pythagorus}
      \begin{enumerate}
          \item $\cosh c=\cosh a \cosh b$.
          \item $\sinh a=\sin \alpha \sinh c$.
          \item $\sinh{a}= \cot \beta \tanh b.$ 
      \end{enumerate}
  \end{theorem}
  The formula in item 1 of Theorem~\ref{pythagorus} is known as the hyperbolic Pythagoras formula. 

%\section{Hyperbolic triangles of type \texorpdfstring{$(p,q,r)$}{}}
%\begin{definition}
 %   Given integers $p,q,r\geq 3$, a triangle of type $(p,q,r)$ is a hyperbolic triangle with angles $\frac{\pi}{p},\frac{\pi}{q}$ and $\frac{\pi}{r}$.
%\end{definition}

\chapter[Hyperbolic Surfaces]{Hyperbolic Surfaces}

\label{Chapter 5}
In this chapter, we study both the topology and geometry of surfaces. We begin by introducing hyperbolic structures on surfaces and analyzing their geometric features. 
We then review the mapping class group, which provides a framework for classifying self-homeomorphisms of a surface. 
Next, we turn to Teichm\"uller space, describing how it parametrizes hyperbolic structures and allows for their comparison. We conclude with a brief discussion of the moduli space and the role of Morse functions defined on it.

\section{Hyperbolic Surfaces}

A \emph{topological surface} is defined as a two-dimensional topological manifold.  
By a \emph{closed surface} we mean a compact surface without boundary.
Every closed orientable surface is obtained by a connected sum of a sphere $S^{2}$ with $g$ many  tori $T^{2}$, where $g\ge0$.  
Conversely, for any closed surface $S$, there exists a non-negative integer $g$ for which
\[
S \cong S^{2} \,\#\, \underbrace{(T^{2} \,\#\, \cdots \,\#\, T^{2})}_{g \ \text{times}},
\]
where $g$ is referred to as the \emph{genus} of $S$.  

A compact surface with boundary can be obtained from a closed surface by removing $b \ge0$ open discs whose closures do not intersect one another.  
If, instead, one deletes $n \ge0$ points from the interior of a compact surface, the result is called a \emph{non-compact surface}.  
The triple $(g,b,n)$, $g\ge0, b\ge0$ and $n\ge0$ is called the \emph{signature} of such a surface, and a surface described in this way is said to be of \emph{finite type}.  
Its Euler characteristic is given by
\[
\chi(S) = 2 - 2g - b - n.
\]
In this thesis, by a surface we always mean orientable surface. 

\begin{definition} (\cite{buser2010geometry})
A hyperbolic structure on a closed surface $S$ consists of an atlas of charts $\{U_i,\phi_i\}_{i\in I}$ such that $\phi_i(U_i)\subset \mathbb{H}$ and the transition maps
\[
\phi_j \circ \phi_i^{-1} : \phi_i(U_i \cap U_j) \to \phi_j(U_i \cap U_j),~ \text{for}~ i, j\in I ~ \text{with}~ U_i\cap U_j \neq \emptyset
\]
are restrictions of orientation-preserving isometries of $\mathbb{H}$.
\end{definition}

Now to introduce hyperbolic structure on surfaces with boundary, we begin with some terminology.  
For $\delta > 0$, $\theta \in (0,2\pi)$, and $p_0 \in \mathbb{H}$, the \emph{circle sector} of radius $\delta$ and angle $\theta$ at $p_0$ is defined by
\[
V(p_0;\delta,\theta) := \{(p,\alpha)\in \mathbb{H} \mid 0 < p < \delta, \; 0 < \alpha < \theta \} \cup \{p_0\},
\]
where $(p,\alpha)$ are polar coordinates centered at $p_0$ ( see Section $1.1$ of \cite{buser2010geometry}).  
A \emph{half-disk} is a circle sector with angle $\pi$.  

Let $S$ be a surface with non-ejmpty piecewise smooth boundary $\partial S$. Its \emph{interior} is
\[
\operatorname{int}(S) := S \setminus \partial S.
\]
The boundary $\partial S$ is a union of smooth arcs, called \emph{sides}, each having zero, one, or two endpoints. A side with no endpoints is a closed boundary curve. Endpoints of sides are called \emph{vertices}, while points of $\partial S$ that are not vertices are \emph{ordinary boundary points}. We allow the boundary to be smooth at certain vertices, and the set of vertices is assumed locally finite.

\begin{definition} (\cite{buser2010geometry})
 A hyperbolic structure on a surface $S$ with non-empty boundary consists of an atlas $\phi_i(U_i)\subset \mathbb{H}$ satisfying the following:
 
\begin{enumerate}
    \item For each $p\in S$ there exists co-ordinate system $(U, \phi)$ in the atlas with $p \in U$ such that $\phi(U)\subset\mathbb{H}$ is
    \begin{itemize}
        \item  a circle sector of angle $\le \pi$, if $p$ is a vertex,

        \item a half-disk at $\phi(p)$, if $p$ is an ordinary boundary point,

        \item an open disk with center $\phi(p)$, if $p$ is an interior point.
    \end{itemize}

    \item The transition maps
\[
\phi_j \circ \phi_i^{-1} : \phi_i(U_i \cap U_j) \to \phi_j(U_i \cap U_j),~ \text{for}~ i, j\in I ~ \text{with}~ U_i\cap U_j \neq \emptyset
\]
are restrictions of orientation-preserving isometries of $\mathbb{H}$.
\end{enumerate}
\end{definition}

\begin{definition}
A surface endowed with a hyperbolic structure is called a  hyperbolic surface.
\end{definition}

\begin{note}
Throughout this thesis, all surfaces are assumed to be complete.
\end{note}

\begin{example}
The upper half-plane $\mathbb{H}$ itself carries a hyperbolic surface structure.  
Indeed, the atlas $\{(\mathbb{H}, \mathrm{id})\}$, where $\mathrm{id} : \mathbb{H} \to \mathbb{H}$ is the identity map, provides the required hyperbolic charts.
\end{example}

A more interesting family of hyperbolic surfaces can be obtained by gluing together hyperbolic polygons along their sides using isometries of $\mathbb{H}$.  
Let $P_1, \dots, P_m$ be hyperbolic polygons in the upper half-plane with total number of sides $2n,~ n\in\mathbb{N} $, and define $P = P_1 \cup \cdots \cup P_m$, equipped with a side-pairing.

Suppose $2n$ distinct sides of $P$ are labeled as $
\gamma_1, \overline{\gamma}_1, \dots, \gamma_n, \overline{\gamma}_n,$ where each pair $(\gamma_k, \overline{\gamma}_k)$ is assigned opposite orientations.  
For each $k = 1, \dots, n$, assume that both sides are parametrized by
\[
\gamma_k : I_k \to P, 
\qquad 
\overline{\gamma}_k : I_k \to P,
\]
with $I_k$ a closed interval.  
In particular, the parametrization guarantees that $\gamma_k$ and $\overline{\gamma}_k$ have the same hyperbolic length.

The identification rule
\[
\gamma_k(t) \sim \overline{\gamma}_k(t), 
\qquad t \in I_k, \quad k = 1, \dots, n,
\]
Glues each $\gamma_k$ to its partner $\overline{\gamma}_k$ via an orientation-preserving hyperbolic isometry, that is,h $\gamma_k$ is paired with $\gamma^\prime_k$.  
Carrying out these identifications produces an orientable surface $S$.  

The collection of vertices of $P$ that are mapped to a single point of $S$ under this gluing process is called a \emph{vertex cycle}.

\begin{theorem}[Theorem 1.3.5 \cite{buser2010geometry}] \label{thm: pasting}
    Assume that the following conditions hold for the above setting. 
    \begin{enumerate}
        \item The sum of the angles at each vertex cycle, which yields an interior point, is $2\pi$. 
        \item For each vertex cycle that yields a boundary point of $S$, the sum of the interior angles is $\leq \pi$. 
        \item $S$ is connected and any two non-adjacent sides of the polygons $P_1,\dots,P_m$ has positive distances.
    \end{enumerate}
    Then $S$ admits a hyperbolic structure, and the natural projection $\sigma: P_1\cup \dots \cup P_m \to S$ is a local isometry.
\end{theorem}

\begin{example}
    Consider a regular $4g$-gon in $\mathbb{H}$ whose interior angles are each $\pi/(2g)$. Identifying opposite sides of this polygon via hyperbolic isometries yields a hyperbolic surface of genus $g$. The case $g = 2$ is illustrated in Figure~\ref{fig:hyp surf example}.

    \begin{figure}[htbp]
    \centering
    \begin{tikzpicture}
           \draw [bend left, ->-=.5] ({2*cos(22.7)},{2*sin(22.7)}) to ({2*cos(67.5)},{2*sin(67.5)});
          \draw [bend left, ->-=.5] ({2*cos(67.5)},{2*sin(67.5)}) to ({2*cos(112.5)},{2*sin(112.5)});
          \draw [bend left, ->-=.5] ({2*cos(112.5)},{2*sin(112.5)}) to ({2*cos(157.5)},{2*sin(157.5)});
          \draw [bend left, ->-=.5] ({2*cos(157.5)},{2*sin(157.5)}) to ({2*cos(202.5)},{2*sin(202.5)});
          \draw [bend left, -<-=.5] ({2*cos(202.5)},{2*sin(202.5)}) to ({2*cos(247.5)}, {2*sin(247.5)});
          \draw [bend left, -<-=.5] ({2*cos(247.5)}, {2*sin(247.5)}) to ({2*cos(292.5)},{2*sin(292.5)});
          \draw [bend left, -<-=.5] ({2*cos(292.5)},{2*sin(292.5)}) to ({2*cos(337.5)},{2*sin(337.5)});
          \draw [bend left, -<-=.5] ({2*cos(337.5)},{2*sin(337.5)}) to ({2*cos(22.7)},{2*sin(22.7)});

          \draw ({1.93*cos(0)},{1.93*sin(0)}) node {$c$} ({1.9*cos(45)},{1.9*sin(45)}) node {$d$} ({1.9*cos(90)},{1.9*sin(90)}) node {$a$} ({1.9*cos(135)},{1.9*sin(135)}) node {$b$} ({1.9*cos(180)},{1.9*sin(180)}) node {$c$} ({1.9*cos(225)},{1.9*sin(225)}) node {$d$} ({1.9*cos(270)},{1.9*sin(270)}) node {$a$} ({1.93*cos(315)},{1.93*sin(315)}) node {$b$};
    \end{tikzpicture}
    \caption{Regular hyperbolic octagon with each angle $\pi/4$}
    \label{fig:hyp surf example}
\end{figure}
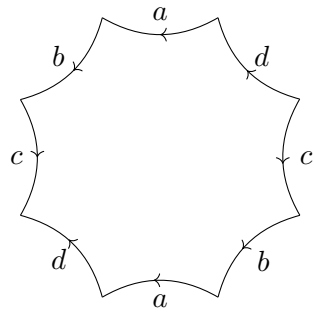
\end{example}
\subsection{Constructing Additional Hyperbolic Surfaces}
We now turn to another approach for producing hyperbolic surfaces.  
To this end, it is convenient to recall the notions of group actions that are free and properly discontinuous, and then to introduce the concept of Fuchsian groups.

\begin{definition}\label{def: prop discnt action}
Let $X$ be a topological space and let $G$ be a group of homeomorphisms of $X$ onto itself.  
The action of $G$ on $X$ is said to be \emph{properly discontinuous} if for every compact subset $K \subseteq X$, the intersection
\[
g(K) \cap K \neq \varnothing
\]
occurs for only finitely many elements $g \in G$.
\end{definition}

\begin{definition}
Consider a group action $G \times X \to X$.
\begin{enumerate}
    \item For each $x \in X$, the set
    \[
    \mathrm{Stab}(x) = \{ g \in G \mid g \cdot x = x \}
    \]
    is called the \emph{stabilizer} of $x$.
    \item The action is said to be \emph{free} if $\mathrm{Stab}(x) = \{\mathrm{id}\}$ for every $x \in X$.
\end{enumerate}
\end{definition}

\subsection*{Fuchsian Groups}
The group $\mathrm{SL}(2,\mathbb{R})$ inherits its topology from the standard embedding in $\mathbb{R}^{4}$.  
Consequently, the quotient space $\mathrm{SL}(2,\mathbb{R}) / \{\pm \mathrm{Id}\}$ is equipped with the corresponding quotient topology.  
Since we have 
\[
\mathrm{PSL}(2,\mathbb{R}) = \mathrm{SL}(2,\mathbb{R}) / \{\pm \mathrm{Id}\},
\]
this gives $\mathrm{PSL}(2,\mathbb{R})$ a natural topological structure.

\begin{definition}
A Fuchsian group is defined as a discrete subgroup of $\mathrm{PSL}(2,\mathbb{R}) \,(= \mathrm{Isom}^{+}(\mathbb{H}))$.
\end{definition}

Fuchsian groups act in a canonical way on the upper half-plane $\mathbb{H}$.  
They may also be described in terms of properly discontinuous actions:  
a subgroup $\Gamma \leq \mathrm{PSL}(2,\mathbb{R})$ is Fuchsian precisely when it acts properly discontinuously on $\mathbb{H}$ (see Theorem~2.2.6 in~\cite{Katok}).\label{thm:Fuchsian}

\begin{example}
Define $T : \mathbb{H} \to \mathbb{H}$ by $T(z) = z+1$.  
The infinite cyclic group $\langle T \rangle$ generated by this translation is an example of a Fuchsian group.
\end{example}

\begin{example}
The modular group $\mathrm{PSL}(2,\mathbb{Z}) \subseteq \mathrm{Isom}^{+}(\mathbb{H})$ is another classical example of a Fuchsian group.
\end{example}

To produce further examples of Fuchsian groups, one can use the celebrated Poincar\'e theorem (for details, see Theorem~9.8.4 in~\cite{Beardon}).  
Let $P$ be a hyperbolic polygon with $2r$ sides, and suppose that the sides are paired by isometries
\[
\{ g_1, \dots, g_r \} \subset \mathrm{Isom}^{+}(\mathbb{H}).
\]
Assume that at every vertex the total angle from the identified copies adds up to $2\pi$.  
Let
\[
G = \langle g_1, \dots, g_r \rangle
\]
denote the group generated by these side-pairing maps.

\begin{theorem}[Poincar\'e's Theorem]\label{thm:PoincareTheorem}
The group $G$ is a Fuchsian group.
\end{theorem}

We next record a result that allows one to construct new hyperbolic surfaces from suitable group actions.  
Let $S$ be a hyperbolic surface, and let $\Gamma \leq \mathrm{Isom}^{+}(S)$ act freely and properly discontinuously on $S$.

\begin{theorem}[{\cite[Theorem~1.2.4]{buser2010geometry}}]\label{thm:FuchsianQuotient}
The quotient space $S/\Gamma$ inherits the structure of a hyperbolic surface, and the canonical projection
\[
\pi : S \to S/\Gamma
\]
is a local isometry.
\end{theorem}

\begin{remark}
If $\Gamma$ is a Fuchsian group without elliptic elements, then by Theorem~\ref{thm:FuchsianQuotient} the quotient $\mathbb{H}/\Gamma$ is itself a hyperbolic surface.
\end{remark}

\section{Geodesics on Hyperbolic Surfaces}
On a hyperbolic surface, geodesics are defined locally by pulling back geodesics in $\mathbb{H}$ through the coordinate charts.  
A curve is called a geodesic if it is locally geodesic at each of its points.  
When a geodesic is topologically a circle, it is referred to as a \emph{simple closed geodesic}.  

In this section we focus on a fundamental property of simple closed geodesics, known as the \emph{Collar Lemma}.  
A \emph{collar} is a tubular neighbourhood of a simple closed geodesic that is homeomorphic to a cylinder.  
The Collar Lemma provides a quantitative description of the width of such collars, showing in particular that short closed geodesics are surrounded by relatively wide annular neighbourhoods.  
This phenomenon was first established by Keen~\cite{Keen}, and subsequently appeared in different formulations in the works of Basmajian~\cite{Basmajian}, Buser~\cite{BuserCollar}, and Chavel–Feldman~\cite{Chavel}, among others.

\begin{lemma}[Lemma 13.6, \cite{farb2011primer}]\label{lem: Collar}
Let $\gamma$ be a simple closed geodesic on a hyperbolic surface $S$.  
Then the neighbourhood
\[
N_\gamma = \{ x \in S : d(x,\gamma) \leq w \}
\]
is an embedded annulus, where the width $w$ is given by
\[
w = \operatorname{arcsinh}\!\left(\frac{1}{\sinh\!\left(\tfrac{1}{2}l(\gamma)\right)}\right).
\]
\end{lemma}

\section{Teichm\"uller Space}

Let $S = S_{g,b}$ be a surface of genus $g$ with $b$ boundary components.  
A \emph{marked hyperbolic surface} of signature $(g,b)$ is a pair $(X,\phi)$, where $X$ is a hyperbolic surface and $\phi : S \to X$ is a homeomorphism, referred to as the \emph{marking}.  

We introduce an equivalence relation on the set of marked hyperbolic surfaces as follows:  
\[
(X_1,\phi_1) \sim (X_2,\phi_2) \quad \Longleftrightarrow \quad \exists \ \text{an isometry } I : X_1 \to X_2 \ \text{such that } I \circ \phi_1 \simeq \phi_2,
\]
where $\simeq$ denotes homotopy.  
Equivalently, the diagram in Figure~\ref{fig:Teichmuller} commutes up to homotopy.  

\begin{figure}[htbp]
    \centering
    \begin{tikzpicture}
        \draw (0,0) node {S};
        \draw[->-=1] (-.2,-.3)--(-1.5,-1.5);
        \draw (-1.65,-1.85) node{$X_1$};
        \draw[->-=1] (.2,-.3)--(1.5,-1.5);
        \draw (1.65,-1.85) node{$X_2$};
        \draw[->-=1] (-1.3,-1.85)--(1.3,-1.85);
        \draw (-1.05,-.7) node {$\phi_1$} (1.05,-.7) node {$\phi_2$} (0,-1.6) node {$I$};
    \end{tikzpicture}
    \caption{Equivalence of two markings}
    \label{fig:Teichmuller}
\end{figure}
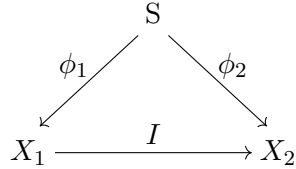

\begin{definition}
The \emph{Teichm\"uller space} of $S$ is the quotient
\[
\mathcal{T}(S) = \{ (X,\phi) \} / \sim.
\]
\end{definition}

Here, $S_{g,b}$ is referred to as the \emph{base surface}.  

\begin{example}[Proposition 10.5, \cite{farb2011primer}]
A \emph{pair of pants} is a compact surface of genus $0$ with three boundary components.  
Its Teichm\"uller space is homeomorphic to $\mathbb{R}^{3}$.
\end{example}

More generally, if $S_{g,b}$ is a surface of genus $g$ with $b$ boundary components satisfying $\chi(S_{g,b}) < 0$, then its Teichm\"uller space is topologically Euclidean.  

\begin{theorem}[Theorem 11, \cite{Schmutz1999}]
For a surface $S_{g,b}$ of negative Euler characteristic, the Teichm\"uller space is homeomorphic to
\[
\mathcal{T}(S_{g,b}) \cong \mathbb{R}^{6g - 6 + 3b}.
\]
\end{theorem}

Let \(X \in \mathcal{T}_g\) be a hyperbolic surface of genus \(g\).  
The collection of lengths of all simple closed geodesics on \(X\) is called the \emph{length spectrum} of \(X\), and is denoted by \(\mathrm{Spec}(X)\).  
It is known that \(\mathrm{Spec}(X) \subset \mathbb{R}_{>0}\) is a discrete set and has no accumulation point in \(\mathbb{R}\).  
The \emph{systolic length} of \(X\) is defined as
\[
\mathrm{sys}(X)=\min \mathrm{Spec}(X),
\]
and any simple closed geodesic on \(X\) achieving this minimum is called a \emph{systolic geodesic}.

This defines the \emph{systole function}
$
\mathrm{sys} : \mathcal{T}_g \longrightarrow \mathbb{R}_{>0},
$
which assigns to each hyperbolic surface its systolic length.  
A theorem of Akrout~\cite{akrout} states that this function is a topological Morse function on \(\mathcal{T}_g\).

\begin{proposition} \label{prop:no-minimum}
For any genus $g \geq 2$, the systole function $\operatorname{sys} : \mathcal{T}_g \to \mathbb{R}_{>0}$ does not attain a minimum on the moduli space $\mathcal{M}_g$.
\end{proposition}

\begin{proof}
We argue by contradiction. Suppose that the systole function attains a minimum at some surface $X_0 \in \mathcal{M}_g$. Since $X_0$ is a compact hyperbolic surface, its length spectrum is discrete and bounded away from zero. Thus, the shortest non-trivial simple closed geodesic on $X_0$ has strictly positive length, meaning $\operatorname{sys}(X_0) = c > 0$. We will show that $\inf_{X \in \mathcal{M}_g} \operatorname{sys}(X) = 0$, which yields a contradiction.

Fix a pair-of-pants decomposition on a reference surface of genus $g \geq 2$. This defines Fenchel-Nielsen coordinates $(\ell_i, \tau_i)_{i=1}^{3g-3}$ on Teichmüller space $\mathcal{T}_g$ (see Section $10.6$, \cite{farb2011primer}). Let $\gamma_1$ be one of the curves in this decomposition. We construct a $1$-parameter family of hyperbolic surfaces $X_\epsilon \in \mathcal{T}_g$ by setting $\ell_{X_\epsilon}(\gamma_1) = \epsilon$ for some $\epsilon > 0$, while fixing all other length parameters $\ell_i$ ($i \geq 2$) and twist parameters $\tau_i$ to arbitrary positive constants.

By the Collar Lemma (Lemma \ref{lem: Collar}), the simple closed geodesic $\gamma_1$ on $X_\epsilon$ has an embedded collar neighborhood of width
$$w(\epsilon) = \operatorname{arcsinh}\!\left(\frac{1}{\sinh\!\left(\tfrac{1}{2}\epsilon\right)}\right).$$
Notice that as $\epsilon \to 0^+$, the collar width $w(\epsilon) \to \infty$. 

For any simple closed geodesic $\beta$ on $X_\epsilon$ distinct from $\gamma_1$, there are two cases to consider:
\begin{enumerate}
    \item If $\beta$ intersects $\gamma_1$, it must cross the collar neighborhood of $\gamma_1$ at least twice. Therefore, $\ell_{X_\epsilon}(\beta) \geq 2w(\epsilon)$. Since $w(\epsilon) \to \infty$ as $\epsilon \to 0^+$, the length of any such intersecting curve becomes arbitrarily large.
    \item If $\beta$ does not intersect $\gamma_1$, its length is bounded from below by a positive constant that depends only on the fixed Fenchel-Nielsen length parameters $\ell_2, \dots, \ell_{3g-3}$.
\end{enumerate}

Consequently, for $\epsilon > 0$ sufficiently small, $\gamma_1$ is strictly shorter than any other non-trivial simple closed curve on $X_\epsilon$. This implies that
$$\operatorname{sys}(X_\epsilon) = \ell_{X_\epsilon}(\gamma_1) = \epsilon.$$

Passing to the moduli space $\mathcal{M}_g$, we take the limit as $\epsilon$ approaches zero:
$$\inf_{X \in \mathcal{M}_g} \operatorname{sys}(X) \leq \lim_{\epsilon \to 0^+} \operatorname{sys}(X_\epsilon) = 0.$$

Since $\operatorname{sys}(X) > 0$ for every $X \in \mathcal{M}_g$, the infimum of $0$ is not achieved by any surface in the moduli space. Therefore, the systole function does not attain a minimum on $\mathcal{M}_g$.
\end{proof}
 
\section{Bolza Surface}

As established in Proposition~\ref{prop:no-minimum}, the systole
function does not attain a minimum on the moduli space $\mathcal{M}_g$
for any genus $g \geq 2$. In contrast, the problem of determining the
global maximum of $\operatorname{sys}$ remains widely open for general
genus. The only completely understood case is genus $g=2$.

In this case, \cite{Jenni} proved that for any hyperbolic surface $X \in \mathcal{T}_2$, the following sharp inequality holds:
\[
\cosh\!\left(\tfrac{1}{2} \mathrm{sys}(X)\right) \leq \sqrt{2} + 1,
\]
where $\mathrm{sys}(X)$ denotes the systole of $X$. 

Moreover, equality holds if and only if $X$ is isometric to a unique surface of genus $2$, namely the \emph{Bolza surface}. In particular, the Bolza surface maximizes the systole function on $\mathcal{T}_2$.

\begin{definition}[Bolza surface]
The Bolza surface is the unique compact hyperbolic surface of genus $2$ obtained by the following construction:
\begin{itemize}
    \item Start with a regular hyperbolic octagon in $\mathbb{H}^2$ whose interior angles are all equal to $\pi/4$.
    \item Identify each pair of opposite sides via hyperbolic isometries.
\end{itemize}
This side-pairing produces a closed surface of genus $2$.

Equivalently, the Bolza surface can be described as a quotient
\[
B = \mathbb{H}^2 / \Gamma,
\]
where $\Gamma < \mathrm{PSL}(2,\mathbb{R})$ is a Fuchsian group generated by the side-pairing transformations (see Section~5 of \cite{schmutz1993riemann}).
\end{definition}

\subsection{Properties of the Bolza surface}

The Bolza surface enjoys several remarkable geometric and symmetry properties:

\begin{enumerate}
    \item \textbf{Number of systoles:}  
    The Bolza surface has exactly $12$ systolic geodesics, i.e., $12$ simple closed geodesics that realize the systole (the shortest length among all non-contractible closed curves) \cite{schmutz1993riemann}.

    \item \textbf{Geometric decomposition:}  
    These $12$ systolic geodesics decompose the surface into a collection of hyperbolic triangles (see Figure~\ref{systoles: Bolza surface}). Each triangle has interior angles equal to $\pi/4$, producing a highly symmetric tiling of the surface (see Theorem~3.1 in \cite{MR3877282}).

    \item \textbf{Position in moduli space:}  
    The Bolza surface lies in the Thurston set $\chi_2$, which consists of genus-$2$ surfaces exhibiting extremal geometric properties.

    \item \textbf{Automorphism group:}  
    The Bolza surface has the largest possible conformal automorphism group among all genus-$2$ hyperbolic surfaces. In fact, it admits exactly $48$ conformal automorphisms, making it maximally symmetric in genus $2$ \cite{schmutz1993riemann}.

    \item \textbf{Second systoles:}  
    In addition to the $12$ systoles, there are exactly $12$ \emph{second systolic geodesics}, i.e., simple closed geodesics realizing the second smallest length. These curves form a triangulation of the surface in which each triangle has interior angles $\pi/4$, $\pi/3$, and $\pi/3$ (see Theorem~4.1 in \cite{saha_sanki_2023_complexity_in_bolza_surface}).
\end{enumerate}

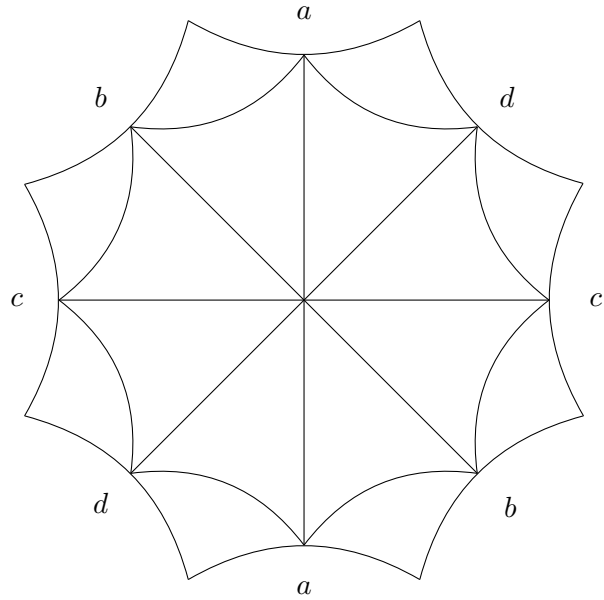
\begin{figure}[htbp]
          \centering
          \begin{tikzpicture}[scale=2]
          \draw [bend left] ({2*cos(22.7)},{2*sin(22.7)}) to ({2*cos(67.5)},{2*sin(67.5)});
          \draw [bend left] ({2*cos(67.5)},{2*sin(67.5)}) to ({2*cos(112.5)},{2*sin(112.5)});
          \draw [bend left] ({2*cos(112.5)},{2*sin(112.5)}) to ({2*cos(157.5)},{2*sin(157.5)});
          \draw [bend left] ({2*cos(157.5)},{2*sin(157.5)}) to ({2*cos(202.5)},{2*sin(202.5)});
          \draw [bend left] ({2*cos(202.5)},{2*sin(202.5)}) to ({2*cos(247.5)}, {2*sin(247.5)});
          \draw [bend left] ({2*cos(247.5)}, {2*sin(247.5)}) to ({2*cos(292.5)},{2*sin(292.5)});
          \draw [bend left] ({2*cos(292.5)},{2*sin(292.5)}) to ({2*cos(337.5)},{2*sin(337.5)});
          \draw [bend left] ({2*cos(337.5)},{2*sin(337.5)}) to ({2*cos(22.7)},{2*sin(22.7)});

          \draw ({1.93*cos(0)},{1.93*sin(0)}) node {$c$} ({1.9*cos(45)},{1.9*sin(45)}) node {$d$} ({1.9*cos(90)},{1.9*sin(90)}) node {$a$} ({1.9*cos(135)},{1.9*sin(135)}) node {$b$} ({1.9*cos(180)},{1.9*sin(180)}) node {$c$} ({1.9*cos(225)},{1.9*sin(225)}) node {$d$} ({1.9*cos(270)},{1.9*sin(270)}) node {$a$} ({1.93*cos(315)},{1.93*sin(315)}) node {$b$};

          \draw[bend left] ({1.62*cos(0)},{1.62*sin(0)}) to ({1.62*cos(45)},{1.62*sin(45)});
          \draw [bend left] ({1.62*cos(45)},{1.62*sin(45)}) to ({1.62*cos(90)},{1.62*sin(90)});
          \draw [bend left] ({1.62*cos(90)},{1.62*sin(90)}) to ({1.62*cos(135)},{1.62*sin(135)});
          \draw [bend left] ({1.62*cos(135)},{1.62*sin(135)}) to ({1.62*cos(180)},{1.62*sin(180)});
          \draw [bend left] ({1.62*cos(180)},{1.62*sin(180)}) to ({1.62*cos(225)},{1.62*sin(225)});
          \draw [bend left] ({1.62*cos(225)},{1.62*sin(225)}) to ({1.62*cos(270)},{1.62*sin(270)});
          \draw [bend left] ({1.62*cos(270)},{1.62*sin(270)}) to ({1.62*cos(315)},{1.62*sin(315)});
          \draw [bend left] ({1.62*cos(315)},{1.62*sin(315)}) to ({1.62*cos(0)},{1.62*sin(0)});

          \draw ({1.62*cos(0)},{1.62*sin(0)})--({1.62*cos(180)},{1.62*sin(180)});
          \draw ({1.62*cos(45)},{1.62*sin(45)})--({1.62*cos(225)},{1.62*sin(225)});
          \draw ({1.62*cos(90)},{1.62*sin(90)})--({1.62*cos(270)},{1.62*sin(270)});
          \draw ({1.62*cos(135)},{1.62*sin(135)})--({1.62*cos(315)},{1.62*sin(315)});
          
          \end{tikzpicture}
          \caption{Systolic geodesics in the Bolza Surface.}
          \label{systoles: Bolza surface}
\end{figure}

\chapter{Bolza like Surfaces in Thurston Set}
\label{chapter 6}

In this chapter, our main focus will be on constructing of Bolza like surfaces in Thurston set. At first, we define  Bolza like surfaces. Then  we construct Bolza-like surfaces for infinitely many genera $g$. Next, we will see an interesting application of the Bolza-like surfaces that we obtain, namely the construction of global maxima of the systole function. Furthermore, we show that the distribution of the systolic geodesics is symmetric in the sense that any simple closed geodesic intersects the Bolza-like surface at an even number of points.

\section{Construction of Bolza-like Surfaces}

A hyperbolic triangle of type $(p,q,r)$, where $p, q, r \in \mathbb{N}$ satisfies $\tfrac{1}{p} + \tfrac{1}{q} + \tfrac{1}{r} < 1$, is defined as a hyperbolic triangle whose interior angles are $\pi/p$, $\pi/q$, and $\pi/r$. Such triangles serve as fundamental building blocks for many hyperbolic surfaces.

It is well known that the systolic geodesics on the Bolza surface partition it into hyperbolic triangles of type $(4,4,4)$ (see Section~5.5.1). This remarkable geometric decomposition provides a natural way to study the structure of the surface through its shortest closed geodesics.

Inspired by this observation, we generalize the notion to a broader class of hyperbolic surfaces, called Bolza-like surfaces, which are defined as follows.

\begin{definition}
A Bolza-like surface is a closed hyperbolic surface whose systolic geodesics decompose the surface into hyperbolic triangles of type $(p,q,r)$, for some integers $p,q,r$ satisfying $\tfrac{1}{p} + \tfrac{1}{q} + \tfrac{1}{r} < 1$.
\end{definition}

In this section, we prove the main result of this chapter in Theorem~\ref{Bolza like surfaces}. Before going on to it, we prove two lemmas that are essential in proving the theorem.

\begin{lemma}\label{lemma1.2}
For $0\leq \epsilon<\frac{\pi}{4}$, there exists a unique hyperbolic quadrilateral with consecutive angles $\left(\frac{\pi}{4}-\epsilon\right),\left(\frac{\pi}{4}+\epsilon\right),\left(\frac{\pi}{4}-\epsilon\right) \text{ and }\left(\frac{\pi}{4}+\epsilon\right)$ and sides of equal length.
\end{lemma}
\begin{proof}
Consider a hyperbolic triangle with angles $(\pi/4+\epsilon), (\pi/8-\epsilon/2) \text{ and }(\pi/8-\epsilon/2)$ [For the existence of such a triangle, see Theorem 7.16.2 of \cite{beardon2012geometry}]. The uniqueness (up to isometry) of this triangle follows from the trigonometric identity $(3)$ in Theorem~\ref{Cosine Rule}.

The quadrilateral is obtained by taking two disjoint isometric copies of such triangle and attaching them along the side opposite to angle $(\pi/4+\epsilon)$ (see Figure~\ref{quadrilateral}). 

\noindent\textbf{Uniqueness:} If such quadrilateral is cut along a diagonal, we get two isometric triangles. The uniqueness of the quadrilateral follows from the uniqueness of the triangles.
\begin{figure}[htbp]
\begin{center}
\begin{tikzpicture}[xscale=1,yscale=1]

\draw (0,0) circle (2);
\draw[dashed] (1,0)--(-1,0);
\draw (-1,0) [bend right] to (0,1) to (1,0) to (0,-1) to (-1,0); 
\draw (0,1.3) node{A} (-1.2,0) node{B} (0,-1.3) node{C} (1.2,0) node{D};

\end{tikzpicture}
\end{center}
 \caption{$\triangle ABD$ and $\triangle BDC$ are the isometric triangles with interior angles at each of the vertices $A$ and $C$ is $\frac{\pi}{4}-\epsilon$.}
\label{quadrilateral}
\end{figure}
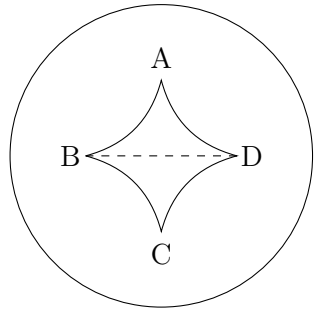
\end{proof}

\textbf{Notation:} We call the quadrilateral in Lemma~\ref{lemma1.2} as \emph{square} and denote it by $Q_\epsilon$.

\begin{lemma}\label{lem:3.2} (\cite{DeySahaSankiToAppear})
    The common perpendiculars between the opposite sides of the square $Q_\epsilon$ pass through the point of intersection of the diagonals and they have the same length $\sqrt{2} \cos\epsilon +1$. 
\end{lemma}

\begin{proof}
    \begin{figure}[htbp]
        \centering
        \begin{tikzpicture}
            \draw [bend left=20] (-1.5,-1.5) to (1.5,-1.5) to (1.5,1.5) to (-1.5,1.5) to (-1.5,-1.5);
            \draw (-1.5,-1.5)--(1.5,1.5);
            \draw (1.5,-1.5)--(-1.5,1.5);
            \draw (0,-.35) node {$O$};
            \draw [dashed, thick, bend left=10] (-1.2,0) to (0,0) to (1.22,0);
            \draw (-1.6,0) node{$M$} (-1.75,1.5) node {$A$} (-1.75,-1.5) node {$B$} (1.5,0) node{$N$} (1.75,1.5) node {$D$} (1.75,-1.5) node {$C$};
        \end{tikzpicture}
        \caption{}
        \label{fig:3.2}
    \end{figure}
    Label the vertices of $Q_\epsilon$ by $A,B,C,$ and $D$ and let the interior angles at these vertices be $\frac{\pi}{4}+\epsilon$, $\frac{\pi}{4}-\epsilon$, $\frac{\pi}{4}+\epsilon$, and $\frac{\pi}{4}-\epsilon$, respectively. Let the diagonals of the square meet at $O$. Suppose that the perpendiculars from $O$ into the sides $AB$ and $CD$ meet at $M$ and $N$, respectively (see Figure \ref{fig:3.2}). We show that $MON$ is a geodesic, that is $\angle MON = \pi$ and this will prove that the common perpendicular between $AB$ and $CD$ passes through $O$. In triangles $\triangle AMO$ and $\triangle CNO$, $\angle MAO=\frac{\pi}{8}+ \frac{\epsilon}{2}=\angle NCO$, $\angle OMA=\frac{\pi}{2}= \angle ONC$, $OA=\frac{1}{2} AC= OC$. From the trigonometric identities for right-angled triangle, as discussed in Section 2, the triangles $\triangle AMO$ and $\triangle CNO$ are isometric and hence $\angle{AOM} =\angle{CON}$. This implies that $\angle MON= \pi$ because $AC$ is a geodesic. Now we find the length of the perpendicular $MN$. Using the identity $(2)$ in Theorem~\ref{pythagorus} in the triangle $\triangle AMO$, we have
    \begin{align*}
        \sinh{OM}=\sinh{OA} \sin{\left(\frac{\pi}{8}+\frac{\epsilon}{2}\right)}.
    \end{align*}
    Therefore,
    \begin{align*}
        \cosh{MN}&=2\sinh^2{OM}+1\\
        &=2 \sinh^2{OA} \sin^2{\left(\frac{\pi}{8}+\frac{\epsilon}{2}\right)}+1\\
        &=(\cosh{AC}-1) \sin^2{\left(\frac{\pi}{8}+\frac{\epsilon}{2}\right)} +1\\
        &=\sqrt{2} \cos{\epsilon} +1.
    \end{align*}
    Similarly, one can check that the common perpendicular between the sides $AD$ and $BC$ passes through $O$ and its length is also given by $\sqrt{2} \cos{\epsilon} +1$.
\end{proof}
\begin{note}
    The distance between the opposite sides of the square $Q_\epsilon$ decreases with $\epsilon$.
\end{note}

In this section, we construct a family of hyperbolic surfaces $S_g(t)$ of sufficiently large genus $g$, lying in Thurston’s set, with the property that their systolic length increases monotonically with the parameter $t$. 
More precisely, we establish Theorem~\ref{Bolza like surfaces}, which we restate below for clarity.

\begin{repthm}[\ref{Bolza like surfaces}] (\cite{DeySahaSankiToAppear})
   For every integer $g=kn+1$, where $k,n\in \N, k\geq 2, n\geq 4$, there is a continuous function $\mathcal{D}_g:\left[ 0, \frac{\pi}{12}\right] \to \mathcal{M}_g$ satisfying the following.
    \begin{enumerate}
        \item $\mathcal{D}_g(t)\in \chi_g$, for all $t\in \left[ 0, \frac{\pi}{12}\right]$, where $\chi_g$ is the Thurston set.
        \item $\mathrm{sys}(\mathcal{D}_g(t_1))<\mathrm{sys}(\mathcal{D}_g(t_2))$, for all $0\leq t_1 < t_2 \leq \frac{\pi}{12}.$
        \item The surface $\mathcal{D}_g\left(\frac{\pi}{12}\right)$ has $(6g-6)$ many systolic geodesics that give a triangle decomposition of the surface. In other words, $\mathcal{D}_g\left(\frac{\pi}{12}\right)$ is Bolza-like.
    \end{enumerate}
\end{repthm}

\begin{proof}
    First, we construct a hyperbolic torus with $2g-2$ cone points with angles $\pi$ at each cone point as follows. For any two numbers $m,n\in \mathbb{N}$, take $mn$ identical copies of the squares $Q_\epsilon$ and arrange them in an $m\times n$ rectangular grid. We further identify the sides of the rectangles in a natural way so that we obtain a torus $T_\epsilon$ with $mn$ singular points and the angle at each singular point is $\pi$. If $mn$ is even, then there is a hyperbolic surface of the genus $g=\frac{mn+2}{2}$, say, $S_g(\epsilon)$ and an orientation preserving involution $\sigma_\epsilon$ of $S_g(\epsilon)$ with $2g-2$ fixed points such that the quotient of the surface by the involution is the torus $T_\epsilon$, that is $S_g(\epsilon)/\sigma_\epsilon\cong T_\epsilon$ (see Figure \ref{fig:involution}). We define the path $\mathcal{D}_g: \left[0,\frac{\pi}{4}\right) \to \mathcal{M}_g$ by
    $$\mathcal{D}_g(\epsilon)=S_g(\epsilon), \text{ for all } \epsilon \in [0,\pi/4).$$

    \begin{figure}[htbp]
        \centering
        \begin{tikzpicture}
            \draw (-1.2,2)..controls(-.5,6) and (.5,6)..(1.2,2);
            %\draw (-4,1)--(-2,1);
            \draw (-1.2,2)..controls(-1.3,1.5) and (-1.7,1)..(-2,1);
            %\draw (4,1)--(2,1);
            \draw (1.2,2)..controls(1.3,1.5) and (1.7,1)..(2,1);

            \draw (-1.2,-2)..controls(-.5,-6) and (.5,-6)..(1.2,-2);
            %\draw (-4,-1)--(-2,-1);
            \draw (-1.2,-2)..controls(-1.3,-1.5) and (-1.7,-1)..(-2,-1);
            %\draw (4,-1)--(2,-1);
            \draw (1.2,-2)..controls(1.3,-1.5) and (1.7,-1)..(2,-1);

            \draw (-2,1)..controls(-4,.7) and (-4,-.7)..(-2,-1);
            \draw (2,1)..controls(4,.7) and (4,-.7)..(2,-1);

            \draw (0.05,.95) to [bend left=60] (0.05,2.05);
            \draw (-.08,1.1) to [bend right=60] (-.08,1.9);

            \draw (0.05,2.95) to [bend left=60] (0.05,4.05);
            \draw (-.08,3.1) to [bend right=60] (-.08,3.9);

            \draw (0.05,-.95) to [bend right=60] (0.05,-2.05);
            \draw (-.08,-1.1) to [bend left=60] (-.08,-1.9);

            \draw (0.05,-2.95) to [bend right=60] (0.05,-4.05);
            \draw (-.08,-3.1) to [bend left=60] (-.08,-3.9);

            %\draw [very thick] (-.1,0) node {$\vdots$};
            \foreach \y in {0,0.5,-.5}
            {
                \fill (-0.1,\y) circle (1.5pt);
            }

            \draw [dotted] (0,6)--(0,-6);

            % \draw (1.45,.05) to [bend right=60] (2.55,.05);
            % \draw (1.55,.0) to [bend left=60] (2.45,.0);
            
            \draw (1.45,0.05) to [bend right=60] (2.55,0.05);
            \draw (1.6,-.08) to [bend left=60] (2.4,-.08);

            \draw (-1.45,0.05) to [bend left=60] (-2.55,0.05);
            \draw (-1.6,-.08) to [bend right=60] (-2.4,-.08);

            % \draw (-2.05,.95-1.5) to [bend left=60] (-2.05,2.05-1.5);
            % \draw (1.92-4.1,1.1-1.5) to [bend right=60] (1.92-4.1,1.9-1.5);

            \draw [->-=1] (-.3,5.5) arc
            [
                start angle=100,
                end angle=420,
                x radius=1cm,
                y radius =.12cm
            ] ;
        
            \draw [->-=1] (4,0)--(5,0);
            \draw (1.1,5.2) node {$\sigma_\epsilon$};

            \draw (-1.2+9,2)..controls(-1.3+9,1.5) and (-1.7+9,1)..(-2+9,1);
            \draw (-1.2+9,-2)..controls(-1.3+9,-1.5) and (-1.7+9,-1)..(-2+9,-1);
            \draw (-2+9,1)..controls(-4+9,.7) and (-4+9,-.7)..(-2+9,-1);

            \draw (-1.45+9,0.05) to [bend left=60] (-2.55+9,0.05);
            \draw (-1.6+9,-.08) to [bend right=60] (-2.4+9,-.08);
            %\draw (5.5,0) node {$\bullet$};
            \draw (9,5)..controls (8.8,5) and (8.5,4.2)..(-1.2+9,2);
            \draw (9,-5)..controls (8.8,-5) and (8.5,-4.2)..(-1.2+9,-2);

            \draw (0.05+9,.95) to [bend left=60] (0.05+9,2.05);
            \draw (0.05+9,2.95) to [bend left=60] (0.05+9,4.05);
            \draw (0.05+9,-.95) to [bend right=60] (0.05+9,-2.05);
            \draw (0.05+9,-2.95) to [bend right=60] (0.05+9,-4.05);

            \draw (9,5) to [bend right=10] (0.05+9,4.05);
            \draw (0.05+9,2.95) to [bend right=10] (0.05+9,2.05);
            \draw [very thick, dotted] (0.05+9,.95) to [bend right=10] (0.05+9,-.95);
            \draw (9,-5) to [bend left=10] (0.05+9,-4.05);
            \draw (0.05+9,-2.95) to [bend left=10] (0.05+9,-2.05);
        \end{tikzpicture}
        \caption{}
        \label{fig:involution}
    \end{figure}

    Now, we find all the systolic geodesics of the surface $S_g(\epsilon)$. As the involution $\sigma_\epsilon$ acts as a rotation by the angle $\pi$, any geodesic segment in $T_\epsilon$ joining two distinct singular points (not passing through any other singular points) lift to a geodesic in $S_g(\epsilon)$. In particular, the sides of squares and the diagonals of the square joining the vertices with interior angle $\left (\frac{\pi}{4}+\epsilon\right)$ lift to geodesic in $S_g(\epsilon)$. We denote the length of these segments by $l_\epsilon$ and $x_\epsilon$, respectively (see Figure \ref{FIG:0.1}). We show that $\sys(S_g(\epsilon)) =\min\{2l_\epsilon,2x_\epsilon\}$ and the systolic geodesics are either the lifts of the sides of $Q_\epsilon$ or the lifts of the sides of diagonals or both (depending on $\epsilon$). By hyperbolic trigonometric identity $(3)$, Theorem \ref{Cosine Rule}, it follows that
    \begin{align}
        \cosh\left(l_\epsilon\right)&=\frac{\cos\left(\frac{\pi}{8}-\frac{\epsilon}{2}\right)+\cos\left(\frac{\pi}{8}-\frac{\epsilon}{2}\right)\cos\left(\frac{\pi}{4}+\epsilon\right)}{\sin\left(\frac{\pi}{8}-\frac{\epsilon}{2}\right)\sin\left(\frac{\pi}{4}+\epsilon\right)} \label{eqn: 3.1},\\
        \cosh (x_\epsilon)&= \frac{\cos\left( \frac{\pi}{4}-\epsilon\right)+ \cos^2\left( \frac{\pi}{8}+\frac{\epsilon}{2}\right)}{\sin^2\left( \frac{\pi}{8}+\frac{\epsilon}{2}\right)}\label{eqn:3.2}.
    \end{align}

    Initially, $l_0\approx 2.448 <3.057 \approx x_0$. Differentiating equations \eqref{eqn: 3.1} and \eqref{eqn:3.2} with respect to $\epsilon$, we have
    \begin{align*}
        \frac{d}{d\epsilon}\left(\cosh\left(l_\epsilon\right)\right)&=\frac{1}{4}\cdot \frac{\sin \left(\frac{\pi}{4}+\epsilon\right)-\sin \left(\frac{\pi}{4}- \epsilon\right)}{\sin^2 \left(\frac{\pi}{8}-\frac{\epsilon}{2}\right )\sin^2 \left(\frac{\pi}{4}+\epsilon\right )}>0, \ \text{ for } 0\leq \epsilon < \frac{\pi}{4}\\
        \frac{d}{d\epsilon}(\cosh{x_\epsilon})&= -\frac{\sqrt{2} \cdot\cos{\left( \frac{\pi}{8}-\frac{\epsilon}{2}\right)}}{\sin^3{\left( \frac{\pi}{8}+\frac{\epsilon}{2}\right)}}<0 \text{ for } 0\leq \epsilon<\frac{\pi}{4}.
    \end{align*}
    Thus, $l_\epsilon$ is strictly increasing and $x_\epsilon$ strictly decreasing in $0\leq \epsilon<\frac{\pi}{4}$ with $x_\epsilon\to 0$ as $\epsilon \to \frac{\pi}{4}$. Therefore, there exists some $\epsilon_0\in (0, \pi/4)$ such that $l_{\epsilon_0}=x_{\epsilon_0}$. But in this case 
    $$\frac{\pi}{8}+\frac{\epsilon_0}{2}=\frac{\pi}{4}-\epsilon_0 \implies \epsilon_0=\frac{\pi}{12}.$$

    Now we show that for any simple closed geodesic $\gamma$ in $S_g(\epsilon)$, $l(\gamma) \geq \min \{2l_\epsilon,2x_\epsilon\}$, for $0\leq \epsilon \leq \frac{\pi}{12}$. The following two cases to consider.

    \textbf{Case 1:} $\gamma$ passes through a fixed point of the involution $\sigma_\epsilon$. Then $\gamma$ must pass through another fixed point that is diametrically opposite to the first one. In the torus $T_\epsilon$, the shortest arc joining two singular points has lengths either $l_\epsilon$ or $x_\epsilon$ (depending on $\epsilon$). Thus, $l(\gamma)\geq \min \{2l_\epsilon,2x_\epsilon\}$.

    \textbf{Case 2:} $\gamma$ does not pass through any fixed points. Then $\gamma$ project onto a closed curve (possibly non-simple) in $T_\epsilon$. In this case, we choose $m,n$ large enough such that $l(\gamma)\geq \min \{2l_\epsilon,2x_\epsilon\}$. 
    
    Now, we find lower bounds of $m,n$ and, consequently, of the genus $g$. In $0\leq \epsilon \leq \pi/12$, $l_\epsilon< x_\epsilon$, maximum value $l_\epsilon$= $l_\frac{\pi}{12} \approx 2.55$. By Lemma \ref{lem:3.2}, the shortest distance $y_\epsilon$ between the opposite sides is decreasing and $y_{\pi/12} \approx 1.506$. As $4$ is the smallest integer such that $4\times y_{\pi/12}> 2\times l_{\pi/12}$, we have $m,n \geq 4$. Therefore, the genus $g=\frac{mn+2}{2} \geq 9$. As $mn$ is also even, we take $m=2k, k\geq 2$ and hence $g=kn+1$, where $k,n\in \N, k\geq 2, n\geq 4$.

    Now we show that $S_g(\epsilon)\in \chi_g$ for $0\leq \epsilon \leq \pi/12$. As $l_\epsilon< x_\epsilon$ for $0\leq \epsilon < \pi/12$, the systolic geodesics decompose $S_g(\epsilon)$ into the squares $Q_\epsilon$. For $\epsilon= \pi/12$, the diagonals further decompose it into triangles with each interior angle $\pi/6$. This completes the proof.

    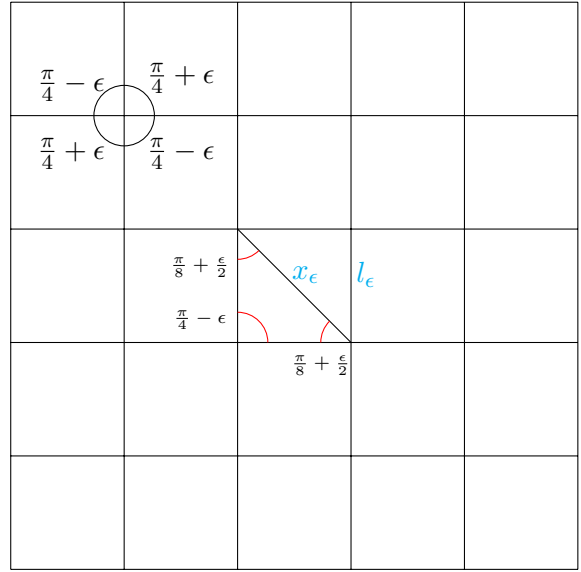
\begin{figure}[htbp]
    \begin{center}
    \begin{tikzpicture}
    
     \draw[step=1.5 cm] (0,0) grid (7.5,7.5) ;
     \draw[thin]  (1.5,6) circle (.4 cm);
     \draw (.8,6.4) node{$\frac{\pi}{4}-\epsilon$} (2.25,6.5) node{$\frac{\pi}{4}+\epsilon$} (2.25,5.5) node{$\frac{\pi}{4}-\epsilon$} (.8,5.5) node{$\frac{\pi}{4}+\epsilon$};
     \draw (3,4.5)--(4.5,3);
     \draw [red] (3,4.1) arc
    [
        start angle=270,
        end angle=315,
        x radius=.4cm,
        y radius =.4cm
    ] ;
    \draw [red] (4.1,3) arc
    [
        start angle=180,
        end angle=135,
        x radius=.4cm,
        y radius =.4cm
    ] ;
    \draw [red] (3.4,3) arc
    [
        start angle=0,
        end angle=90,
        x radius=.4cm,
        y radius =.4cm
    ] ;
     \draw (2.5,4) node {\tiny$\frac{\pi}{8}+\frac{\epsilon}{2}$} (2.5,3.3) node {\tiny$\frac{\pi}{4}-\epsilon$} (4.1,2.7) node {\tiny$\frac{\pi}{8}+\frac{\epsilon}{2}$};
     \draw[thick, cyan] (3.9,3.9) node{$x_\epsilon$};
     \draw[thick, cyan] (4.7,3.9) node{$l_\epsilon$};
     \end{tikzpicture}
    \end{center}
    \caption{The torus $T_\epsilon$}
    \label{FIG:0.1}
    \end{figure}
\end{proof}

\begin{note}
\begin{enumerate}
   \item The set of systoles on the surface $S_g\left(\frac{\pi}{12}\right)$ forms a distance Delaunay triangulation with $2g-2$ vertices (for a detailed discussion on distance Delaunay triangulation, see Section 2, \cite{MinimalDelaunay}).

    \item Though, the systolic geodesics in $S_g\left(\frac{\pi}{12}\right)$ decompose the surface into triangles as in the Bolza surface, it is still not a maximal surface. Because a maximal surface of genus $g$ contains at least $6g-5$ systolic geodesics (see Theorem 2.8 \cite{Schmutz}). Note that a maximal surface is a surface where the systole function attains a maxima (local or global). 
    \end{enumerate}
\end{note}

\section{An application of Bolza-like surfaces}

\textbf{A construction of global maximal surfaces} (Section 3, \cite{DeySahaSankiToAppear}): Let $X$ be a closed orientable hyperbolic surface which admits a tesselation by triangles of equal side lengths and the number $N$ of triangles around each vertex is the same. Let $H_a$ denote a right-angled hyperbolic hexagon with lengths of each alternating side $a$. Then the surface $X(a)$ which has the same tesselation structure as $X$ in which the triangles are replaced by the hexagon $H_a$ and the triangle vertices by boundary geodesics of equal length $Na$ (see Figure \ref{fig:localPic}). Then $X(0)$ is the surface where the boundaries become cusps. In Section 3.1 \cite{SchmutzGlobalMaximal}, P. Schmutz proved that the surface $X(0)$ is a global maximal surface with a trivial automorphism group.

In the Bolza-like surface $S_g\left( \frac{\pi}{12}\right)$, as obtained in Theorem \ref{Bolza like surfaces}, the systolic geodesics gives a triangle decomposition with equal side lengths of the triangle and the number of vertices of the triangles is $2g-2$. Therefore, we will be able to find the global maximal surface in the Teichm\"uller space of genus $g$ and with $2g-2$ punctures, where $g=kn+1$, $k,n\in \N, k\geq 2, n\geq 4$. 

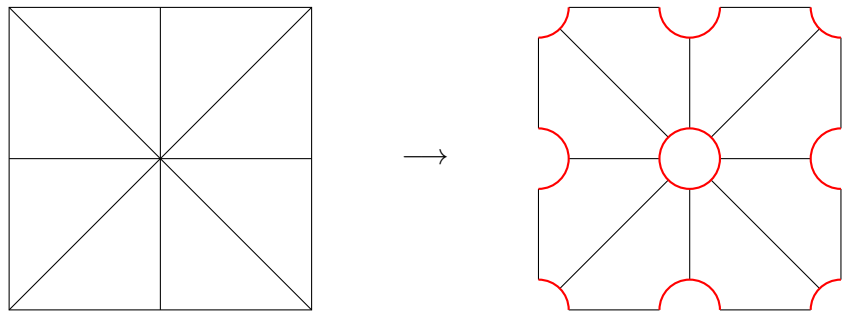
\begin{figure}[htbp]
    \centering
    \begin{tikzpicture}
        \draw (-2,-2)--(2,-2)--(2,2)--(-2,2)--cycle;
        \draw (-2,0)--(2,0) ;
        \draw (0,2)--(0,-2) ;
        \draw (-2,-2)--(2,2) ;
        \draw (2,-2)--(-2,2) ;

        \draw (3+2,1.6)--(3+2,.4);
        \draw (3+2,-1.6)--(3+2,-.4);
        \draw (7+2,1.6)--(7+2,.4);
        \draw (7+2,-1.6)--(7+2,-.4);
        \draw (3.4+2,2)--(4.6+2,2);
        \draw (5.4+2,2)--(6.6+2,2);
        \draw (3.4+2,-2)--(4.6+2,-2);
        \draw (5.4+2,-2)--(6.6+2,-2);
        \draw (3.4+2,0)--(4.6+2,0);
        \draw (5.4+2,0)--(6.6+2,0);
        \draw (5+2,.4)--(5+2,1.6);
        \draw (5+2,-.4)--(5+2,-1.6);
        \draw ({3+.4*cos(315)+2},{2+.4*sin(315)})--({5+.4*cos(135)+2},{.4*sin(135)});
        \draw ({5+.4*cos(315)+2},{.4*sin(315)})--({7+.4*cos(135)+2},{-2+.4*sin(135)});
        \draw ({7+.4*cos(225)+2},{2+.4*sin(225)})--({5+.4*cos(45)+2},{.4*sin(45)});
        \draw ({5+.4*cos(225)+2},{.4*sin(225)})--({3+.4*cos(45)+2},{-2+.4*sin(45)});

        \draw [thick, red] (5+2,0) circle(.4 cm);
        \draw[thick, red] (3+2,1.6) arc[start angle=270, end angle=360, radius=.4];
        \draw[thick, red] (6.6+2,2) arc[start angle=180, end angle=270, radius=.4];
        \draw[thick, red] (3+2,-1.6) arc[start angle=90, end angle=0, radius=.4];
        \draw[thick, red] (6.6+2,-2) arc[start angle=180, end angle=90, radius=.4];
        \draw[thick, red] (3+2,-.4) arc[start angle=270, end angle=450, radius=.4];
        \draw[thick, red] (7+2,.4) arc[start angle=90, end angle=270, radius=.4];
        \draw[thick, red] (4.6+2,2) arc[start angle=180, end angle=360, radius=.4];
        \draw[thick, red] (4.6+2,-2) arc[start angle=180, end angle=0, radius=.4];
        
        \draw (3.5,0) node {$\longrightarrow$};
    \end{tikzpicture}
    \caption{Local Picture around a vertex}
    \label{fig:localPic}
\end{figure}

\section{Distribution of the systolic geodesics on the Bolza-like surfaces}

In this section, we establish the symmetry in the distribution of systolic geodesics on the Bolza-like surfaces. 
More precisely, we prove Theorem~\ref{Distribution}, which we restate below for convenience.
\begin{repthm}[\ref{Distribution}] (\cite{DeySahaSankiToAppear}):
    Any simple closed geodesic on the Bolza-like surfaces $(\mathcal{D}_g\left(\frac{\pi}{12} \right))$, as obtained in Theorem \ref{Bolza like surfaces}, intersects the collection $\mathrm{Sys}\left(\mathcal{D}_g\left(\frac{\pi}{12} \right)\right)$ of all systolic geodesics at an even number of points.
\end{repthm}
\begin{proof}
    Let $\gamma$ be a simple closed non-systolic geodesic in $S_g\left(\frac{\pi}{12} \right)$. Then the following two cases may arise.

    \noindent \textbf{Case 1:} Let $\gamma$ pass through some fixed points of the involution $\sigma_\epsilon$ (as discussed in Theorem \ref{Bolza like surfaces}). Then $\gamma$ is preserved by the involution $\sigma_\epsilon$ and it passes through an even number of fixed points of $\sigma_\epsilon
    $. Hence, if $\gamma$ intersects a systolic geodesic at a point, another point must exist at which $\gamma$ intersects the same geodesic. This point is the image of the first point of intersection under the involution. Thus, the theorem holds in this case.

    \noindent \textbf{Case 2:} Let $\gamma$ not pass through the fixed points of the involution $\sigma$. Then $\gamma$ descends to a closed geodesic $\overline{\gamma}$ in the torus $T_{\frac{\pi}{12}}$ which is not necessarily simple. As the parity of the geometric intersection number depends only on the topology of the surface, for convenience, we may consider $T=T_{\frac{\pi}{12}}$ as a topological surface. Let $T$ be obtained by identifying the opposite sides of the unit Euclidean square (see Figure \ref{fig torus}). Then the following two subcases may arise. 
    \begin{figure}[H]
        \centering
        \begin{tikzpicture}
            \draw[step=1cm,gray,very thin] (0,0) grid (5,4);
            % \foreach \x in {0,1,...,4} 
            % {
            % \draw[gray, very thin] (0,4-\x) -- (\x,4);
            % \draw[gray, very thin] (\x+1, 0) -- (5, 4-\x);
            % }
            \foreach \x in {0,1,...,4} 
            {
            \draw[gray, very thin] (0,\x) -- (\x,0);
            \draw[gray, very thin] (\x+1, 4) -- (5, \x);
            }
            % \draw (-.5,-.2) node {(0,0)};
            % \draw (5.5,4.1) node {(1,1)};
        \end{tikzpicture}
        \caption{Flat torus taking $m=4,n=5$}
        \label{fig torus}
    \end{figure}
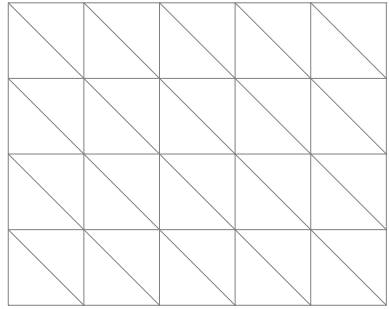

    \noindent \textbf{Subcase 1:} $\overline{\gamma}$ is simple. Then $\overline{\gamma}$ is homotopic to some $(p,q)$ curve $\delta$ in the torus $T$ where $p,q$ are integers with $\gcd(p,q)=1$. Note that the number of points of intersection of $\overline{\gamma}$ with the systolic geodesics is the same as the number of points of intersection of $\delta$ and the curves triangulating the torus. Now, consider everything in the universal cover $\R^2$ (see Figure \ref{fig univer1}). Let $\Tilde{\delta}$ be the lift of $\delta$ in the universal cover. Then the number of points of intersection of $\Tilde{\delta}$ with the lifts of the triangulating curves is the same as the number of points of intersection of the red arc in Figure \ref{fig univer1} with the lifts of the triangulating curves. Now, we count the number of points of intersection of the red arc. First, suppose that $pq>0$. The number of points of intersection of the horizontal part of the red arc with the triangulating curves is $2pn+1$ and for the vertical part is $2qm-1$ (1 is subtracted because one intersection point is already counted). Hence, the total number is $2(pn+qm)$, which is even.

    Now, suppose that $pq < 0$. Without loss of generality, assume that $p<0$. In this case, the number of points of intersection of the horizontal part is $2|p|n-1$, and for the vertical part is $2qn-1$ (see Figure \ref{fig univer2}). Hence the total is $2(|p|n+qn-1)$, which is even.

    \begin{figure}[htbp]
        \centering
        \begin{tikzpicture}
            \fill[gray!20!white]  (0,0) rectangle (5,4);
            \fill[gray!20!white]  (6,5) rectangle (11,7);
            \draw[step=1cm,gray,very thin] (-.9,-.9) grid (11.9,6.9);
            \foreach \x in {0,1,...,6} 
            {
            \draw[gray, very thin] (-1,\x+1) -- (\x+1,-1);
            }
            \foreach \x in {0,1,...,3} 
            {
            \draw[gray, very thin] (\x,7) -- (\x+8,-1);
            }
            \foreach \x in {0,1,...,6} 
            {
            \draw[gray, very thin] (\x+4,7) -- (12,\x-1);
            }
            \draw [red] (0,.3)--(5.4,.3)--(5.4,5.4)--(6,5.4);
            \draw [cyan] (0,0) rectangle (5,4);
            \draw [cyan] (6,7)--(6,5)--(11,5)--(11,7);

            \draw (0,.3) node {\tiny $\bullet$} (6,5.4) node {\tiny $\bullet$} (6,5) node {\tiny $\bullet$} (0,0) node {\tiny $\bullet$};
            \fill[white!50!white] (5.1+1,4.5) rectangle (5.9+1,4.9); 
            \draw (-.5,-.3) node {$(0,0)$} (5.5+1,4.75) node {$(p,q)$};
        \end{tikzpicture}
        \caption{Tessellation in the universal cover; the shaded rectangle represents a fundamental domain}
        \label{fig univer1}
    \end{figure}

    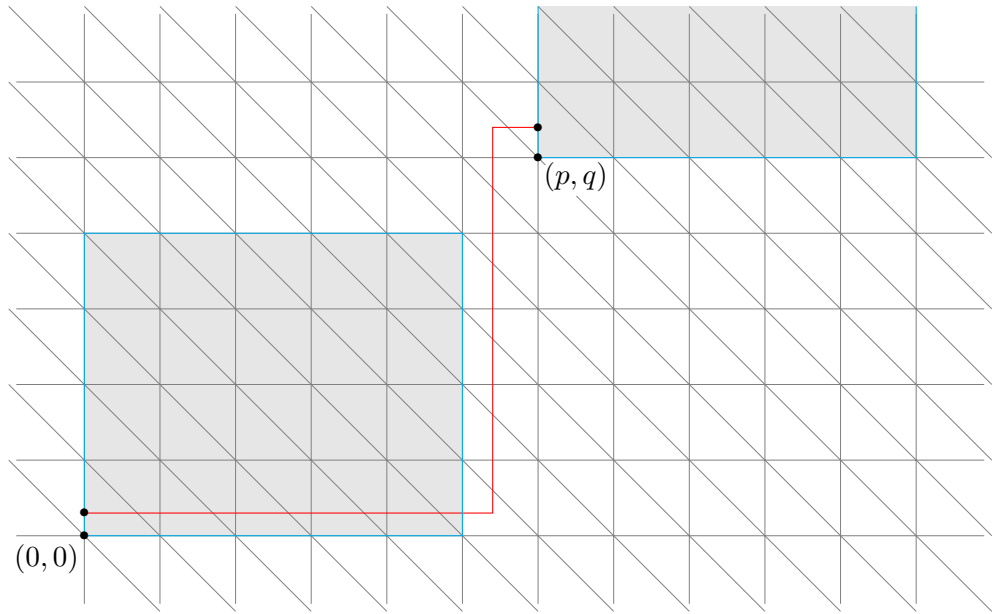
\begin{figure}[htbp]
        \centering
        \begin{tikzpicture}
            
            \fill[gray!20!white]  (6,0) rectangle (11,4);
            \fill[gray!20!white]  (0,5) rectangle (5,7);
            \draw[step=1cm,gray,very thin] (-.9,-.9) grid (11.9,6.9);
            \foreach \x in {0,1,...,6} 
            {
            \draw[gray, very thin] (-1,\x+1) -- (\x+1,-1);
            }
            \foreach \x in {0,1,...,3}
            {
            \draw[gray, very thin] (\x,7) -- (\x+8,-1);
            }
            \foreach \x in {0,1,...,6} 
            {
            \draw [gray, very thin] (\x+4,7) -- (12,\x-1);
            }
            \draw [red] (6,.75)--(0.4,.75)--(0.4,5.5)--(0,5.5);
            \draw [cyan] (6,0) rectangle (11,4);
            \draw [cyan] (0,7)--(0,5)--(5,5)--(5,7);

            \draw (6,.75) node {\tiny $\bullet$} (0,5.5) node {\tiny $\bullet$} (0,5) node {\tiny $\bullet$} (6,0) node {\tiny $\bullet$};
            \fill[white!50!white] (-1.1+1,4.5) rectangle (-1.9+1,4.9);
            \fill[white!50!white] (-1.1+1+6,4.5-5) rectangle (-1.9+1+6,4.9-5);
            \draw (-.5+6,-.3) node {$(0,0)$} (5.5+1-7,4.75) node {$(p,q)$};

            %\draw[thick] plot[smooth cycle, tension=1] coordinates {(0,0) (2,1) (3,-1) (1,-3) (-1,-1)};
             %\draw[thick, domain=0:4, samples=100] plot (\x, {sin(90*\x)});
        \end{tikzpicture}
        \caption{Tessellation in the universal cover; the shaded rectangle represents a fundamental domain}
        \label{fig univer2}
    \end{figure}

\noindent\textbf{Subcase 2:} $\overline{\gamma}$ is not simple. Take a lift of $\overline{\gamma}$ on the universal cover $\R^2$ (see Figure \ref{fig univer3}). The red loop (or loops in general) intersects each triangulating curve at an even number (0 or 2) of points. The remaining portion produces a simple closed curve on the torus and by Subcase 1, this part also intersects the triangulating curve at an even number of points. This completes the proof.

\begin{figure}[htbp]
        \centering
        \begin{tikzpicture}
            \fill[gray!20!white]  (0,0) rectangle (5,4);
            \fill[gray!20!white]  (6,5) rectangle (11,7);
            \draw[step=1cm,gray,very thin] (-.9,-.9) grid (11.9,6.9);
            \foreach \x in {0,1,...,6} 
            {
            \draw[gray, very thin] (-1,\x+1) -- (\x+1,-1);
            }
            \foreach \x in {0,1,...,3} 
            {
            \draw[gray, very thin] (\x,7) -- (\x+8,-1);
            }
            \foreach \x in {0,1,...,6} 
            {
            \draw[gray, very thin] (\x+4,7) -- (12,\x-1);
            }
           \draw (0,.3) node {\tiny $\bullet$} (6,5.4) node {\tiny $\bullet$};
           \draw [red] (2.2,1.5)..controls(4.5,5) and (-1,3)..(2.2,1.5);
           \draw [cyan] (0,.3)..controls(1,.2) and (2,1.2)..(2.2,1.5);
           \draw [cyan] (2.2,1.5)..controls(4,1) and (5.5,5)..(6,5.4);
           \fill[white!50!white] (5.1+1,4.5) rectangle (5.9+1,4.9);
           \draw (-.5,-.3) node {$(0,0)$} (5.5+1,4.75) node {$(p,q)$};
           % \draw[thick] plot[smooth cycle, tension=1] coordinates {(2.2,1.7) (2.4,2.8) (1.5,3.5) (.5,2.5) (1,1.5)};
             %\draw[thick, domain=0:4, samples=100] plot (\x, {sin(90*\x)});
        \end{tikzpicture}
        \caption{Tessellation in the universal cover; the shaded rectangle represents a fundamental domain}
        \label{fig univer3}
    \end{figure}
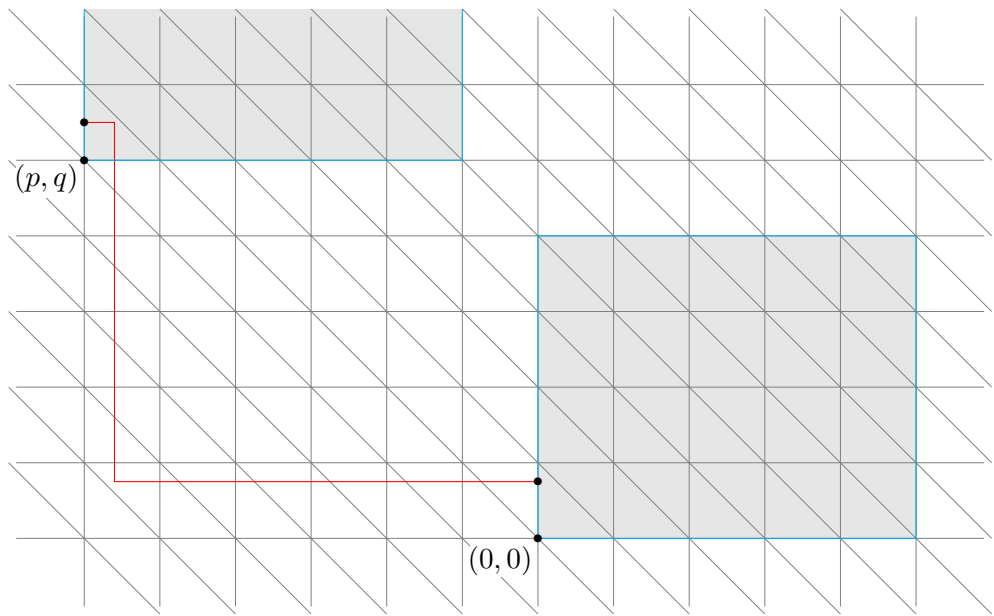
    
\end{proof}

\clearpage
\thispagestyle{empty}
\hfill
\clearpage

\bibliographystyle{alpha}
\bibliography{bibliography.bib}

@book{farb2011primer,
  title={A primer on mapping class groups (pms-49)},
  author={Farb, Benson and Margalit, Dan},
  year={2011},
  publisher={Princeton university press}
}

@article{anderson2011small,
  title={Small filling sets of curves on a surface},
  author={Anderson, James W and Parlier, Hugo and Pettet, Alexandra},
  journal={Topology and its Applications},
  volume={158},
  number={1},
  pages={84--92},
  year={2011},
  publisher={Elsevier}
}

@book{buser2010geometry,
  title={Geometry and spectra of compact Riemann surfaces},
  author={Buser, Peter},
  year={2010},
  publisher={Springer Science \& Business Media}
}

@article {MR3877282,
    AUTHOR = {Sanki, Bidyut},
     TITLE = {Systolic fillings of surfaces},
   JOURNAL = {Bull. Aust. Math. Soc.},
  FJOURNAL = {Bulletin of the Australian Mathematical Society},
    VOLUME = {98},
      YEAR = {2018},
    NUMBER = {3},
     PAGES = {502--511},
      ISSN = {0004-9727},
   MRCLASS = {57M15 (05C10)},
  MRNUMBER = {3877282},
MRREVIEWER = {James W. Anderson},
       DOI = {10.1017/S0004972718000862},
       URL = {https://doi.org/10.1017/S0004972718000862}
}

@article{thurstonspine,
  title={A spine for {T}eichmuller space},
  author={Thurston,W.},
  journal={preprint},
  year={1986},
}

@article{mathieu2023estimating,
  title={Estimating the dimension of {T}hurston spine},
  author={Mathieu, Olivier},
  journal={arXiv preprint arXiv:2310.15618},
  year={2023}
}

@incollection {CannonHypGeom,
    AUTHOR = {Cannon, James W. and Floyd, William J. and Kenyon, Richard and
              Parry, Walter R.},
     TITLE = {Hyperbolic geometry},
 BOOKTITLE = {Flavors of geometry},
    SERIES = {Math. Sci. Res. Inst. Publ.},
    VOLUME = {31},
     PAGES = {59--115},
 PUBLISHER = {Cambridge Univ. Press, Cambridge},
      YEAR = {1997},
      ISBN = {0-521-62048-1; 0-521-62962-4},
   MRCLASS = {57M50 (57N10)},
  MRNUMBER = {1491098},
MRREVIEWER = {Lee\ Mosher},
}

@book{euclid2002elements,
 author={Fitzpatrick, Richard},
  title     = {Euclid's Elements of Geometry},
  year      = {2002},
  publisher = {Green Lion Press},
  address   = {Santa Fe, NM},
  note      = {Originally written around 300 BCE}
}

@book {Beardon,
    AUTHOR = {Beardon, Alan F.},
     TITLE = {The geometry of discrete groups},
    SERIES = {Graduate Texts in Mathematics},
    VOLUME = {91},
      NOTE = {Corrected reprint of the 1983 original},
 PUBLISHER = {Springer-Verlag, New York},
      YEAR = {1995},
     PAGES = {xii+337},
      ISBN = {0-387-90788-2},
   MRCLASS = {22E40 (11F06 20H15 30F35 57N10)},
  MRNUMBER = {1393195},
}

@book {Katok,
    AUTHOR = {Katok, Svetlana},
     TITLE = {Fuchsian groups},
    SERIES = {Chicago Lectures in Mathematics},
 PUBLISHER = {University of Chicago Press, Chicago, IL},
      YEAR = {1992},
     PAGES = {x+175},
      ISBN = {0-226-42582-7; 0-226-42583-5},
   MRCLASS = {20H10 (30F35)},
  MRNUMBER = {1177168},
MRREVIEWER = {I.\ Kra},
}

@incollection {Keen,
    AUTHOR = {Keen, Linda},
     TITLE = {Collars on {R}iemann surfaces},
 BOOKTITLE = {Discontinuous groups and {R}iemann surfaces ({P}roc. {C}onf.,
              {U}niv. {M}aryland, {C}ollege {P}ark, {M}d., 1973)},
    SERIES = {Ann. of Math. Stud.},
    VOLUME = {No. 79},
     PAGES = {263--268},
 PUBLISHER = {Princeton Univ. Press, Princeton, NJ},
      YEAR = {1974},
   MRCLASS = {30A46},
  MRNUMBER = {379833},
MRREVIEWER = {C.\ Maclachlan},
}

@article {Basmajian,
    AUTHOR = {Basmajian, Ara},
     TITLE = {Generalizing the hyperbolic collar lemma},
   JOURNAL = {Bull. Amer. Math. Soc. (N.S.)},
  FJOURNAL = {American Mathematical Society. Bulletin. New Series},
    VOLUME = {27},
      YEAR = {1992},
    NUMBER = {1},
     PAGES = {154--158},
      ISSN = {0273-0979,1088-9485},
   MRCLASS = {30F40 (53C22 57M50)},
  MRNUMBER = {1145576},
       DOI = {10.1090/S0273-0979-1992-00298-7},
       URL = {https://doi.org/10.1090/S0273-0979-1992-00298-7},
}

@article {BuserCollar,
    AUTHOR = {Buser, Peter},
     TITLE = {The collar theorem and examples},
   JOURNAL = {Manuscripta Math.},
  FJOURNAL = {Manuscripta Mathematica},
    VOLUME = {25},
      YEAR = {1978},
    NUMBER = {4},
     PAGES = {349--357},
      ISSN = {0025-2611,1432-1785},
   MRCLASS = {53C22 (10D05 30F35 32G15 58G25)},
  MRNUMBER = {509590},
MRREVIEWER = {R.\ J.\ Sibner},
       DOI = {10.1007/BF01168048},
       URL = {https://doi.org/10.1007/BF01168048},
}

@article {Chavel,
    AUTHOR = {Chavel, Isaac and Feldman, Edgar A.},
     TITLE = {Cylinders on surfaces},
   JOURNAL = {Comment. Math. Helv.},
  FJOURNAL = {Commentarii Mathematici Helvetici},
    VOLUME = {53},
      YEAR = {1978},
    NUMBER = {3},
     PAGES = {439--447},
      ISSN = {0010-2571,1420-8946},
   MRCLASS = {53C20},
  MRNUMBER = {493868},
MRREVIEWER = {Sh\^ukichi\ Tanno},
       DOI = {10.1007/BF02566089},
       URL = {https://doi.org/10.1007/BF02566089},
}

@article {Schmutz1999,
    AUTHOR = {Schmutz Schaller, Paul},
     TITLE = {Teichm\"uller space and fundamental domains of {F}uchsian
              groups},
   JOURNAL = {Enseign. Math. (2)},
  FJOURNAL = {L'Enseignement Math\'ematique. Revue Internationale. 2e
              S\'erie},
    VOLUME = {45},
      YEAR = {1999},
    NUMBER = {1-2},
     PAGES = {169--187},
      ISSN = {0013-8584},
   MRCLASS = {30F60 (20H10 30F35)},
  MRNUMBER = {1703366},
MRREVIEWER = {Steffen\ Timmann},
}

@book {Anderson,
    AUTHOR = {Anderson, James W.},
     TITLE = {Hyperbolic geometry},
    SERIES = {Springer Undergraduate Mathematics Series},
   EDITION = {Second},
 PUBLISHER = {Springer-Verlag London, Ltd., London},
      YEAR = {2005},
     PAGES = {xii+276},
      ISBN = {1-85233-934-9},
   MRCLASS = {51-01 (30-01 30F35 30F45 51M10)},
  MRNUMBER = {2161463},
}

@article {Bo21,
    AUTHOR = {Boissy, Corentin and Geninska, Slavyana},
     TITLE = {Systoles in translation surfaces},
   JOURNAL = {Bull. Soc. Math. France},
  FJOURNAL = {Bulletin de la Soci\'et\'e{} Math\'ematique de France},
    VOLUME = {149},
      YEAR = {2021},
    NUMBER = {2},
     PAGES = {417--438},
      ISSN = {0037-9484,2102-622X},
   MRCLASS = {32G15 (30F30)},
  MRNUMBER = {4307431},
MRREVIEWER = {Fr\'ed\'eric\ Palesi},
       DOI = {10.24033/bsmf.2833},
       URL = {https://doi.org/10.24033/bsmf.2833},
}

@incollection {Mas22,
    AUTHOR = {Massart, Daniel},
     TITLE = {A short introduction to translation surfaces, {V}eech
              surfaces, and {T}eichm\"uller dynamics},
 BOOKTITLE = {Surveys in geometry {I}},
     PAGES = {343--388},
 PUBLISHER = {Springer, Cham},
      YEAR = {[2022] \copyright2022},
      ISBN = {978-3-030-86694-5; 978-3-030-86695-2},
   MRCLASS = {37D40 (30F60 32G15)},
  MRNUMBER = {4404400},
       DOI = {10.1007/978-3-030-86695-2\_9},
       URL = {https://doi.org/10.1007/978-3-030-86695-2_9},
}

@article {Pa19,
    AUTHOR = {Judge, Chris and Parlier, Hugo},
     TITLE = {The maximum number of systoles for genus two {R}iemann
              surfaces with abelian differentials},
   JOURNAL = {Comment. Math. Helv.},
  FJOURNAL = {Commentarii Mathematici Helvetici. A Journal of the Swiss
              Mathematical Society},
    VOLUME = {94},
      YEAR = {2019},
    NUMBER = {2},
     PAGES = {399--437},
      ISSN = {0010-2571,1420-8946},
   MRCLASS = {32G15 (30F10 53C22)},
  MRNUMBER = {3942788},
MRREVIEWER = {Yunhui\ Wu},
       DOI = {10.4171/CMH/463},
       URL = {https://doi.org/10.4171/CMH/463},
}

@article {san19,
    AUTHOR = {Sanki, Bidyut and Gadgil, Siddhartha},
     TITLE = {Graphs of systoles on hyperbolic surfaces},
   JOURNAL = {J. Topol. Anal.},
  FJOURNAL = {Journal of Topology and Analysis},
    VOLUME = {11},
      YEAR = {2019},
    NUMBER = {1},
     PAGES = {1--20},
      ISSN = {1793-5253,1793-7167},
   MRCLASS = {57M50 (05C10 57M15)},
  MRNUMBER = {3918059},
MRREVIEWER = {Elena\ Pavelescu},
       DOI = {10.1142/S1793525319500018},
       URL = {https://doi.org/10.1142/S1793525319500018},
}

@article {sa19,
    AUTHOR = {Sanki, Bidyut},
     TITLE = {Embedding of metric graphs on hyperbolic surfaces},
   JOURNAL = {Bull. Aust. Math. Soc.},
  FJOURNAL = {Bulletin of the Australian Mathematical Society},
    VOLUME = {99},
      YEAR = {2019},
    NUMBER = {3},
     PAGES = {508--520},
      ISSN = {0004-9727,1755-1633},
   MRCLASS = {57M15 (05C10)},
  MRNUMBER = {3952485},
MRREVIEWER = {Ottorino\ Ori},
       DOI = {10.1017/s0004972719000145},
       URL = {https://doi.org/10.1017/s0004972719000145},
}

@article {AR15,
    AUTHOR = {Wright, Alex},
     TITLE = {Translation surfaces and their orbit closures: an introduction
              for a broad audience},
   JOURNAL = {EMS Surv. Math. Sci.},
  FJOURNAL = {EMS Surveys in Mathematical Sciences},
    VOLUME = {2},
      YEAR = {2015},
    NUMBER = {1},
     PAGES = {63--108},
      ISSN = {2308-2151,2308-216X},
   MRCLASS = {32G15 (30D05 30F30 30F60)},
  MRNUMBER = {3354955},
MRREVIEWER = {Jayadev\ S.\ Athreya},
       DOI = {10.4171/EMSS/9},
       URL = {https://doi.org/10.4171/EMSS/9},
}

@article {MR4790974,
    AUTHOR = {Columbus, Tobias and Herrlich, Frank and Muetzel, Bjoern and
              Weitze-Schmith\"usen, Gabriela},
     TITLE = {Systolic geometry of translation surfaces},
   JOURNAL = {Exp. Math.},
  FJOURNAL = {Experimental Mathematics},
    VOLUME = {33},
      YEAR = {2024},
    NUMBER = {3},
     PAGES = {357--378},
      ISSN = {1058-6458,1944-950X},
   MRCLASS = {30F10 (32G15 53C22)},
  MRNUMBER = {4790974},
       DOI = {10.1080/10586458.2022.2108946},
       URL = {https://doi.org/10.1080/10586458.2022.2108946},
}

@article {MR4089838,
    AUTHOR = {Lee, Hung-Chih and Chen, Zhen-Chun},
     TITLE = {Decomposing the complete graph into {H}amiltonian paths
              (cycles) and 3-stars},
   JOURNAL = {Discuss. Math. Graph Theory},
  FJOURNAL = {Discussiones Mathematicae. Graph Theory},
    VOLUME = {40},
      YEAR = {2020},
    NUMBER = {3},
     PAGES = {823--839},
      ISSN = {1234-3099,2083-5892},
   MRCLASS = {05C70 (05C45)},
  MRNUMBER = {4089838},
MRREVIEWER = {Ioan\ Tomescu},
       DOI = {10.7151/dmgt.2153},
       URL = {https://doi.org/10.7151/dmgt.2153},
}

@book{beardon2012geometry,
  title={The geometry of discrete groups},
  author={Beardon, Alan F},
  volume={91},
  year={2012},
  publisher={Springer Science \& Business Media}
}

@article {Akrout,
    AUTHOR = {Akrout, H.},
     TITLE = {Singularit\'es topologiques des systoles g\'en\'eralis\'ees},
   JOURNAL = {Topology},
  FJOURNAL = {Topology. An International Journal of Mathematics},
    VOLUME = {42},
      YEAR = {2003},
    NUMBER = {2},
     PAGES = {291--308},
      ISSN = {0040-9383},
   MRCLASS = {32G15 (57R70)},
  MRNUMBER = {1941437},
MRREVIEWER = {Yoichi\ Imayoshi},
       DOI = {10.1016/S0040-9383(01)00024-6},
       URL = {https://doi.org/10.1016/S0040-9383(01)00024-6},
}

@article {anderson6146relative,
    AUTHOR = {Anderson, James W. and Parlier, Hugo and Pettet, Alexandra},
     TITLE = {Relative shapes of thick subsets of moduli space},
   JOURNAL = {Amer. J. Math.},
  FJOURNAL = {American Journal of Mathematics},
    VOLUME = {138},
      YEAR = {2016},
    NUMBER = {2},
     PAGES = {473--498},
      ISSN = {0002-9327,1080-6377},
   MRCLASS = {32G15 (14H15 30F60 57N05 58D29)},
  MRNUMBER = {3483473},
MRREVIEWER = {Milagros\ Izquierdo},
       DOI = {10.1353/ajm.2016.0010},
       URL = {https://doi.org/10.1353/ajm.2016.0010},
}

@article{bourque2022dimension,
    AUTHOR = {Fortier Bourque, Maxime},
     TITLE = {The dimension of {T}hurston's spine},
   JOURNAL = {Int. Math. Res. Not. IMRN},
  FJOURNAL = {International Mathematics Research Notices. IMRN},
      YEAR = {2024},
    NUMBER = {4},
     PAGES = {3487--3496},
      DOI = {10.1093/imrn/rnad211},
       URL = {https://doi.org/10.1093/imrn/rnad211},
}

@article {MinimalDelaunay,
    AUTHOR = {Ebbens, Matthijs and Parlier, Hugo and Vegter, Gert},
     TITLE = {Minimal {D}elaunay triangulations of hyperbolic surfaces},
   JOURNAL = {Discrete Comput. Geom.},
  FJOURNAL = {Discrete \& Computational Geometry. An International Journal
              of Mathematics and Computer Science},
    VOLUME = {69},
      YEAR = {2023},
    NUMBER = {2},
     PAGES = {568--592},
      ISSN = {0179-5376,1432-0444},
   MRCLASS = {57K20 (32G15 52C45 57M15)},
  MRNUMBER = {4544681},
       DOI = {10.1007/s00454-022-00373-0},
       URL = {https://doi.org/10.1007/s00454-022-00373-0},
}

@book {Buser,
    AUTHOR = {Buser, Peter},
     TITLE = {Geometry and spectra of compact {R}iemann surfaces},
    SERIES = {Modern Birkh\"auser Classics},
      NOTE = {Reprint of the 1992 edition},
 PUBLISHER = {Birkh\"auser Boston, Ltd., Boston, MA},
      YEAR = {2010},
     PAGES = {xvi+454},
      ISBN = {978-0-8176-4991-3},
   MRCLASS = {58J50 (30F10 32G15 58J53)},
  MRNUMBER = {2742784},
       DOI = {10.1007/978-0-8176-4992-0},
       URL = {https://doi.org/10.1007/978-0-8176-4992-0},
}

@article {Schmutz,
    AUTHOR = {Schmutz, P.},
     TITLE = {Riemann surfaces with shortest geodesic of maximal length},
   JOURNAL = {Geom. Funct. Anal.},
  FJOURNAL = {Geometric and Functional Analysis},
    VOLUME = {3},
      YEAR = {1993},
    NUMBER = {6},
     PAGES = {564--631},
      ISSN = {1016-443X,1420-8970},
   MRCLASS = {30F99 (53A35 57N05)},
  MRNUMBER = {1250756},
MRREVIEWER = {Christophe\ Bavard},
       DOI = {10.1007/BF01896258},
       URL = {https://doi.org/10.1007/BF01896258},
}

@article {Jenni,
    AUTHOR = {Jenni, Felix},
     TITLE = {\"Uber den ersten {E}igenwert des {L}aplace-{O}perators auf
              ausgew\"ahlten {B}eispielen kompakter {R}iemannscher
              {F}l\"achen},
   JOURNAL = {Comment. Math. Helv.},
  FJOURNAL = {Commentarii Mathematici Helvetici},
    VOLUME = {59},
      YEAR = {1984},
    NUMBER = {2},
     PAGES = {193--203},
      ISSN = {0010-2571,1420-8946},
   MRCLASS = {58G25 (30F20)},
  MRNUMBER = {749104},
MRREVIEWER = {Scott\ Wolpert},
       DOI = {10.1007/BF02566345},
       URL = {https://doi.org/10.1007/BF02566345},
}

@article {SchmutzGlobalMaximal,
    AUTHOR = {Schmutz, Paul},
     TITLE = {Systoles on {R}iemann surfaces},
   JOURNAL = {Manuscripta Math.},
  FJOURNAL = {Manuscripta Mathematica},
    VOLUME = {85},
      YEAR = {1994},
    NUMBER = {3-4},
     PAGES = {429--447},
      ISSN = {0025-2611,1432-1785},
   MRCLASS = {57M50 (30F60 53C23)},
  MRNUMBER = {1305753},
MRREVIEWER = {Athanase\ Papadopoulos},
       DOI = {10.1007/BF02568209},
       URL = {https://doi.org/10.1007/BF02568209},
}

@incollection {Masur06,
    AUTHOR = {Masur, Howard},
     TITLE = {Ergodic theory of translation surfaces},
 BOOKTITLE = {Handbook of dynamical systems. {V}ol. 1{B}},
     PAGES = {527--547},
 PUBLISHER = {Elsevier B. V., Amsterdam},
      YEAR = {2006},
      ISBN = {0-444-52055-4},
   MRCLASS = {37A25 (28D15 37D40 37E35)},
  MRNUMBER = {2186247},
MRREVIEWER = {Richard\ Kenyon},
       DOI = {10.1016/S1874-575X(06)80032-9},
       URL = {https://doi.org/10.1016/S1874-575X(06)80032-9},
}

@book{Riemann1857,
  author    = {Bernhard Riemann},
  title     = {Theorie der Abel'schen Functionen},
  year      = {1857},
  publisher = {Crelle's Journal}
}

@book{Birkhoff1927,
  author    = {George D. Birkhoff},
  title     = {Dynamical Systems},
  year      = {1927},
  publisher = {AMS Colloquium Publications}
}

@article{Masur82,
  author    = {Howard Masur},
  title     = {Interval exchange transformations and measured foliations},
  journal   = {Annals of Mathematics},
  volume    = {115},
  number    = {1},
  pages     = {169--200},
  year      = {1982}
}

@article{Veech89,
  author    = {William A. Veech},
  title     = {Teichmüller curves in moduli space, Eisenstein series and an application to triangular billiards},
  journal   = {Inventiones Mathematicae},
  volume    = {97},
  pages     = {553--583},
  year      = {1989}
}

@article{MasurTabachnikov02,
  author    = {Howard Masur and Serge Tabachnikov},
  title     = {Rational billiards and flat structures},
  journal   = {Handbook of Dynamical Systems, Vol. 1A},
  pages     = {1015--1089},
  year      = {2002}
}

@incollection{Zorich06,
  author    = {Anton Zorich},
  title     = {Flat Surfaces},
  booktitle = {Frontiers in Number Theory, Physics, and Geometry. Vol. I},
  publisher = {Springer},
  year      = {2006},
  pages     = {439--586}
}

@article{McMullen03,
  author    = {Curtis T. McMullen},
  title     = {Billiards and {T}eichmüller curves on Hilbert modular surfaces},
  journal   = {Journal of the American Mathematical Society},
  volume    = {16},
  number    = {4},
  pages     = {857--885},
  year      = {2003}
}

@book{hatcher2002,
  author    = {Allen Hatcher},
  title     = {Algebraic Topology},
  publisher = {Cambridge University Press},
  year      = {2002},
  url       = {https://pi.math.cornell.edu/~hatcher/AT/ATpage.html}
}

@book{bridson1999metric,
  title={Metric Spaces of Non-Positive Curvature},
  author={Bridson, Martin R. and Haefliger, Andr{\'e}},
  series={Grundlehren der mathematischen Wissenschaften},
  volume={319},
  publisher={Springer-Verlag},
  year={1999},
  isbn={978-3-540-64324-1},
  doi={10.1007/978-3-662-12494-9},
  url={https://link.springer.com/book/10.1007/978-3-662-12494-9}
}

@article{schmutz1993riemann,
  author    = {Schmutz, Peter},
  title     = {Riemann surfaces with shortest geodesic of maximal length},
  journal   = {Geometric and Functional Analysis},
  volume    = {3},
  number    = {6},
  pages     = {564--631},
  year      = {1993},
  publisher = {Springer},
  doi       = {10.1007/BF01896444}
}

@article{saha_sanki_2023_complexity_in_bolza_surface,
  author       = {Bhola Nath Saha and Bidyut Sanki},
  title        = {Complexity in Bolza Surface},
  eprint       = {2311.18483},
  archivePrefix= {arXiv},
  primaryClass = {math.DG},
  year         = {2023},
  month        = {Nov},
  note         = {preprint},
  url          = {https://arxiv.org/abs/2311.18483}
}

@article {MR4985436,
    AUTHOR = {Dey, Achintya and Sanki, Bidyut},
     TITLE = {Systolic embedding of graphs on translation surfaces},
   JOURNAL = {Geom. Dedicata},
  FJOURNAL = {Geometriae Dedicata},
    VOLUME = {219},
      YEAR = {2025},
    NUMBER = {6},
     PAGES = {Paper No. 91, 20},
      ISSN = {0046-5755,1572-9168},
   MRCLASS = {30F10 (32G15 53C22)},
  MRNUMBER = {4985436},
MRREVIEWER = {Yumin\ Zhong},
       DOI = {10.1007/s10711-025-01056-9},
       URL = {https://doi.org/10.1007/s10711-025-01056-9},
}

@article{DeySahaSankiToAppear,
  author  = {Achintya Dey and Bhola Nath Saha and Bidyut Sanki},
  title   = {Bolza-like Surfaces in the {Thurston} Set},
  journal = {Proceedings of the American Mathematical Society},
  year    = {2026},
  note    = {doi: 10.1090/proc/17851}
}

\end{document}